\documentclass[11pt]{article}
\usepackage[numbers,square]{natbib}
\usepackage{amsmath}
\usepackage{amsthm}
\usepackage{amsfonts}
\usepackage{amssymb}
\usepackage{siunitx}
\usepackage{commath}
\usepackage{graphicx}
\usepackage{xspace}
\usepackage{color}
\usepackage[lofdepth,lotdepth]{subfig}
\usepackage{stmaryrd}
\usepackage{url}
\usepackage{amsthm}
\usepackage{footnote}
\makesavenoteenv{tabular}
\makesavenoteenv{table}
\usepackage[hidelinks]{hyperref}
\hypersetup{colorlinks = true, urlcolor = blue,
  linkcolor = blue, citecolor = blue}
\usepackage{cleveref}

\usepackage[margin=0.7in]{geometry}
\makeatletter
\newcommand{\tnorm}{\@ifstar\@tnorms\@tnorm}
\newcommand{\@tnorms}[1]{%
  \left|\mkern-1.5mu\left|\mkern-1.5mu\left|
   #1
  \right|\mkern-1.5mu\right|\mkern-1.5mu\right|
}
\newcommand{\@tnorm}[2][]{%
  \mathopen{#1|\mkern-1.5mu#1|\mkern-1.5mu#1|}
  #2
  \mathclose{#1|\mkern-1.5mu#1|\mkern-1.5mu#1|}
}
\makeatother

\newcommand{\jump}[1]{\llbracket #1 \rrbracket}
\newcommand{\av}[1]{\{\!\!\{#1\}\!\!\}}
\newtheorem{theorem}{Theorem}
\newtheorem{lemma}{Lemma}
\newtheorem{corollary}{Corollary}
\newtheorem{remark}{Remark}
\title{A strongly conservative hybridizable discontinuous Galerkin
  method for the coupled time-dependent Navier--Stokes and Darcy
  problem}
\author{A. Cesmelioglu\thanks{Department of Mathematics and
    Statistics, Oakland University, MI, USA
    (\url{cesmelio@oakland.edu}),
    \url{https://orcid.org/0000-0001-8057-6349}}
  \and
  Jeonghun J. Lee\thanks{Department of Mathematics, Baylor University,
    TX, USA (\url{jeonghun_lee@baylor.edu}),
    \url{https://orcid.org/0000-0001-5201-8526}}
  \and
  S. Rhebergen\thanks{Department of Applied Mathematics, University of
    Waterloo, ON, Canada (\url{srheberg@uwaterloo.ca}),
    \url{http://orcid.org/0000-0001-6036-0356}}}
\begin{document}
\maketitle
\begin{abstract}
  We present a strongly conservative and pressure-robust hybridizable
  discontinuous Galerkin method for the coupled time-dependent
  Navier--Stokes and Darcy problem. We show existence and uniqueness
  of a solution and present an optimal a priori error analysis for the
  fully discrete problem when using Backward Euler time stepping. The
  theoretical results are verified by numerical examples.
\end{abstract}
\section{Introduction}
\label{sec:introduction}

In this paper we present an analysis of a hybridizable discontinuous
Galerkin (HDG) method for the coupled Navier--Stokes and Darcy
equations that model surface/subsurface flow. While various conforming
and nonconforming finite element methods have been studied for the
stationary Navier--Stokes and Darcy problem, see for example
\cite{Badea:2010,Chidyagwai:2009,Chidyagwai:2010,Discacciati:2017,Discacciati:2009,Girault:2009,Girault:2013},
the literature on numerical methods for the time-dependent problem is
limited. The first numerical methods for the time-dependent problem
were studied in \cite{Cesmelioglu:2008,Cesmelioglu:2009}. To simplify
the analysis, however, these papers included inertia effects in the
balance of forces at the interface. Existence and uniqueness of a weak
solution to the physically more relevant model, without inertia
effects on the interface, was proven in \cite{Cesmelioglu:2013b},
while convergence of a discontinuous Galerkin method for this model
was proven in \cite{Chaabane:2017}. Conforming methods for the
transient problem have been studied in \cite{Jia:2019,Xue:2020}.

The aforementioned papers for the time-dependent Navier--Stokes and
Darcy problem have in common that they consider the primal form of the
Darcy problem. In contrast, we consider the mixed form of the Darcy
problem as this facilitates the formulation of a \emph{strongly
  conservative} discretization, i.e., a discretization that is mass
conserving in the sense of $H(\text{div};\Omega)$ where the
velocity is globally $H(\text{div};\Omega)$-conforming and, in the
absence of sources and sinks, pointwise divergence-free on the
elements \cite{Kanschat:2010}. In particular, we consider an HDG
method \cite{Cockburn:2009a} that is based on the HDG method for the
Navier--Stokes equations \cite{Rhebergen:2018a} and a hybridized
formulation of the mixed form of the Darcy problem \cite{Arnold:1985},
although nonconforming formulations based on other strongly
conservative discretizations, for example,
\cite{Cockburn:2004b,Fu:2018,Lehrenfeld:2016,Wang:2007}, are possible.

Previously, we proved pressure-robustness of strongly conservative HDG
methods for the Stokes/Darcy \cite{Cesmelioglu:2020} and stationary
Navier--Stokes/Darcy \cite{Cesmelioglu:2023} problems, leading to a
priori error estimates for the velocity that do not depend on the best
approximation of the pressure scaled by the inverse of the viscosity
(see \cite{John:2017,Linke:2014} for a review of other pressure-robust
discretizations). Using Backward Euler time stepping we now show
existence and uniqueness of a solution and derive an a priori error
estimate to the fully-discrete time-dependent problem. Compared to
previous work on the time-dependent Navier--Stokes/Darcy problem
\cite{Cesmelioglu:2008,Cesmelioglu:2009,Cesmelioglu:2013b,Chaabane:2017,Jia:2019,Xue:2020},
the novel contributions of this work is therefore the introduction and
analysis of a strongly conservative HDG discretization and an a priori
error estimate for the velocity that is independent of pressure.

The remainder of this paper is organized as follows. We present the
time-dependent Navier--Stokes/Darcy problem in \cref{sec:nsd} and its
HDG discretization in \cref{sec:hdg}. Consistency and well-posedness
of the discrete problem are shown in \cref{sec:consiswellpose} while a
priori error estimates are proven in \cref{sec:apriori}. We end this
paper with numerical examples in \cref{s:numexamples} and conclusions
in \cref{s:conclusions}.

\section{The Navier--Stokes and Darcy problem}
\label{sec:nsd}
We consider the time-dependent incompressible Navier--Stokes equations
coupled to the Darcy equations on a polyhedral domain $\Omega$ in
$\mathbb{R}^{\dim}$, $\dim=2,3$, and on the time interval
$J=(0,T)$. The domain is partitioned into two non-overlapping
subdomains $\Omega^s$ and $\Omega^d$ such that
$\Omega = \Omega^s \cup \Omega^d$,
$\Omega^s \cap \Omega^d = \emptyset$, and
$\Gamma^I := \partial\Omega^s \cap \partial\Omega^d$. The boundary of
the domain $\partial \Omega$ and the interface $\Gamma^I$ are assumed
to be Lipschitz polyhedral. We define $\Gamma^s$ and $\Gamma^d$ to be
the exterior boundaries of $\Omega^s$ and $\Omega^d$, respectively. We
partition $\Gamma^d := \Gamma_N^d \cup \Gamma_D^d$, with
$\Gamma_N^d \cap \Gamma_D^d = \emptyset$ and $|\Gamma_N^d| > 0$ and
$|\Gamma_D^d| > 0$, and denote the outward unit normal on $\Gamma^j$
to $\Omega^j$ ($j=s,d$) by $n$.

The Navier--Stokes equations are given by
\begin{subequations}
  \label{eq:ns}
  \begin{align}
    \label{eq:ns-a}
    \partial_t u^s + \nabla \cdot (u^s \otimes u^s) + \nabla p^s - \nabla \cdot (2\mu \varepsilon(u^s))
    &= f^s && \text{ in } \Omega^s \times J,
    \\
    \label{eq:ns-b}
    \nabla \cdot u^s
    &= 0 && \text{ in } \Omega^s \times J,
  \end{align}
\end{subequations}
where $u^s : \Omega^s \times J \to \mathbb{R}^{\dim}$ is the velocity
in $\Omega^s$, $p^s : \Omega^s \times J \to \mathbb{R}^{\dim}$ is the
pressure in $\Omega^s$,
$\varepsilon(w) = \tfrac{1}{2}(\nabla w + (\nabla w)^T)$, $\mu >0$ is
the constant fluid viscosity, and
$f^s : \Omega^s \times J \to \mathbb{R}^{\dim}$ is a body force. In
$\Omega^d$ the Darcy equations are given by:
\begin{subequations}
  \label{eq:darcy}
  \begin{align}
    \label{eq:darcy-a}
    \mu \kappa^{-1}u^d + \nabla p^d
    &= 0 && \text{ in } \Omega^d \times J,
    \\
    \label{eq:darcy-b}
    -\nabla \cdot u^d
    &= f^d && \text{ in } \Omega^d \times J,
  \end{align}
\end{subequations}
where $u^d : \Omega^d \times J \to \mathbb{R}^{\dim}$ is the fluid
velocity in $\Omega^d$, $p^d : \Omega^d \times J \to \mathbb{R}$ is
the piezometric head in $\Omega^d$, and $\kappa > 0$ is the permeability
constant. The Navier--Stokes equations are coupled to the Darcy
equations by the following interface conditions
\begin{subequations}
  \label{eq:interface}
  \begin{align}
    \label{eq:interface-a}
    u^s \cdot n
    &= u^d \cdot n
    && \text{ on } \Gamma^I \times J,
    \\
    \label{eq:interface-b}
    -2\mu(\varepsilon(u^s)n)^t
    &= \alpha\mu\kappa^{-1/2} (u^s)^t
    && \text{ on } \Gamma^I \times J,
    \\
    \label{eq:interface-c}
    (p^sn - 2\mu\varepsilon(u^s)n)\cdot n
    &= p^d
    && \text{ on } \Gamma^I \times J,   
  \end{align}
\end{subequations}
where $n$ is the unit normal vector on $\Gamma^I$ pointing from
$\Omega^s$ to $\Omega^d$, $(v)^t := v - (v\cdot n)n$ is the tangential
component of a vector $v$, and $\alpha > 0$ is an experimentally
determined dimensionless constant. Note that \cref{eq:interface-a}
ensures continuity of the normal component of the velocity across the
interface, \cref{eq:interface-b} is the Beavers--Joseph--Saffman law
\cite{Beavers:1967,Saffman:1971}, and \cref{eq:interface-c} is a
balance of forces. We assume the following initial and boundary
conditions:
\begin{subequations}
  \label{eq:ic-bc}
  \begin{align}
    \label{eq:ic-bc-a}
    u^s(x,0)
    &= u_0(x) && \text{ in } \Omega^s,
    \\
    \label{eq:ic-bc-b}
    u^s
    &=0 && \text{ on } \Gamma^s \times J,
    \\
    \label{eq:ic-bc-c}
    u^d \cdot n
    &= 0 && \text{ on } \Gamma^d_N \times J,
    \\
    \label{eq:ic-bc-d}
    p^d
    &= 0 && \text{ on } \Gamma^d_D \times J,
  \end{align}
\end{subequations}
where $u_0 : \Omega^s \to \mathbb{R}^{\dim}$ is a solenoidal initial
velocity field. We close this section by introducing
$u :\Omega \times J \to \mathbb{R}^{\dim}$ and
$p : \Omega \times J \to \mathbb{R}$ to be the functions such that
$u|_{\Omega^j}=u^j$ and $p|_{\Omega^j}=p^j$ for $j=s,d$.

\section{The HDG method}
\label{sec:hdg}

\subsection{Notation}
\label{ss:dnotation}

Let $j=s,d$. We denote by $\mathcal{T}_h^j = \cbr[0]{K}$ a conforming
triangulation of $\Omega^j$ of shape-regular simplices $K$. We assume
that $\mathcal{T}_h = \mathcal{T}_h^s \cup \mathcal{T}_h^d$ is a
matching simplicial mesh, i.e., $\mathcal{T}_h^s$ and
$\mathcal{T}_h^d$ match at the interface. We denote by $h_K$ the
diameter of $K \in \mathcal{T}_h$ and define the meshsize as
$h := \max_{K \in \mathcal{T}_h} h_K$. A face $F$ is an interior face
if for two elements $K^+$ and $K^-$ in $\mathcal{T}_h$,
$F = \partial K^+ \cap \partial K^-$, and a boundary face if
$F \in \partial K$ lies on the boundary $\partial \Omega$. The set of
all facets in $\overline{\Omega}$ and $\overline{\Omega}^j$ are
denoted by, respectively, $\mathcal{F}_h$ and $\mathcal{F}_h^j$, while
the set of all facets on the interface $\Gamma^I$ is denoted by
$\mathcal{F}_h^I$. The set of all facets on $\Gamma^j$ are denoted by
$\mathcal{F}_h^{B,j}$ while the set of all facets interior to
$\Omega^j$ are denoted by $\mathcal{F}_h^{int,j}$. The sets of facets
on $\Gamma_N^d$ and $\Gamma_D^d$ are denoted by, respectively,
$\mathcal{F}_h^{N,d}$ and $\mathcal{F}_h^{D,d}$. By $\Gamma_0$ and
$\Gamma_0^j$ we denote the union of facets in $\overline{\Omega}$ and
$\overline{\Omega}^j$. The outward unit normal vector on $\partial K$
for any element $K \in \mathcal{T}_h^j$ is denoted by $n^j$. On the
interface $\Gamma^I$, $n=n^s=-n^d$. We will drop the superscript $j$
if the definition of the outward unit normal vector is clear.

We partition the time interval $J$ into $N$ equal intervals of length
$\Delta t = T/N$. We define $t^n := n\Delta t$ for $n=0,\hdots, N$ and
note that $t^0=0$ and $t^N = T$. A function $f$ evaluated at $t=t^n$
will be denoted by $f^n:=f(t^n)$. Furthermore, we define
$\delta f^{n+1} = f^{n+1} - f^n$ and
$d_tf^{n+1} = \delta f^{n+1}/\Delta t = (f^{n+1} - f^n)/\Delta t$.

Denoting by $P_m(D)$ the space of polynomials of total degree $m$ on a
domain $D$, we define the following finite element spaces for the
velocity approximation:
\begin{equation*}
  \begin{split}
    X_h &:= \cbr[1]{v_h \in \sbr[0]{L^2(\Omega)}^{\dim} : \ v_h \in
      \sbr[0]{P_k(K)}^{\dim}, \ \forall\ K \in \mathcal{T}},
    \\
    X_h^j &:= \cbr[1]{v_h \in \sbr[0]{L^2(\Omega^j)}^{\dim} : \ v_h \in
      \sbr[0]{P_k(K)}^{\dim}, \ \forall\ K \in \mathcal{T}^j}, \quad j=s,d,
    \\
    \bar{X}_h &:= \cbr[1]{\bar{v}_h \in \sbr[0]{L^2(\Gamma_0^s)}^{\dim}:\
      \bar{v}_h \in \sbr[0]{P_{k}(F)}^{\dim}\ \forall\ F \in \mathcal{F}^s,\
      \bar{v}_h = 0 \text{ on } \Gamma^s}.
  \end{split}
\end{equation*}
For notational purposes, we write
$\boldsymbol{v}_h = (v_h, \bar{v}_h) \in \boldsymbol{X}_h := X_h
\times \bar{X}_h$ and
$\boldsymbol{v}_h^s = (v_h^s, \bar{v}_h) \in \boldsymbol{X}_h^s :=
X_h^s \times \bar{X}_h$. Furthermore, for the pressure approximation
we define the finite element spaces
\begin{equation*}
  \begin{split}
    Q_h &:= \cbr[1]{q_h \in L^2(\Omega) : \ q_h \in P_{k-1}(K) ,\
      \forall \ K \in \mathcal{T}},
    \\
    Q_h^j &:= \cbr[1]{q_h \in L^2(\Omega^j) : \ q_h \in P_{k-1}(K) ,\
      \forall \ K \in \mathcal{T}^j}, \quad j=s,d,
    \\
    \bar{Q}_h^s &:= \cbr[1]{\bar{q}_h^s \in L^2(\Gamma_0^s) : \ \bar{q}_h^s
      \in P_{k}(F) \ \forall\ F \in \mathcal{F}^s},
    \\
    \bar{Q}_h^d &:= \cbr[1]{\bar{q}_h^d \in L^2(\Gamma_0^d) : \ \bar{q}_h^d
      \in P_{k}(F) \ \forall\ F \in \mathcal{F}^d,\ \bar{q}_h^d = 0 \text{ on } \Gamma_D^d}.
  \end{split}
\end{equation*}
We write
$\boldsymbol{q}_h = (q_h, \bar{q}_h^s, \bar{q}_h^d) \in
\boldsymbol{Q}_h := Q_h \times \bar{Q}_h^s \times \bar{Q}_h^d$ and
$\boldsymbol{q}_h^j = (q_h, \bar{q}_h^j) \in \boldsymbol{Q}_h^j :=
Q_h^j \times \bar{Q}_h^j$.

For scalar functions $p$ and $q$, we define
\begin{align*}
  (p, q)_K &:= \int_K p q \dif x, \quad \forall K \in \mathcal{T}_h,
  &
  \langle p, q \rangle_{\partial K} &:= \int_{\partial K} p q \dif s, \quad \forall K \in \mathcal{T}_h,
  \\
  \langle p, q \rangle_F &:= \int_F p q \dif s, \quad F \subset \partial K,\ \forall K \in \mathcal{T}_h,
  &
  (p, q)_{\Omega^j} &:= \sum_{K \in \mathcal{T}_h^j} (p, q)_K, \quad j=s,d,
  \\
  \langle p, q \rangle_{\partial\mathcal{T}_h^j} &:= \sum_{K \in \mathcal{T}_h^j} \langle p, q \rangle_{\partial K}, \quad j=s,d,
  &  
  (p, q)_{\Omega} &:= \sum_{K \in \mathcal{T}_h} \int_K p q \dif x, 
  \\  
  \langle p, q \rangle_{\partial\mathcal{T}_h} &:= \sum_{K \in \mathcal{T}_h} \langle p, q \rangle_{\partial K}, 
  &
  \langle p, q \rangle_{\Gamma^I} &:= \sum_{F \in \mathcal{F}_h^I} \langle p, q \rangle_F.
\end{align*}
Similar notation is used for vector- and matrix-valued functions.

\subsection{The semi-discrete problem}
\label{ss:semidiscrete}
An HDG method for the stationary Navier--Stokes and Darcy problem was
proposed in \cite{Cesmelioglu:2023}. Its extension to the
time-dependent problem is given by: Let
$u_h^{s,0} \in X_h^s \cap H({\rm div};\Omega^s)$ be the initial
condition for the velocity in $\Omega^s$ such that
$\nabla \cdot u_h^{s,0} = 0$ pointwise on each
$K \in \mathcal{T}_h^s$. For $t \in J$, find
$(\boldsymbol{u}_h(t), \boldsymbol{p}_h(t)) \in \boldsymbol{X}_h
\times \boldsymbol{Q}_h$ such that for all
$(\boldsymbol{v}_h, \boldsymbol{q}_h) \in \boldsymbol{X}_h \times
\boldsymbol{Q}_h$
\begin{equation}
  \label{eq:hdg-semi}
  (\partial_tu_h, v_h)_{\Omega^s}
  + a_h(u_h; \boldsymbol{u}_h, \boldsymbol{v}_h)
  + b_h(\boldsymbol{v}_h, \boldsymbol{p}_h)
  + b_h(\boldsymbol{u}_h, \boldsymbol{q}_h)
  = (f^s, v_h)_{\Omega^s} + (f^d, q_h)_{\Omega^d}.
\end{equation}
The different forms are defined as:
\begin{subequations}
  \begin{align}
    \label{eq:ah_s}
    a_h^s(\boldsymbol{u}, \boldsymbol{v})
    :=&\,
        (2\mu \varepsilon(u), \varepsilon(v))_{\Omega^s}
        + \langle 2\beta\mu h_K^{-1}(u-\bar{u}), v-\bar{v} \rangle_{\partial\mathcal{T}_h^s}
    \\
    \nonumber
      & - \langle 2\mu\varepsilon(u)n, v-\bar{v} \rangle_{\partial\mathcal{T}_h^s}
        - \langle 2\mu\varepsilon(v)n, u-\bar{u} \rangle_{\partial\mathcal{T}_h^s},
    \\
    \label{eq:ad}
    a^d(u,v)
    :=&\, (\mu\kappa^{-1}u, v)_{\Omega^d}
    \\
    \label{eq:aI}
    a^I(\bar{u}, \bar{v})
    := &\, \langle \alpha \mu \kappa^{-1/2} \bar{u}^t, \bar{v}^t \rangle_{\Gamma^I},
    \\
    \label{eq:def_ah_L}
    a_h^L(\boldsymbol{u}, \boldsymbol{v})
    :=& \,a_h^s(\boldsymbol{u}, \boldsymbol{v})
        + a^d(u, v) + a^I(\bar{u}, \bar{v}),
    \\
    \label{eq:oh}
    t_h(w; \boldsymbol{u}, \boldsymbol{v})
    :=&
        - (u \otimes w, \nabla v)_{\Omega^s}
        + \langle \tfrac{1}{2}w \cdot n \, (u+\bar{u}), v-\bar{v} \rangle_{\partial\mathcal{T}_h^s}
    \\ \nonumber
      & + \langle \tfrac{1}{2}\envert{w\cdot n}(u-\bar{u}), v-\bar{v} \rangle_{\partial\mathcal{T}_h^s}
        + \langle (w\cdot n)\bar{u}, \bar{v} \rangle_{\Gamma^I},        
    \\
    \label{eq:def_ah}
    a_h(w; \boldsymbol{u}, \boldsymbol{v})
    :=& \,t_h(w; \boldsymbol{u}, \boldsymbol{v})
        + a_h^L(\boldsymbol{u}, \boldsymbol{v})
  \end{align}
\end{subequations}
where $\beta > 0$ is a penalty parameter and where $a_h^L$ is the
linear part of $a_h$. For the velocity-pressure coupling we have, for $j=s,d$, the
forms:
\begin{align*}
  b_h^{j}(v, \boldsymbol{q}^{j})
  :=& -(q, \nabla \cdot v)_{\Omega^j} + \langle \bar{q}^j, v\cdot n^j \rangle_{\partial\mathcal{T}_h^j},
  \\
  b_h^{I,j}(\bar{v}, \bar{q}^j)
  :=& -\langle \bar{q}^j, \bar{v} \cdot n^j \rangle_{\Gamma^I},
  \\
  b_h(\boldsymbol{v}, \boldsymbol{q})
  :=& \,b_h^s(v, \boldsymbol{q}^s) + b_h^{I,s}(\bar{v}, \bar{q}^s) + b_h^d(v, \boldsymbol{q}^d) + b_h^{I,d}(\bar{v}, \bar{q}^d).  
\end{align*}

\subsection{The fully-discrete problem}
\label{ss:fullydiscrete}
Using backward Euler time-stepping, and lagging the convective
velocity in the nonlinear term, we obtain the following linear
implicit discretization: Let
$u_h^{s,0} \in X_h^s \cap H({\rm div};\Omega^s)$ be the initial
condition for the velocity in $\Omega^s$ such that
$\nabla \cdot u_h^{s,0} = 0$ pointwise on each
$K \in \mathcal{T}_h^s$. Find
$(\boldsymbol{u}_h^{n+1}, \boldsymbol{p}_h^{n+1}) \in \boldsymbol{X}_h
\times \boldsymbol{Q}_h$ with $n \ge 0$ such that for all
$(\boldsymbol{v}_h, \boldsymbol{q}_h) \in \boldsymbol{X}_h \times
\boldsymbol{Q}_h$:
\begin{equation}
  \label{eq:hdg-fully}
  \del[0]{ d_t u_h^{n+1}, v_h}_{\Omega^s}
  + a_h(u_h^n; \boldsymbol{u}_h^{n+1}, \boldsymbol{v}_h)
  + b_h(\boldsymbol{v}_h, \boldsymbol{p}_h^{n+1})
  + b_h(\boldsymbol{u}_h^{n+1}, \boldsymbol{q}_h)
  = (f^{s,n+1}, v_h)_{\Omega^s} + (f^{d,n+1}, q_h)_{\Omega^d}.
\end{equation}

\begin{remark}
  \label{rem:properties-hdg-fully}
  As observed previously in \cite{Cesmelioglu:2023} for the stationary
  Navier--Stokes and Darcy problem, the velocity solution to
  \cref{eq:hdg-fully} satisfies the following properties: (i) it is
  exactly divergence-free on elements in $\Omega^s$, i.e.,
  $\nabla \cdot u_h^n = 0$ pointwise on each $K \in \mathcal{T}_h^s$;
  (ii) it satisfies $-\nabla \cdot u_h^n = \Pi_Q^d f^{d,n}$ pointwise
  on each $K \in \mathcal{T}_h^d$ (where $\Pi_Q^d$ is the
  $L^2$-projection operator into $Q_h^d$); (iii) the velocity solution
  is globally divergence-conforming, i.e.,
  $u_h^n \in H(\text{div};\Omega)$; and (iv)
  $u_h^n \cdot n = \bar{u}_h^n \cdot n$ pointwise on each
  $F \in \mathcal{F}^I$. Furthermore, $u_h^{d,n} \cdot n = 0$ on
  $\Gamma_N^d$ and $u_h^{s,n} \cdot n = 0$ on $\Gamma^s$.
\end{remark}

\section{Well-posedness}
\label{sec:consiswellpose}

\subsection{Preliminary results}
\label{ss:cnotation}

Let $D$ be a domain. Norms on $W_p^k(D)$, $L^p(D)=W_p^0(D)$,
$H^k(D)=W_2^k(D)$, and $L^2(D)$ are denoted by, respectively,
$\norm{\cdot}_{W_p^k(D)}$, $\norm{\cdot}_{L^p(D)}$,
$\norm{\cdot}_{k,D}$, and $\norm{\cdot}_D$. Furthermore, for two real
numbers $a$, $b$, and a Banach space $X$ with norm
$\norm[0]{\cdot}_X$, $L^2(a,b;X)$ is defined as the space of square
integrable functions from $[a,b]$ into $X$ with norm
$\norm[0]{f}_{L^2(a,b;X)}:=(\int_a^b\norm[0]{f(t)}_X^2 \dif t)^{1/2}$
and $L^{\infty}(a,b;X)$ is the space of essentially bounded functions
from $[a,b]$ to $X$ with norm
$\norm[0]{f}_{L^{\infty}(a,b;X)} := \text{ess}
\sup_{[a,b]}\norm[0]{f(t)}_X$.

On $\Omega^s$ and $\Omega^d$, we define the function spaces
\begin{align*}
  X^s
  &:= \cbr[0]{v \in H^2(\Omega^s)^{\dim}\, :\, v=0 \text{ on } \Gamma^s},
  &
    Q^s
  &:= H^1(\Omega^s),
  \\
  X^d
  &:= \cbr[0]{v \in H^1(\Omega^d)^{\dim}\, : \, v\cdot n = 0 \text{ on } \Gamma^d_N},
  &
    Q^d
  &:= \cbr[0]{q \in H^2(\Omega^d)\, :\, q = 0 \text{ on } \Gamma^d_D}.  
\end{align*}
On $\Omega$, we then define
$X := \cbr[0]{ v \in H(\text{div};\Omega)\, :\, u^s \in X^s,\ u^d \in
  X^d}$ and
$Q := \{ q \in L^2(\Omega)\, :\, q^s \in Q^s, \ q^d \in Q^d \}$. The
trace space of $X^s$ on facets in $\Gamma_0^s$ is denoted by
$\bar{X}$. If $u \in X^s$, we denote its trace by
$\bar{u} := \gamma_X(u)$ where $\gamma_X : X^s \to \bar{X}$ is the
trace operator restricting functions in $X^s$ to
$\Gamma_0^s$. Similarly, the trace space of $Q^j$ on facets
$\Gamma_0^j$ is denoted by $\bar{Q}^j$,
$\gamma_{Q^j} : Q^j \to \bar{Q}^j$ is the trace operator, and if
$q \in Q^j$, then $\bar{q} := \gamma_{Q^j}(q) \in \bar{Q}^j$.

Using the notation $\boldsymbol{X} := X \times \bar{X}$ and
$\boldsymbol{Q} := Q \times \bar{Q}^s \times \bar{Q}^d$, we define
\begin{equation*}
  X(h) := X_h + X, \quad X^s(h) := X_h^s + X^s, \quad
  \boldsymbol{X}(h) := \boldsymbol{X}_h + \boldsymbol{X}, \quad
  \boldsymbol{Q}(h) := \boldsymbol{Q}_h + \boldsymbol{Q}.
\end{equation*}
As in \cite{Cesmelioglu:2023}, we define the following norms on
the extended function spaces:
\begin{align*}
  \tnorm{\boldsymbol{v}}_{v}^2
  :&=
     \tnorm{\boldsymbol{v}}_{v,s}^2
     + \tnorm{\boldsymbol{v}}_{v,d}^2
     + \norm[0]{ \bar{v}^t }^2_{\Gamma^I} && \boldsymbol{v} \in \boldsymbol{X}(h),
  \\
  \tnorm{\boldsymbol{v}}_{v'}^2
  :&= \tnorm{\boldsymbol{v}}_{v}^2
     + \sum_{K\in \mathcal{T}^s_h} h_K^2 \envert{v}_{2,K}^2
     =\tnorm{\boldsymbol{v}}_{v',s}^2  
     + \tnorm{\boldsymbol{v}}_{v,d}^2
     + \norm[0]{ \bar{v}^t }^2_{\Gamma^I} && \boldsymbol{v} \in \boldsymbol{X}(h),
  \\
  \tnorm{\boldsymbol{q}}_p^2 &:= \tnorm{\boldsymbol{q}^s}_{p,s}^2 + \tnorm{\boldsymbol{q}^d}_{p,d}^2
                                          && \boldsymbol{q} \in \boldsymbol{Q}(h),  
\end{align*}
where
\begin{equation*}
  \begin{split}
    \tnorm{\boldsymbol{v}}_{v,s}^2
    &:= \sum_{K\in \mathcal{T}^s_h}\del[1]{\norm[0]{\nabla v}_K^2
      + h_K^{-1}\norm{v-\bar{v}}^2_{\partial K}},
    \\
    \tnorm{\boldsymbol{v}}_{v',s}^2
    &:= \tnorm{\boldsymbol{v}}_{v,s}^2
    + \sum_{K\in \mathcal{T}^s_h} h_K^2 \envert{v}_{2,K}^2,
    \\
    \tnorm{\boldsymbol{v}}_{v,d}^2
    &:= \norm{v}_{\text{div};\Omega^d}^2 + \sum_{F \in \mathcal{F}^d_h\backslash(\mathcal{F}_h^I\cup\mathcal{F}_h^{D,d})}h_F^{-1}\norm[0]{\jump{v\cdot n}}_{F}^2
    + \sum_{K \in \mathcal{T}^d_h} h_K^{-1}\norm{(v - \bar{v})\cdot n}_{\partial K \cap \Gamma^I}^2,
    \\
    \tnorm{\boldsymbol{q}^j}_{p,j}^2
    &:= \norm{q}_{\Omega^j}^2 + \sum_{K
      \in \mathcal{T}^j_h} h_K \norm[0]{\bar{q}^j}_{\partial K}^2,\ j=s,d.    
  \end{split}  
\end{equation*}
Here $\jump{v \cdot n}$ is the usual jump operator and
$\norm{v}_{\text{div};\Omega^d}^2 := \norm{v}_{\Omega^d}^2 +
\norm{\nabla \cdot v}_{\Omega^d}^2$. Let us furthermore note that
$\norm{v_h}_{1,h,\Omega^s} := \tnorm{(v_h, \av{v_h})}_{v,s}$, where
$\norm{v_h}_{1,h,\Omega^s}$ is the standard discrete $H^1$-norm of
$v_h$ in $\Omega^s$ \cite{Cesmelioglu:2017}. Finally, we will also
require the following two norms on the pressure in $\Omega^d$:
\begin{align*}
  \norm[0]{q_h}_{1,h,\Omega^d}^2
  &:= \sum_{K \in \mathcal{T}_h^d} \norm[0]{\nabla q_h}_K^2
    + \sum_{F \in \mathcal{F}_h^{int,d} \cup \mathcal{F}_h^{D,d}} h_F^{-1} \norm[0]{\jump{q_h}}_F^2
  && \forall q_h \in Q_h^d,
  \\
  \tnorm{\boldsymbol{q}_h}_{1,h,d}^2
  &:= \sum_{K \in \mathcal{T}_h^d}
    \del[1]{ \norm[0]{\nabla q_h}_K^2 + h_K^{-1} \norm[0]{q_h - \bar{q}_h}_{\partial K}^2}
  && \forall \boldsymbol{q}_h \in \boldsymbol{Q}_h^d.  
\end{align*}
That $\norm[0]{q_h}_{1,h,\Omega^d}$ is a norm on $Q_h^d$ follows because
$|\Gamma_D^d| > 0$.

The following inequalities will be used in the remainder of this paper
(see \cite[eq.~(5.5)]{Wells:2011}, \cite[Theorem 4.4 and Proposition
4.5]{Girault:2009}, and \cite[Lemma~1.46]{Pietro:book}):
\begin{subequations}
  \begin{align}
    \label{eq:equivalencetnorm}
    \tnorm{\boldsymbol{v}_h}_v
    &\le \tnorm{\boldsymbol{v}_h}_{v'} \le c_e
      \tnorm{\boldsymbol{v}_h}_v
    && \forall \boldsymbol{v}_h \in \boldsymbol{X}_h,
    \\
    \label{eq:dpoincareineq}
    \norm{v_h}_{\Omega^s}
    &\le c_{p} \norm{v_h}_{1,h,\Omega^s}
      \le c_{p} \tnorm{\boldsymbol{v}_h}_{v,s}
    && \forall \boldsymbol{v}_h \in \boldsymbol{X}_h^s,
    \\
    \label{eq:dpoincareineq-qh}
    \norm{q_h}_{\Omega^d}
    &\le c_{pp} \norm{q_h}_{1,h,\Omega^d}
      \le c_{pp} \tnorm{\boldsymbol{q}_h}_{1,h,d}
    && \forall \boldsymbol{q}_h \in \boldsymbol{Q}_h^d,
    \\    
    \label{eq:dtrpoincareineq}
    \norm{v_h^s}_{L^r(\Gamma^I)}
    &\le c_{si,r} \norm{v_h}_{1,h,\Omega^s}
      \le c_{si,r} \tnorm{\boldsymbol{v}_h}_{v,s}
    && \forall \boldsymbol{v}_h \in \boldsymbol{X}_h^s,\ r \ge 2,
    \\
    \label{eq:trace-1}
    \norm{v}_{\partial K}
    &\le c_{tr} h_K^{-1/2} \norm[0]{v}_K
    && \forall v \in P_k(K), \ K \in \mathcal{T}_h,    
  \end{align}  
\end{subequations}  
where $c_e$, $c_p$, $c_{si,r}$, and $c_{tr}$ are positive constants
independent of $h$ and $\Delta t$.

For $b_h$ we have:
\begin{subequations}
  \label{eq:boundsbh}
  \begin{align}
    \label{eq:infsupbh}
    c_{bb} \tnorm{\boldsymbol{q}_h}_p
    &\le
      \sup_{0 \ne \boldsymbol{v}_h \in \boldsymbol{X}_h}
    \frac{b_h(\boldsymbol{v}_h, \boldsymbol{q}_h)}{\tnorm{\boldsymbol{v}_h}_{v}} && \forall \boldsymbol{q}_h \in \boldsymbol{Q}_h,
    \\
    \label{eq:bhbnd}
    |b_h(\boldsymbol{v}, \boldsymbol{q})|
    &\le c_{bc} \tnorm{\boldsymbol{v}}_v \tnorm{\boldsymbol{q}}_p
    && \forall (\boldsymbol{v}, \boldsymbol{q}) \in \boldsymbol{X}(h) \times \boldsymbol{Q}_h.
  \end{align}    
\end{subequations}
Due to the use of different function spaces, the inf-sup condition
\cref{eq:infsupbh} is different from that proven in
\cite{Cesmelioglu:2023}. We therefore prove \cref{eq:infsupbh} in
\cref{ap:infsup}. \Cref{eq:bhbnd} is proven in \cite[Lemma
3]{Cesmelioglu:2023}. For $a_h^s$, $a^d$, and $a^I$, we have that for
all $\boldsymbol{u}, \boldsymbol{v} \in \boldsymbol{X}(h)$,
\begin{equation}
  \label{eq:ahboundedness}
  |a_h^s(\boldsymbol{u}, \boldsymbol{v})|
  \le \mu c_{ac}^s \tnorm{\boldsymbol{u}}_{v',s}\tnorm{\boldsymbol{v}}_{v',s},
  \quad
  |a^d(u, v)|
  \le \mu\kappa^{-1} \norm{u}_{\Omega^d}\norm{v}_{\Omega^d},
  \quad
  |a^I(\bar{u}, \bar{v})|
  \le \alpha\mu\kappa^{-1/2}\norm[0]{\bar{u}^t}_{\Gamma^I}\norm[0]{\bar{v}^t}_{\Gamma^I},      
\end{equation}
where $c_{ac}^s>0$ is a constant independent of $h$ and $\Delta
t$. For $\boldsymbol{v}_h \in \boldsymbol{X}_h$ we have
\begin{equation}
  \label{eq:ahcoercivity}
  a_h^s(\boldsymbol{v}_h, \boldsymbol{v}_h)
  \ge \mu c_{ae}^s \tnorm{\boldsymbol{v}_h}_{v,s}^2,
  \quad
  a^d(v_h, v_h)
  \ge \mu \kappa^{-1} \norm{v_h}_{\Omega^d}^2,
  \quad
  a^I(\bar{v}_h, \bar{v}_h)
  \ge \alpha\mu\kappa^{-1/2}\norm[0]{\bar{v}_h^t}_{\Gamma^I}^2,      
\end{equation}
where the first inequality holds for $\beta$ large enough and where
$c_{ae}^s > 0$ is a constant independent of $h$ and $\Delta t$. A
direct consequence of \cref{eq:ahboundedness,eq:ahcoercivity} is that
\begin{subequations}
  \begin{align}
    \label{eq:ahLbddness}
    |a_h^L(\boldsymbol{u}, \boldsymbol{v})|
    &\leq \mu c_{ac}^{L}\tnorm{\boldsymbol{u}}_{v'} \tnorm{\boldsymbol{v}}_{v'}
    && \forall \boldsymbol{u}, \boldsymbol{v} \in \boldsymbol{X}(h),
    \\
    \label{eq:ahLcoercive}
    |a_h^L(\boldsymbol{v}_h, \boldsymbol{v}_h)|
    &\geq \mu c_{ae}^{L}\tnorm{\boldsymbol{v}_h}_{v}^2
    && \forall \boldsymbol{v}_h \in \boldsymbol{X}_h,
  \end{align}
\end{subequations}
where $c_{ac}^L:=\max( c_{ac}^s, \kappa^{-1}, \alpha\kappa^{-1/2})>0$
and $c_{ae}^L:=\min( c_{ae}^s, \kappa^{-1}, \alpha\kappa^{-1/2})>0$
are constants independent of $h$ and $\Delta t$, and where \cref{eq:ahLcoercive}
holds for $\beta$ large enough.

We also recall the following inequality from \cite[Lemma
4]{Cesmelioglu:2023}, \cite[Proposition 3.4]{Cesmelioglu:2017} related
to the form $t_h$. Assuming that
$w_1,w_2 \in X^s(h) \cap H(\text{div};\Omega^s)$ are such that
$\nabla \cdot w_1 = \nabla \cdot w_2 = 0$ on each
$K \in \mathcal{T}^s$ it holds for any
$\boldsymbol{u} \in \boldsymbol{X}^s(h)$,
$\boldsymbol{v} \in \boldsymbol{X}_h^s$ that
\begin{equation}
  \label{eq:boundedness_th}
  \envert[0]{t_h(w_1; \boldsymbol{u}, \boldsymbol{v}) - t_h(w_2; \boldsymbol{u}, \boldsymbol{v})}
  \le c_w \norm{w_1 - w_2}_{1,h,\Omega^s} \tnorm{\boldsymbol{u}}_{v,s} \tnorm{\boldsymbol{v}}_{v,s},
\end{equation}
where $c_w > 0$ is a constant independent of $h$ and $\Delta t$.

Assuming $w \in X^s(h) \cap H(\text{div};\Omega^s)$ is such that
$\nabla \cdot w = 0$ on each $K \in \mathcal{T}_h^s$, then \cite[Lemma
5]{Cesmelioglu:2023}
\begin{subequations}
  \begin{align}
    \label{eq:ahboundedX}
    \envert[0]{a_h(w;\boldsymbol{u},\boldsymbol{v})}
    &\le c_{ac} \mu \tnorm{\boldsymbol{u}}_{v'}\tnorm{\boldsymbol{v}}_{v'}
    && \forall \boldsymbol{u}, \boldsymbol{v} \in \boldsymbol{X}(h),
    \\
    \label{eq:ahboundedXh}
    \envert[0]{a_h(w;\boldsymbol{u}_h,\boldsymbol{v}_h)}
    &\le c_{ac} \mu \tnorm{\boldsymbol{u}_h}_{v}\tnorm{\boldsymbol{v}_h}_{v}
    && \forall \boldsymbol{u}_h, \boldsymbol{v}_h \in \boldsymbol{X}_h,
  \end{align}
\end{subequations}
where
$c_{ac} = 2c_e^2\max(c_w\mu^{-1}\norm{w}_{1,h,\Omega^s} + c_{ac}^s,
\kappa^{-1}, \alpha \kappa^{-1/2})$. Let
us now define
\begin{align*}
  \boldsymbol{Z}_h^s
  :&= \cbr[1]{ \boldsymbol{v}_h \in \boldsymbol{X}_h^s:\ b_h^s(v_h,\boldsymbol{q}^s_h) + b_h^{I,s}(\bar{v}_h,\bar{q}_h^s)
     = 0 \ \forall \boldsymbol{q}_h^s \in \boldsymbol{Q}_h^s},
  \\
  \boldsymbol{Z}_h
  :&= \cbr[1]{ \boldsymbol{v}_h \in \boldsymbol{X}_h:\
     \sum_{j=s,d}\del[0]{b_h^j(v_h,\boldsymbol{q}^j_h) + b_h^{I,j}(\bar{v}_h,\bar{q}_h^j)}
     = 0 \ \forall \boldsymbol{q}_h \in \boldsymbol{Q}_h}.
\end{align*}
If $w \in X^s(h) \cap H(\text{div};\Omega^s)$ such that
$\nabla \cdot w = 0$ on each $K \in \mathcal{T}_h^s$ and
$\norm{w \cdot n}_{\Gamma^I} \le \tfrac{1}{2}\mu c_{ae}^s /
(c_{pq}^2+c_{si,4}^2)$ on the interface, and if $\beta$ is large
enough that the first inequality in \cref{eq:ahcoercivity} holds, then
it was shown in \cite[Lemma 6]{Cesmelioglu:2023} that,
\begin{equation}
  \label{eq:coercivity_awhvhvh}
  a_h(w; \boldsymbol{v}_h, \boldsymbol{v}_h)
  \ge c_{ae} \mu
  \tnorm{\boldsymbol{v}_h}_{v}^2
  \quad \forall \boldsymbol{v}_h \in \boldsymbol{Z}_h,
\end{equation}
where
$c_{ae} = \min\del[1]{\tfrac{1}{2}c_{ae}^s, \kappa^{-1},
  \alpha\kappa^{-1/2}}$.

Using a proof similar to \cite[Lemma 1]{Cesmelioglu:2023}, it is
straightforward to obtain the following result.

\begin{lemma}[Consistency]
  \label{lem:consistency}
  Suppose that $(u, p)$ is the solution to
  \cref{eq:ns,eq:darcy,eq:interface,eq:ic-bc} that satisfies
  $u \in L^2(J;X)$, $p \in L^2(J;Q)$, and
  $\partial_tu \in L^2(J;L^2(\Omega^s))$. Let
  $\boldsymbol{u}=(u,\bar{u})$ and
  $\boldsymbol{p}=(p, \bar{p}^s, \bar{p}^d)$ and assume that
  $f^s\in C^0(J;L^2(\Omega^s)^{\dim})$ and
  $f^d\in C^0(J;L^2(\Omega^d))$. Then
  $(\boldsymbol{u}, \boldsymbol{p})$ satisfies \cref{eq:hdg-semi} for
  all $t > 0$.
\end{lemma}

\subsection{Existence and uniqueness}
\label{ss:energystab}

We start this section with some auxiliary results.

\begin{lemma}
  \label{lem:tnormphduhd}
  For $\boldsymbol{p}_h^{d,n}$ and $u_h^{d,n}$ that satisfy
  \cref{eq:hdg-fully}, there exists a $c_{pd}>0$, independent of $h$
  and $\Delta t$, such that
  \begin{equation}
    \label{eq:boundphdf}
    \tnorm{\boldsymbol{p}_h^{d,n}}_{1,h,d}
    \le c_{pd} \mu\kappa^{-1} \norm[0]{u_h^{d,n}}_{\Omega^d}.
  \end{equation}
\end{lemma}
\begin{proof}
  We will prove \cref{eq:boundphdf} in three dimensions only noting
  that the proof in two dimensions is similar. To ease notation we
  will drop the ``time'' superscript $n$. The proof follows the proof
  of \cite[Lemma 2.1]{Lovadina:2006} with modifications made to take
  into account Brezzi--Douglas--Marini (BDM) elements and HDG facet
  functions.

  The local degrees of freedom for the BDM element
  are~\cite[Proposition 2.3.2]{Boffi:book}:
  \begin{equation}
    \label{eq:local-dofs-bdm}
    \langle v_h \cdot n, \bar{r}_h \rangle_{\partial K},
    \ \forall \bar{r}_h \in R_k(\partial K)
    \qquad \text{ and } \qquad
    (v_h, z_h)_K,
    \ \forall z_h \in \mathcal{N}_{k-2}(K),    
  \end{equation}
  where
  $R_k(\partial K) := \cbr[0]{ \bar{r} \in L^2(\partial K)\, :\,
    \bar{r}|_F \in P_k(F),\ \forall F \subset \partial K}$ and
  $\mathcal{N}_{k-2}(K)$ is the N\'ed\'elec space. Therefore, given
  $\boldsymbol{p}_h^d \in \boldsymbol{Q}_h^d$, we define
  $w_h \in V_h^d \cap H(\text{div};\Omega^d)$ such that
  \begin{subequations}
    \label{eq:whdofs}
    \begin{align}
      \label{eq:whdofs-a}
      \langle w_h \cdot n, \bar{r}_h \rangle_{\partial K}
      &= h_K^{-1}\langle p_h^d - \bar{p}_h^d, \bar{r}_h \rangle_{\partial K}
      && \forall \bar{r}_h \in R_k(\partial K), \ \forall K \in \mathcal{T}_h^d,
      \\
      \label{eq:whdofs-c}
      (w_h, z_h)_K
      &= -(\nabla p_h^d, z_h)_K
      && \forall z_h \in \mathcal{N}_{k-2}(K), \ \forall K \in \mathcal{T}_h^d.         
    \end{align}
  \end{subequations}
  Since
  $\nabla p_h^d \in \nabla P_{k-1}(K) \subset \sbr[0]{P_{k-2}}^3
  \subset \mathcal{N}_{k-2}(K)$ and since
  $p_h^d-\bar{p}_h^d \in R_k(\partial K)$, we obtain from
  \cref{eq:whdofs} that
  \begin{subequations}
    \label{eq:propertieswh}
    \begin{align}
      \langle w_h \cdot n, p_h^d-\bar{p}_h^d \rangle_{\partial K}
      &= h_K^{-1} \norm[0]{p_h^d - \bar{p}_h^d}_{\partial K}^2
      && \forall K \in \mathcal{T}_h^d,
      \\      
      (w_h, \nabla p_h^d)_K
      &= -\norm[0]{\nabla p_h^d}_K^2
      && \forall K \in \mathcal{T}_h^d.
    \end{align}
  \end{subequations}

  Setting now $\boldsymbol{v}_h^{s}=0$ and $\boldsymbol{q}_h=0$ in
  \cref{eq:hdg-fully}, and after integration by parts, we find for all
  $v_h \in V_h^d$ that:
  \begin{equation}
    \label{eq:whzzhnew}
    \begin{split}
      0
      &= 
      (\mu\kappa^{-1}u_h^d, v_h)_{\Omega^d} - (p_h^d, \nabla \cdot v_h)_{\Omega^d}
      + \langle \bar{p}_h^d, v_h \cdot n^d \rangle_{\partial\mathcal{T}_h^d}
      \\
      &
      =
      (\mu\kappa^{-1}u_h^d, v_h)_{\Omega^d}
      + (\nabla p_h^d, v_h)_{\Omega^d}
      - \langle p_h^d - \bar{p}_h^d, v_h \cdot n^d \rangle_{\partial \mathcal{T}_h^d}.
    \end{split}
  \end{equation}
  Choose $v_h = w_h$, with $w_h$ defined in \cref{eq:whdofs}. By
  \cref{eq:propertieswh}, \cref{eq:whzzhnew}, and the definition of
  $\tnorm{\boldsymbol{p}_h^d}_{1,h,d}^2$, we find
  \begin{equation}
    \label{eq:boundphd}
    \tnorm{\boldsymbol{p}_h^d}_{1,h,d}^2
    = (\mu\kappa^{-1}u_h^d, w_h)_{\Omega^d}
    \le \mu\kappa^{-1} \norm[0]{u_h^d}_{\Omega^d}\norm[0]{w_h}_{\Omega^d}.
  \end{equation}
  To find out more about $\norm[0]{w_h}_{\Omega^d}$, let us define
  the norm $\norm[0]{\cdot}_{0,h}$ for functions in
  $V_h^d \cap H(\text{div};\Omega^d)$:
  \begin{equation}
    \norm[0]{w_h}_{0,h}^2 := \norm[0]{w_h}_{\Omega^d}^2 + \sum_{F \in \mathcal{F}_h^d}h_F \norm[0]{w_h \cdot n}_F^2.
  \end{equation}
  Consider now a single element $K$ and denote by $\mathcal{F}_K$ the
  set of faces of $K$. In an approach similar to that used in the
  proof of \cite[Lemma 4.4]{Linke:2018}, we have:
  \begin{equation}
    \label{eq:bounding-norm-whdiv}
    \begin{split}
      \norm[0]{w_h}_{K}^2 + \sum_{F \in \mathcal{F}_K} h_F \norm[0]{w_h \cdot n}_{F}^2      
      &\lesssim \sup_{\substack{ z_h \in \mathcal{N}_{k-2}(K)^3 \\ \norm[0]{z_h}_{K}=1}} | (w_h, z_h)_K |^2 +
      \sup_{\substack{ \bar{r}_h \in R_k(\partial K) \\ \norm[0]{\bar{r}_h}_{\partial K}=1}} h_K |\langle w_h \cdot n, \bar{r}_h \rangle_F |^2
      \\
      &= \sup_{\substack{ z_h \in \mathcal{N}_{k-2}(K)^3 \\ \norm[0]{z_h}_{K}=1}} | (\nabla p_h^d, z_h)_K |^2 +
      \sup_{\substack{ \bar{r}_h \in R_k(\partial K) \\ \norm[0]{\bar{r}_h}_{\partial K}=1}} h_Kh_K^{-2} |\langle p_h^d - \bar{p}_h^d, \bar{r}_h \rangle_{\partial K} |^2
      \\
      &\le \norm[0]{\nabla p_h^d}_{K}^2 + h_K^{-1} \norm[0]{p_h^d - \bar{p}_h^d}_{\partial K}^2,
    \end{split}
  \end{equation}
  where the first line on the right hand side is by using the degrees
  of freedom \cref{eq:local-dofs-bdm}, the second by definition of
  $w_h$ given by \cref{eq:whdofs}, and the last is by the
  Cauchy--Schwarz inequality. Therefore, after summing
  \cref{eq:bounding-norm-whdiv} over all $K$ in $\mathcal{T}_h^d$:
  \begin{equation*}
    \norm[0]{w_h}_{\Omega^d}^2
    \le
    \norm{w_h}_{0,h}^2
    \lesssim \sum_{K \in \mathcal{T}_h^d} \del[1]{\norm[0]{\nabla p_h^d}_{K}^2 + h_K^{-1} \norm[0]{p_h^d - \bar{p}_h^d}_{\partial K}^2}
    = \tnorm{\boldsymbol{p}_h^d}_{1,h,d}^2.
  \end{equation*}
  The result follows after combining this with \cref{eq:boundphd}.
\end{proof}

An immediate consequence of \cref{eq:dpoincareineq-qh} and Lemma
\ref{lem:tnormphduhd} is that if $\boldsymbol{p}_h^{d,n}$ and
$u_h^{d,n}$ satisfy \cref{eq:hdg-fully}, then for $1 \le n \le N$:
\begin{equation}
  \label{eq:phdnvsuhdn}
  \norm[0]{p_h^{d,n}}_{\Omega^d} \le c_{pp} \norm[0]{p_h^{d,n}}_{1,h,\Omega^d} \le c_{pp}\tnorm{\boldsymbol{p}_h^{d,n}}_{1,h,d}
  \le c_{td} \mu \kappa^{-1} \norm[0]{u_h^{d,n}}_{\Omega^d},
\end{equation}
where $c_{td} = c_{pp}c_{pd}$.

The following result, which was shown in \cite[Theorem
5.2]{Chaabane:2017}, will be used to prove the next lemma: there
exists a constant $c > 0$, independent of $h$ and $\Delta t$, such
that
\begin{equation}
  \label{eq:chaabaneinterface}
  |\langle q_h, v_h \cdot n \rangle_{\Gamma^I}| \le c \norm[0]{q_h}_{1,h,\Omega^d} \norm[0]{v_h}_{\Omega^s}
  \quad \forall v_h \in \widetilde{V}_h^s,\ \forall q_h \in Q_h^d,
\end{equation}
where
$\widetilde{V}_h^s := \cbr[0]{v_h \in X_h^s\, :\, b_s(v_h, q_h) = 0 \
  \forall q_h \in Q_h^s}$ with
$b_s(v, q) := -(q, \nabla \cdot v)_{\Omega^s} + \sum_{F \in
  \mathcal{F}_h^{int,s} \cup \mathcal{F}_h^{B,s}} \langle \av{q},
\jump{v}\cdot n \rangle_F$.

\begin{lemma}
  \label{lem:phbduhsnGammaI}
  Let $u_h^{s,n}$, $u_h^{d,n}$ and $\bar{p}_h^{d,n}$ be (part of) the
  solution to \cref{eq:hdg-fully}. There exists a constant
  $c_{sdi}> 0$, independent of $h$ and $\Delta t$, such that for all
  $n \ge 1$
  \begin{equation}
    |\langle \bar{p}_h^{d,n}, u_h^{s,n} \cdot n \rangle_{\Gamma^I}|
    \le c_{sdi} \mu \kappa^{-1}\norm[0]{u_h^{d,n}}_{\Omega^d}\norm[0]{u_h^{s,n}}_{\Omega^s}.
  \end{equation}
\end{lemma}
\begin{proof}
  For ease of notation we will drop the ``time'' superscript
  $n$. Then, note that
  \begin{equation}
    \label{eq:pbhdubhs}
    |\langle \bar{p}_h^d, \bar{u}_h^s \cdot n \rangle_{\Gamma^I}|
    \le
    |\langle \bar{p}_h^d-p_h^d, \bar{u}_h^s \cdot n \rangle_{\Gamma^I}| + |\langle p_h^d, \bar{u}_h^s \cdot n \rangle_{\Gamma^I}|
    \le
    |\langle \bar{p}_h^d-p_h^d, u_h^s \cdot n \rangle_{\Gamma^I}| + |\langle p_h^d, u_h^s \cdot n \rangle_{\Gamma^I}|.
  \end{equation}
  Since $u_h^s$ is a solution to \cref{eq:hdg-fully}, by Remark
  \ref{rem:properties-hdg-fully} we know that $\nabla\cdot u_h^s = 0$
  and $\jump{u_h^s}\cdot n = 0$ on
  $F \in \mathcal{F}_h^{int,s} \cup \mathcal{F}_h^{B,s}$ so that
  $u_h^s \in \widetilde{V}_h^s$. Therefore, using
  \cref{eq:chaabaneinterface},
  \begin{equation}
    \label{eq:phduhs}
    |\langle p_h^d, u_h^s \cdot n \rangle_{\Gamma^I}|
    \le C \norm[0]{u_h^s} _{\Omega^s}\norm[0]{p_h^d}_{1,h,\Omega^d}.
  \end{equation}
  Next, using \cref{eq:trace-1} and Lemma \ref{lem:tnormphduhd}, we
  note that
  \begin{equation}
    \label{eq:pbhdqhduhs}
    \begin{split}
      |\langle \bar{p}_h^d-p_h^d, u_h^s \cdot n \rangle_{\Gamma^I}|
      \le&
      \del[1]{\sum_{K \in \mathcal{T}_h^s}h_K\norm[0]{u_h^s\cdot n}_{\partial K}^2}^{1/2}
      \del[1]{\sum_{K \in \mathcal{T}_h^d}h_K^{-1}\norm[0]{\bar{p}_h^d - p_h^d}_{\partial K}^2}^{1/2}
      \\
      \le& C \norm[0]{u_h^s}_{\Omega^s}\del[1]{\sum_{K \in \mathcal{T}_h^d}h_K^{-1}\norm[0]{\bar{p}_h^d - p_h^d}_{\partial K}^2}^{1/2}
      \\
      \le& C \norm[0]{u_h^s}_{\Omega^s} \tnorm{\boldsymbol{p}_h^d}_{1,h,d}      
      \le C \mu \kappa^{-1} \norm[0]{u_h^s}_{\Omega^s}\norm[0]{u_h^d}_{\Omega^d}.
   \end{split}
  \end{equation}
  The result follows by combining
  \cref{eq:pbhdubhs,eq:phduhs,eq:pbhdqhduhs,eq:phdnvsuhdn}.
\end{proof}

For the remainder of this section we define
\begin{equation*}
  \boldsymbol{B}_h^s := \cbr[1]{ \boldsymbol{v}_h^s \in \boldsymbol{Z}_h^s:
    \quad \tnorm{\boldsymbol{v}_h^s}_{v,s} \le \tfrac{1}{2}\mu\min\del[1]{c_{ae}^sc_{si,2}^{-1}(c_{pq}^2 + c_{si,4}^2)^{-1}, c_{ae}c_w^{-1}}}.
\end{equation*}

\begin{lemma}
  \label{lemap:exisuniqnp1}
  For $0\leq n\leq N-1$, let
  $\boldsymbol{u}_h^{s,n} \in \boldsymbol{B}_h^s$. Then
  \cref{eq:hdg-fully} has a unique solution
  $(\boldsymbol{u}_h^{n+1}, \boldsymbol{p}_h^{n+1}) \in
  \boldsymbol{X}_h \times \boldsymbol{Q}_h$.
\end{lemma}
\begin{proof}
  Consider \cref{eq:hdg-fully} for the solution at time level
  $t^{n+1}$ which we write here as:
  \begin{multline}
    \label{eq:aphdg-fully}
    \tfrac{1}{\Delta t}\del[0]{ u_h^{s,n+1}, v_h}_{\Omega^s}
    + a_h(u_h^n; \boldsymbol{u}_h^{n+1}, \boldsymbol{v}_h)
    + b_h(\boldsymbol{v}_h, \boldsymbol{p}_h^{n+1})
    + b_h(\boldsymbol{u}_h^{n+1}, \boldsymbol{q}_h)
    \\
    =
    \tfrac{1}{\Delta t}\del[0]{ u_h^{s,n}, v_h}_{\Omega^s}
    + (f^{s,n+1}, v_h)_{\Omega^s} + (f^{d,n+1}, q_h)_{\Omega^d}.
  \end{multline}
  Given $\boldsymbol{u}_h^{s,n} \in \boldsymbol{B}_h^s$ we remark
  that, by \cref{eq:dtrpoincareineq} with $r=2$ and
  \cref{eq:coercivity_awhvhvh},
  \begin{equation}
    \label{eq:coercivitylhsterms}
    \tfrac{1}{\Delta t}\del[0]{ v_h, v_h}_{\Omega^s}
    + a_h(u_h^n; \boldsymbol{v}_h, \boldsymbol{v}_h)
    \ge c_{ae} \mu \tnorm{\boldsymbol{v}_h}_{v}^2
    \quad \forall \boldsymbol{v}_h \in \boldsymbol{Z}_h,
  \end{equation}
  Furthermore, by \cref{eq:ahboundedXh} and \cref{eq:dpoincareineq},
  we obtain the following boundedness result:
  \begin{equation}
    \label{apeq:boundeness-of-uh-assump}
    \tfrac{1}{\Delta t}(u_h^s, v_h)_{\Omega^s} + \envert[0]{a_h(u_h^{s,n};\boldsymbol{u}_h,\boldsymbol{v}_h)}
    \le \del[1]{\tfrac{1}{\Delta t}c_p^2 + c_f\mu}\tnorm{\boldsymbol{u}_h}_{v}\tnorm{\boldsymbol{v}_h}_{v}
    \quad \forall \boldsymbol{u}_h, \boldsymbol{v}_h \in \boldsymbol{X}_h,
  \end{equation}
  where
  \begin{equation*}
    c_f = 2c_e^2\max\del[2]{\tfrac{1}{2}\min\del[2]{c_wc_{ae}^sc_{si,2}^{-1}(c_{pq}^2 + c_{si,4}^2)^{-1}, c_{ae}} + c_{ac}^s,
      \kappa^{-1}, \alpha \kappa^{-1/2}}.
  \end{equation*}
  Here $c_f$ is an upper bound for $c_{ac}$ using that
  $\boldsymbol{u}_h^{s,n} \in \boldsymbol{B}_h^s$. Since
  $\boldsymbol{u}_h^{s,n} \in \boldsymbol{B}_h^s$, boundedness of the
  right hand side of \cref{eq:aphdg-fully} follows from the
  Cauchy--Schwarz inequality. Existence of a unique solution
  $(\boldsymbol{u}_h^{n+1}, \boldsymbol{p}_h^{n+1}) \in
  \boldsymbol{X}_h \times \boldsymbol{Q}_h$ to \cref{eq:hdg-fully} is
  now a consequence of \cref{eq:coercivitylhsterms},
  \cref{apeq:boundeness-of-uh-assump}, \cref{eq:boundsbh} and
  \cite[Theorem 3.4.3]{Boffi:book}.
\end{proof}

Lemma \ref{lemap:exisuniqnp1} guarantees existence and uniqueness of a
solution
$(\boldsymbol{u}_h^{n+1}, \boldsymbol{p}_h^{n+1}) \in \boldsymbol{X}_h
\times \boldsymbol{Q}_h$ at time level $n+1$ provided that
$\boldsymbol{u}_h^{s,n} \in \boldsymbol{B}_h^s$. However, Lemma
\ref{lemap:exisuniqnp1} does not guarantee that
$\boldsymbol{u}_h^{s,n+1} \in \boldsymbol{B}_h^s$. Therefore, the
remainder of this section is dedicated to showing that
$\boldsymbol{u}_h^{s,n+1} \in \boldsymbol{B}_h^s$ under a smallness
assumption on the data. First we obtain bounds on
$\norm[0]{d_tu_h^{s,k}}_{\Omega^s}$ and
$\Delta t^{-1/2}\tnorm{\boldsymbol{u}_h^{s,1}}_{v,s}$ (which are
proven in Lemma \ref{lem:bounddtuhuh1}) after which we prove a bound
on $\tnorm{\boldsymbol{u}_h^{k}}_v$ (see Lemma
\ref{lem:uhskbound}). The steps used to obtain these results are
similar to \cite{Chaabane:2017}. In Lemma
\ref{lem:boundednessvelocitypressure} we then impose a smallness
assumption on the data to show existence and uniqueness of the
solution
$(\boldsymbol{u}_h^{n}, \boldsymbol{p}_h^{n}) \in \boldsymbol{X}_h
\times \boldsymbol{Q}_h$ for all time levels $1 \le n \le N$.

The following lemmas will be proven in three dimensions with similar
proofs holding for two dimensions. We assume that
$f^s \in C^0(J;L^2(\Omega^s)^3)$ and $f^d \in
C^0(J;L^2(\Omega^d))$. It will furthermore be useful to define
\begin{align}
  \label{eq:Fm}
  F^m &:= \frac{c_p^2}{c_{ae}\mu}\Delta t\sum_{k=1}^m\norm[0]{d_t f^{s,k+1}}_{\Omega^s}^2
        + \frac{c_{td}^2\mu}{\kappa^2 c_{ae}}\Delta t \sum_{k=1}^m\norm[0]{d_t f^{d,k+1}}_{\Omega^d}^2,
  \\
  \label{eq:Mm}
  (M^m)^2 & := (M^0)^2 + \tfrac{1}{2}c_{ae}\mu G^2 + F^m,
\end{align}
where 
\begin{align}
  \label{eq:M02}
  M^0 & := \Big(1+c_{sdi}\big(\dfrac{\mu\Delta t}{2\kappa}\big)^{1/2}\Big)\norm[0]{f^{s,1}}_{L^2(\Omega^s)^3}
        + c_{sdi}\mu\kappa^{-1}c_{td}\norm[0]{f^{d,1}}_{L^2(\Omega^d)}
  \\
  \label{eq:G2}
  G^2 &:= \frac{1}{2c_{ae}^s}\del[2]{\del[2]{\dfrac1{\mu}+\frac{c_{sdi}^2}{2\kappa}\Delta t}\norm[0]{f^{s,1}}_{L^2(\Omega^s)^3}^2 
        + \frac{c_{sdi}^2c_{td}^2\mu}{\kappa^2} \norm[0]{f^{d,1}}_{L^2(\Omega^d)}^2}.
\end{align}

\begin{lemma}
  \label{lem:bounddtuhuh1}
  Let $\boldsymbol{u}_h^{s,0}=\boldsymbol{0}$ and let $M^0$, $G$,
  $F^m$ and $G^m$ be defined as in
  \cref{eq:M02,eq:G2,eq:Fm,eq:Mm}. Suppose that \cref{eq:hdg-fully}
  has a solution $(\boldsymbol{u}_h^k,\boldsymbol{p}_h^k)$ for all
  $1 \le k \le n$. For $k=1$,
  \begin{subequations}
    \label{eq:dtuhs1-tnormuhs1dt}
    \begin{align}
      \label{eq:dtuhs1}
      \norm[0]{d_tu_h^{s,1}}_{\Omega^s} & \le M^0,
      \\
      \label{eq:tnormuhs1dt}
      \frac{1}{(\Delta t)^{1/2}}\tnorm{\boldsymbol{u}_h^{s,1}}_{v,s} & \le G.
    \end{align}
  \end{subequations}
  Furthermore, if $\boldsymbol{u}_h^{s,k} \in \boldsymbol{B}_h^s$ for all
  $0 \le k \le n$, with $1 \le n \le N-1$, then
  \begin{equation}
    \label{eq:dtuhsnp1bound}
    \norm[0]{d_t u_h^{s,n+1}}_{\Omega^s} \le M^n.
  \end{equation}
\end{lemma}
\begin{proof}
  We first prove \cref{eq:dtuhs1-tnormuhs1dt}. Choose
  $\boldsymbol{v}_h^s = \boldsymbol{u}_h^{s,1}$,
  $\boldsymbol{q}_h^s = -\boldsymbol{p}_h^{s,1}$, $v_h^d = 0$,
  $q_h^d=0$, and $\bar{q}_h^d = -\bar{p}_h^{d,1}$ in
  \cref{eq:hdg-fully}. At $n=0$, since $u_h^{s,0}=0$, this reduces to:
  \begin{equation}
    \label{eq:simpuhs1dt}
    \tfrac{1}{\Delta t} \norm[0]{u_h^{s,1}}_{\Omega^s}^2
    + a_h^s(\boldsymbol{u}_h^{s,1}, \boldsymbol{u}_h^{s,1})
    + a^I(\bar{u}_h^{s,1}, \bar{u}_h^{s,1})
    = (f^{s,1}, u_h^{s,1})_{\Omega^s} + b_h^d(u_h^{d,1}, (0, \bar{p}_h^{d,1})).
  \end{equation}
  We bound the second term on the right hand side:
  \begin{equation*}
    \begin{split}
      |b_h^d(u_h^{d,1}, (0, \bar{p}_h^{d,1}))|
      &= |\langle \bar{p}_h^{d,1}, u_h^{d,1} \cdot n^d \rangle_{\partial \mathcal{T}_h^d}|
      = |\langle \bar{p}_h^{d,1}, u_h^{d,1} \cdot n^d \rangle_{\Gamma^I}|
      \\
      &= |\langle \bar{p}_h^{d,1}, u_h^{s,1} \cdot n \rangle_{\Gamma^I}|
      \le c_{sdi}\mu\kappa^{-1} \norm[0]{u_h^{d,1}}_{\Omega^d} \norm[0]{u_h^{s,1}}_{\Omega^s},      
    \end{split}
  \end{equation*}
  where the first equality is by definition, the second equality is
  because $\bar{p}_h^{d,1}$ and $u_h^{d,1} \cdot n^d$ are
  single-valued on $F \in \mathcal{F}_h^{int,d}$,
  $\bar{p}_h^{d,1} = 0$ on $\Gamma_D^d$ and $u_h^{d,1} \cdot n^d = 0$
  on $\Gamma_N^d$, and the third equality is because
  $u_h^1 \in H(\text{div};\Omega)$ (see Remark
  \ref{rem:properties-hdg-fully}). Finally, the inequality is by Lemma
  \ref{lem:phbduhsnGammaI}. Combining this with \cref{eq:simpuhs1dt},
  the coercivity of $a_h^s$ and $a^I$ \cref{eq:ahcoercivity}, and the
  Cauchy--Schwarz inequality,
  \begin{equation}
    \label{eq:simpuhs1dtcoerc}
    \tfrac{1}{\Delta t} \norm[0]{u_h^{s,1}}_{\Omega^s}^2
    + \mu c_{ae}^s \tnorm{\boldsymbol{u}_h^{s,1}}_{v,s}^2
    + \alpha\mu\kappa^{-1/2}\norm[0]{ (\bar{u}_h^{s,1})^t }_{\Gamma^I}^2
    \le \norm[0]{f^{s,1}}_{\Omega^s} \norm[0]{u_h^{s,1}}_{\Omega^s}
    + c_{sdi}\mu\kappa^{-1} \norm[0]{u_h^{d,1}}_{\Omega^d} \norm[0]{u_h^{s,1}}_{\Omega^s},
  \end{equation}
  directly implying, after ignoring the non-negative second and third
  terms on the left hand side, and canceling 
  $\norm[0]{u_h^{s,1}}_{\Omega^s}$,
  \begin{equation}
    \label{eq:simpuhs1dtcoerc-imp}
    \norm[0]{d_tu_h^{s,1}}_{\Omega^s}
    \le \norm[0]{f^{s,1}}_{\Omega^s}
    + c_{sdi}\mu\kappa^{-1} \norm[0]{u_h^{d,1}}_{\Omega^d}.
  \end{equation}
  
  Furthermore, applying Young's inequality to both terms on the right
  hand side of \cref{eq:simpuhs1dtcoerc} we also find:
  \begin{equation}
    \label{eq:simpuhs1dtcoerc-young1}
    \tfrac{1}{\Delta t} \norm[0]{u_h^{s,1}}_{\Omega^s}^2
    + \mu c_{ae}^s \tnorm{\boldsymbol{u}_h^{s,1}}_{v,s}^2
    + \alpha\mu\kappa^{-1/2}\norm[0]{ (\bar{u}_h^{s,1})^t }_{\Gamma^I}^2
    \le \frac{1}{2\psi} \norm[0]{f^{s,1}}_{\Omega^s}^2
    + \frac{c_{sdi}^2\mu^2}{2\kappa^2\psi} \norm[0]{u_h^{d,1}}_{\Omega^d}^2 + \psi\norm[0]{u_h^{s,1}}_{\Omega^s}^2.      
  \end{equation}
  Choosing $\psi = 1/\Delta t$ and reordering,
  \begin{equation}
    \label{eq:simpuhs1dtcoerc-young3}
    \frac{1}{\Delta t}\tnorm{\boldsymbol{u}_h^{s,1}}_{v,s}^2
    \le \frac{1}{2\mu c_{ae}^s} \norm[0]{f^{s,1}}_{\Omega^s}^2
    + \frac{c_{sdi}^2\mu}{2\kappa^2 c_{ae}^s} \norm[0]{u_h^{d,1}}_{\Omega^d}^2.
  \end{equation}
  To further bound
  \cref{eq:simpuhs1dtcoerc-imp,eq:simpuhs1dtcoerc-young3} we require a
  bound on $\norm[0]{u_h^{d,1}}_{\Omega^d}^2$. To obtain this bound,
  we set $n=0$ and choose
  $(\boldsymbol{v}_h, \boldsymbol{q}_h) = (\boldsymbol{u}_h^{1},
  -\boldsymbol{p}_h^{1})$ in \cref{eq:hdg-fully} and recall that
  $u_h^{s,0}=0$ to find
  \begin{equation*}
    \del[0]{ d_tu_h^{s,1}, u_h^{s,1}}_{\Omega^s}
    + a_h^L(\boldsymbol{u}_h^{1}, \boldsymbol{u}_h^{1})
    = (f^{s,1}, u_h^{s,1})_{\Omega^s} - (f^{d,1}, p_h^{d,1})_{\Omega^d}.
  \end{equation*}  
  Using that
  $\del[0]{ d_tu_h^{s,1}, u_h^{s,1}}_{\Omega^s} = \Delta
  t^{-1}\norm[0]{u_h^{s,1}}_{\Omega^s}^2$, the Cauchy--Schwarz
  inequality, \cref{eq:ahcoercivity,eq:phdnvsuhdn}, and Young's
  inequality,
  \begin{equation}
    \label{eq:snp1minuhsn-1}
    \begin{split}
    \tfrac{1}{\Delta t}\norm[0]{u_h^{s,1}}_{\Omega^s}^2
    + c_{ae}^s\mu \tnorm{\boldsymbol{u}_h^{s,1}}_{v,s}^2 + \mu\kappa^{-1}\norm[0]{u_h^{d,1}}_{\Omega^d}^2
    &\le \norm[0]{f^{s,1}}_{\Omega^s}\norm[0]{u_h^{s,1}}_{\Omega^s}
    + \norm[0]{f^{d,1}}_{\Omega^d} c_{td} \mu \kappa^{-1}\norm[0]{u_h^{d,1}}_{\Omega^d}
    \\
    &\le \frac{1}{2\psi}\norm[0]{f^{s,1}}_{\Omega^s}^2 + \frac{\psi}{2}\norm[0]{u_h^{s,1}}_{\Omega^s}^2
    +\frac{c_{td}^2\mu^2}{2\kappa^2\phi}\norm[0]{f^{d,1}}_{\Omega^d}^2 + \frac{\phi}{2}\norm[0]{u_h^{d,1}}_{\Omega^d}^2.
    \end{split}
  \end{equation}
  Choosing $\psi = 2/\Delta t$ and $\phi = \mu\kappa^{-1}$, we find
  from \cref{eq:snp1minuhsn-1}, after reordering, that
  \begin{equation}
    \label{eq:snp1minuhsnphdtouhd10afteryoung}
    \norm[0]{u_h^{d,1}}_{\Omega^d}^2
    \le \frac{\kappa\Delta t}{2\mu}\norm[0]{f^{s,1}}_{L^2(\Omega^s)^3}^2
    + c_{td}^2\norm[0]{f^{d,1}}_{L^2(\Omega^d)^3}^2.
  \end{equation}
  \Cref{eq:dtuhs1} follows from
  \cref{eq:snp1minuhsnphdtouhd10afteryoung} and
  \cref{eq:simpuhs1dtcoerc-imp}. \Cref{eq:tnormuhs1dt} follows from
  \cref{eq:snp1minuhsnphdtouhd10afteryoung} and
  \cref{eq:simpuhs1dtcoerc-young3}.
  \\ \\
  We proceed with proving \cref{eq:dtuhsnp1bound}. Let
  $1\le n \le N-1$. Consider \cref{eq:hdg-fully} at time levels $n+1$
  and $n$:
  \begin{align*}
    \del[0]{ d_tu_h^{s,n+1}, v_h^s}_{\Omega^s}
    + a_h(u_h^n; \boldsymbol{u}_h^{n+1}, \boldsymbol{v}_h)
    + b_h(\boldsymbol{v}_h, \boldsymbol{p}_h^{n+1})
    + b_h(\boldsymbol{u}_h^{n+1}, \boldsymbol{q}_h)
    &= (f^{s,n+1}, v_h)_{\Omega^s} + (f^{d,n+1}, q_h)_{\Omega^d},
    \\
    \del[0]{ d_tu_h^{s,n}, v_h^s}_{\Omega^s}
    + a_h(u_h^{n-1}; \boldsymbol{u}_h^{n}, \boldsymbol{v}_h)
    + b_h(\boldsymbol{v}_h, \boldsymbol{p}_h^{n})
    + b_h(\boldsymbol{u}_h^{n}, \boldsymbol{q}_h)
    &= (f^{s,n}, v_h)_{\Omega^s} + (f^{d,n}, q_h)_{\Omega^d}.
  \end{align*}
  Subtracting the latter from the former, choosing
  $(\boldsymbol{v}_h, \boldsymbol{q}_h) = (\delta
  \boldsymbol{u}_h^{n+1}, - \delta \boldsymbol{p}_h^{n+1})$, and
  noting that
  \begin{multline*}
    t_h(u_h^{s,n}; \boldsymbol{u}_h^{n+1}, \delta \boldsymbol{u}_h^{n+1})
    - t_h(u_h^{s,n-1}; \boldsymbol{u}_h^{n}, \delta \boldsymbol{u}_h^{n+1})
    \\
    =  t_h(u_h^{s,n}; \delta \boldsymbol{u}_h^{n+1}, \delta \boldsymbol{u}_h^{n+1})
       + t_h(u_h^{s,n}; \boldsymbol{u}_h^{n}, \delta \boldsymbol{u}_h^{n+1})
       - t_h(u_h^{s,n-1}; \boldsymbol{u}_h^{n}, \delta \boldsymbol{u}_h^{n+1}),
  \end{multline*}
  we find
  \begin{multline}
    \label{eq:deltauhsnp1s}
    \frac{1}{\Delta t}\del[0]{ \delta u_h^{s,n+1}- \delta u_h^{s,n}, \delta u_h^{s,n+1}}_{\Omega^s}
    + t_h(u_h^{s,n}; \boldsymbol{u}_h^{n}, \delta \boldsymbol{u}_h^{n+1})
    - t_h(u_h^{s,n-1}; \boldsymbol{u}_h^{n}, \delta \boldsymbol{u}_h^{n+1})    
    \\
    + a_h(u_h^{s,n}; \delta \boldsymbol{u}_h^{n+1}, \delta \boldsymbol{u}_h^{n+1})    
    = (\delta f^{s,n+1}, \delta u_h^{s,n+1})_{\Omega^s} + (\delta f^{d,n+1}, \delta p_h^{d,n+1})_{\Omega^d}.          
  \end{multline}
  \Cref{eq:dtrpoincareineq} and
  $\boldsymbol{u}_h^{s,k} \in \boldsymbol{B}_h^s$ imply
  $\norm[0]{u_h^k \cdot n}_{\Gamma^I} \le \tfrac{1}{2} \mu
  c_{ae}^s(c_{pq}^2 + c_{si,4}^2)^{-1}$ for $0 \le k \le n$. Therefore,
  coercivity of $a_h$ \cref{eq:coercivity_awhvhvh} holds. Also using
  the Cauchy--Schwarz inequality and \cref{eq:dpoincareineq}:
  \begin{multline}
    \label{eq:deltauhsnp1sc}
    \frac{1}{\Delta t}\del[0]{ \delta u_h^{s,n+1}- \delta u_h^{s,n}, \delta u_h^{s,n+1}}_{\Omega^s}
    + t_h(u_h^{s,n}; \boldsymbol{u}_h^{n}, \delta \boldsymbol{u}_h^{n+1})
    - t_h(u_h^{s,n-1}; \boldsymbol{u}_h^{n}, \delta \boldsymbol{u}_h^{n+1})
    \\
    + c_{ae}\mu \tnorm{\delta \boldsymbol{u}_h^{n+1}}_v^2
    \le c_p \norm[0]{\delta f^{s,n+1}}_{\Omega^s} \tnorm{\delta \boldsymbol{u}_h^{s,n+1}}_{v,s}
    + \norm[0]{\delta f^{d,n+1}}_{\Omega^d} \norm[0]{\delta p_h^{d,n+1}}_{\Omega^d}.          
  \end{multline}
  A simple modification of the proof of Lemma \ref{lem:tnormphduhd}
  allows us to show that
  $\norm[0]{\delta p_h^{d,n+1}}_{1,h,\Omega^d} \le C \mu\kappa^{-1}
  \norm[0]{\delta u_h^{d,n+1}}_{\Omega^d}$. Then, following the same
  steps used to find \cref{eq:phdnvsuhdn},
  $\norm[0]{\delta p_h^{d,n+1}}_{\Omega^d} \le c_{td} \mu
  \kappa^{-1}\norm[0]{\delta u_h^{d,n+1}}_{\Omega^d}$ so that
  \begin{equation}
    \begin{split}
      \label{eq:deltauhsnp1scf}
      \frac{1}{\Delta t}( \delta & u_h^{s,n+1} - \delta u_h^{s,n}, \delta u_h^{s,n+1})_{\Omega^s}
      + c_{ae}\mu \tnorm{\delta \boldsymbol{u}_h^{n+1}}_v^2
      \\
      \le & c_p \norm[0]{\delta f^{s,n+1}}_{\Omega^s} \tnorm{\delta \boldsymbol{u}_h^{s,n+1}}_{v,s}
      + c_{td}\mu\kappa^{-1} \norm[0]{\delta f^{d,n+1}}_{\Omega^d} \norm[0]{\delta u_h^{d,n+1}}_{\Omega^d}
      \\
      & + |t_h(u_h^{s,n}; \boldsymbol{u}_h^{s,n}, \delta \boldsymbol{u}_h^{s,n+1})
      - t_h(u_h^{s,n-1}; \boldsymbol{u}_h^{s,n}, \delta \boldsymbol{u}_h^{s,n+1})|.
    \end{split}
  \end{equation}
  To bound the convective terms we use \cref{eq:boundedness_th},
  \cref{eq:dpoincareineq}, and Young's inequality:
  \begin{align*}
    |t_h(u_h^{s,n}; \boldsymbol{u}_h^{n}, \delta \boldsymbol{u}_h^{n+1})
    - t_h(u_h^{s,n-1}; \boldsymbol{u}_h^{n}, \delta \boldsymbol{u}_h^{n+1})|
    &\le c_w \tnorm{\delta \boldsymbol{u}_h^{s,n}}_{v,s} \tnorm{\boldsymbol{u}_h^{s,n}}_{v,s} \tnorm{\delta \boldsymbol{u}_h^{s,n+1}}_{v,s}
    \\
    &\le \frac{c_w}{2}\tnorm{\boldsymbol{u}_h^{s,n}}_{v,s}\tnorm{\delta \boldsymbol{u}_h^{s,n}}_{v,s}^2
     + \frac{c_w}{2}\tnorm{\boldsymbol{u}_h^{s,n}}_{v,s} \tnorm{\delta \boldsymbol{u}_h^{s,n+1}}_{v,s}^2.
  \end{align*}
  Applying Young's inequality to the first two terms on the right hand
  side of \cref{eq:deltauhsnp1scf},
  \begin{align*}
    c_p& \norm[0]{\delta f^{s,n+1}}_{\Omega^s} \tnorm{\delta \boldsymbol{u}_h^{s,n+1}}_{v,s} +  c_{td}\mu\kappa^{-1} \norm[0]{\delta f^{d,n+1}}_{\Omega^d} \norm[0]{\delta u_h^{d,n+1}}_{\Omega^d}
    \\
    \le& \frac{c_p^2}{2\phi}\norm[0]{\delta f^{s,n+1}}_{\Omega^s}^2 +  \frac{c_{td}^2\mu^2}{2\kappa^2 \phi}\norm[0]{\delta f^{d,n+1}}_{\Omega^d}^2 
    +\frac{\phi}{2} \Big(\tnorm{\delta \boldsymbol{u}_h^{s,n+1}}_{v,s}^2+ \tnorm{\delta \boldsymbol{u}_h^{d,n+1}}_{v,d}^2\Big),
  \end{align*}
  and choosing $\phi = c_{ae}\mu$, we find after combining with
  \cref{eq:deltauhsnp1scf} that
  \begin{multline}
    \label{eq:deltauhsnp1scfby}
    \frac{1}{\Delta t}\del[0]{ \delta u_h^{s,n+1}- \delta u_h^{s,n}, \delta u_h^{s,n+1}}_{\Omega^s}
    + \del[1]{\tfrac{1}{2}c_{ae}\mu - \tfrac{c_w}{2}\tnorm{\boldsymbol{u}_h^{s,n}}_{v,s}} \tnorm{\delta \boldsymbol{u}_h^{n+1}}_v^2
    \\
    \le \frac{c_p^2}{2c_{ae}\mu}\norm[0]{\delta f^{s,n+1}}_{\Omega^s}^2
    + \frac{c_{td}^2\mu}{2\kappa^2 c_{ae}}\norm[0]{\delta f^{d,n+1}}_{\Omega^d}^2
    + \frac{c_w}{2} \tnorm{\boldsymbol{u}_h^{s,n}}_{v,s} \tnorm{\delta \boldsymbol{u}_h^{s,n}}_{v,s}^2.    
  \end{multline}
  Multiplying both sides by 2, using the assumption that
  $\boldsymbol{u}_h^{s,n} \in \boldsymbol{B}_h^s$, that
  $a(a-b) \ge \tfrac{1}{2}(a^2 - b^2)$, and that
  $\tnorm{\delta \boldsymbol{u}_h^{n+1}}_{v,s} \le \tnorm{\delta
    \boldsymbol{u}_h^{n+1}}_v$,
  \begin{multline}
    \label{eq:deltauhsnp1scfbytphi32}
    \frac{1}{\Delta t} \norm[0]{\delta u_h^{s,n+1}}_{\Omega^s}^2 - \frac{1}{\Delta t} \norm[0]{\delta u_h^{s,n}}_{\Omega^s}^2
    + \tfrac{1}{2}c_{ae}\mu \tnorm{\delta \boldsymbol{u}_h^{n+1}}_{v,s}^2
    \\
    \le \frac{c_p^2}{c_{ae}\mu}\norm[0]{\delta f^{s,n+1}}_{\Omega^s}^2
    + \frac{c_{td}^2\mu}{\kappa^2 c_{ae}}\norm[0]{\delta f^{d,n+1}}_{\Omega^d}^2
    + \tfrac{1}{2}c_{ae}\mu \tnorm{\delta \boldsymbol{u}_h^{s,n}}_{v,s}^2.    
  \end{multline}
  Replacing $n$ by $k$, summing for $k=1$ to $k=n$, using that
  $d_tu_h^{s,n+1} = \Delta t^{-1}\delta u_h^{s,n+1}$ and that
  $\delta \boldsymbol{u}_h^{s,1} = \boldsymbol{u}_h^{s,1}$ (because
  $\boldsymbol{u}_h^{s,0}=\boldsymbol{0}$), and the definition of
  $F^n$ (see \cref{eq:Fm}):
  \begin{equation}
    \norm[0]{d_t u_h^{s,n+1}}_{\Omega^s}^2
    \le
    \norm[0]{d_t u_h^{s,1}}_{\Omega^s}^2
    + \frac{c_{ae}\mu}{2\Delta t} \tnorm{\boldsymbol{u}_h^{s,1}}_{v,s}^2
    + F^n.
  \end{equation}
  \Cref{eq:dtuhsnp1bound} now follows by inserting
  \cref{eq:dtuhs1,eq:tnormuhs1dt} into the above inequality.
\end{proof}

We are now ready to prove a bound on $\boldsymbol{u}_h^{n+1}$.

\begin{lemma}
  \label{lem:uhskbound}
  Let $\boldsymbol{u}_h^{s,0}=\boldsymbol{0}$, and let $M^n$ be
  defined as in \cref{eq:M02,eq:Mm} for $0 \le n \le N-1$. If
  \cref{eq:hdg-fully} has a solution
  $(\boldsymbol{u}_h^k,\boldsymbol{p}_h^k)$ for all $0 \le k \le n$
  such that $\boldsymbol{u}_h^{s,k} \in \boldsymbol{B}_h^s$, then
  \begin{equation}
    \label{eq:uhnp1vtnormbound}
    \tnorm{\boldsymbol{u}_h^{n+1}}_{v}^2
    \le
    \frac{2}{c_{ae}\mu} \del[2]{
    \frac{c_p^2}{c_{ae}\mu}(M^n)^2
    + \frac{c_p^2}{c_{ae}\mu}\norm[0]{f^{s}}_{L^{\infty}(J;L^2(\Omega^s))}^2
    + \frac{c_{td}^2\mu}{2\kappa^2c_{ae}}\norm[0]{f^{d}}_{L^{\infty}(J;L^2(\Omega^d))}^2}.
  \end{equation}  
\end{lemma}
\begin{proof}
  Choose
  $(\boldsymbol{v}_h, \boldsymbol{q}_h) = (\boldsymbol{u}_h^{n+1},
  -\boldsymbol{p}_h^{n+1})$ in \cref{eq:hdg-fully}. Coercivity of
  $a_h$ \cref{eq:coercivity_awhvhvh} (which holds by
  \cref{eq:dtrpoincareineq} and the assumption that
  $\boldsymbol{u}_h^{s,n} \in \boldsymbol{B}_h^s$) then implies:
  \begin{equation}
    \label{eq:choosevqupnp1}
    \del[0]{ d_tu_h^{s,n+1}, u_h^{s,n+1}}_{\Omega^s}
    + c_{ae}\mu \tnorm{\boldsymbol{u}_h^{n+1}}_v^2
    \le (f^{s,n+1}, u_h^{s,n+1})_{\Omega^s} - (f^{d,n+1}, p_h^{d,n+1})_{\Omega^d}.
  \end{equation}
  Using the Cauchy--Schwarz inequality,
  \cref{eq:phdnvsuhdn,eq:dpoincareineq}, Young's inequality and
  \cref{eq:dtuhsnp1bound}, we obtain:
  \begin{equation*}
    \begin{split}
      c_{ae} \mu \tnorm{\boldsymbol{u}_h^{n+1}}_{v}^2
      \le& | (f^{s,n+1}, u_h^{s,n+1})_{\Omega^s} - (f^{d,n+1}, p_h^{d,n+1})_{\Omega^d}
      - \del[0]{ d_tu_h^{s,n+1}, u_h^{s,n+1}}_{\Omega^s}|
      \\
      \le&
      c_p \norm[0]{f^{s,n+1}}_{\Omega^s}\tnorm{\boldsymbol{u}_h^{s,n+1}}_{v,s} + c_{td} \mu \kappa^{-1}\norm[0]{f^{d,n+1}}_{\Omega^d}\norm[0]{u_h^{d,n+1}}_{\Omega^d}
      +c_p\norm[0]{d_tu_h^{s,n+1}}_{\Omega^s}\tnorm{\boldsymbol{u}_h^{s,n+1}}_{v,s}
      \\
      \leq & \frac{c_p^2}{2\chi}(M^n)^2 + \frac{c_p^2}{2\chi}\norm[0]{f^{s,n+1}}_{\Omega^s}^2 + 
      \chi \tnorm{\boldsymbol{u}_h^{s,n+1}}_{v,s}^2
      +\frac{c_{td}^2\mu^2}{2\kappa^2\phi}\norm[0]{f^{d,n+1}}_{\Omega^d}^2 + \frac{\phi}{2}\norm[0]{u_h^{d,n+1}}_{\Omega^d}^2.
    \end{split}
  \end{equation*}  
  The result follows by choosing $\chi = \tfrac{1}{2}c_{ae}\mu$, and
  $\phi = c_{ae}\mu$, and using the definition of $\tnorm{\cdot}_v$.
\end{proof}

We end this section by proving existence and uniqueness for all time
levels under a suitable data assumption.

\begin{lemma}  
  \label{lem:boundednessvelocitypressure}
  Let $M^n$ be defined as in \cref{eq:Mm}. Assume the data satisfy 
  for $1 \le n \le N-1$,
  \begin{multline}
    \label{eq:dataassumption}
    \frac{2}{c_{ae}\mu} \del[2]{
      \frac{c_p^2}{c_{ae}\mu}(M^n)^2
      + \frac{c_p^2}{c_{ae}\mu}\norm[0]{f^{s}}_{L^{\infty}(J;L^2(\Omega^s))}^2
      + \frac{c_{td}^2\mu}{2\kappa^2c_{ae}}\norm[0]{f^{d}}_{L^{\infty}(J;L^2(\Omega^d))}^2}      
    \\
    \le
    \sbr[2]{\min\del[2]{\frac{\mu c_{ae}^s}{2c_{si,2}(c_{pq}^2 + c_{si,4}^2)}, \frac{c_{ae}\mu}{2c_w}}}^2.
  \end{multline}
  Then, starting with $\boldsymbol{u}_h^{s,0} = \boldsymbol{0}$,
  \cref{eq:hdg-fully} has a unique solution. Furthermore, for
  $1 \le n \le N$, the velocity solution is such that
  $\boldsymbol{u}_h^{s,n} \in \boldsymbol{B}_h^s$ and the pressure solution
  satisfies,
  \begin{equation}
    \label{eq:pressurebounduniformnh}
    \tnorm{\boldsymbol{p}_h^{n}}_p^2
    \le
    \del[1]{\tfrac{1}{2}c_{ae}^2
    + c_{ac}^2}\frac{3\mu^2}{c_{bb}^2} \sbr[2]{\min\del[2]{\frac{\mu c_{ae}^s}{2c_{si,2}(c_{pq}^2 + c_{si,4}^2)}, \frac{c_{ae}\mu}{2c_w}}}^2.
  \end{equation}  
\end{lemma}
\begin{proof}
  Existence and uniqueness of
  $(\boldsymbol{u}_h^{n+1},\boldsymbol{p}_h^{n+1})$ under the
  assumption that $\boldsymbol{u}_h^{s,n} \in \boldsymbol{B}_h^s$ for
  $0 \le n \le N$ is established by Lemma
  \ref{lemap:exisuniqnp1}. That
  $\boldsymbol{u}_h^{s,n} \in \boldsymbol{B}_h^s$ for $1 \le n \le N$
  is due to \cref{eq:uhnp1vtnormbound,eq:dataassumption}.

  We now prove the pressure bound \cref{eq:pressurebounduniformnh}. By
  the inf-sup condition \cref{eq:infsupbh} and the HDG method
  \cref{eq:hdg-fully}, with $\boldsymbol{q}_h=\boldsymbol{0}$, we find
  for $0\le n \le N-1$:
  \begin{equation*}
    c_{bb} \tnorm{\boldsymbol{p}_h^{n+1}}_p
    \le
    \sup_{0 \ne \boldsymbol{v}_h \in \boldsymbol{X}_h}
    \frac{|b_h(\boldsymbol{v}_h, \boldsymbol{p}_h^{n+1})|}{\tnorm{\boldsymbol{v}_h}_{v}}
    =
    \sup_{0 \ne \boldsymbol{v}_h \in \boldsymbol{X}_h}
    \frac{|
      (f^{s,n+1}, v_h)_{\Omega^s}
      -\del[0]{ d_t u_h^{n+1}, v_h}_{\Omega^s}
      - a_h(u_h^n; \boldsymbol{u}_h^{n+1}, \boldsymbol{v}_h)|
    }{\tnorm{\boldsymbol{v}_h}_{v}}.      
  \end{equation*}
  By the Cauchy--Schwarz inequality,
  \cref{eq:dpoincareineq,eq:ahboundedXh}, squaring and using
  H\"older's inequality for sums,
  \begin{equation*}
   \tnorm{\boldsymbol{p}_h^{n+1}}_p^2
    \le 3\big(c_p^2c_{bb}^{-2} \norm[0]{f^{s,n+1}}_{\Omega^s}^2
     + c_p^2c_{bb}^{-2} \norm[0]{d_tu_h^{n+1}}_{\Omega^s}^2
      + c_{ac}^2c_{bb}^{-2} \mu^2 \tnorm{\boldsymbol{u}_h^{n+1}}_{v}^2\big).
  \end{equation*}
  A bound for $\tnorm{\boldsymbol{u}_h^{n+1}}_v$ is given by Lemma
  \ref{lem:uhskbound} and the data assumption
  \cref{eq:dataassumption}. Together with \cref{eq:dtuhsnp1bound} we
  obtain
  \begin{equation}
    \label{eq:pnormboundhnp1a}
    \tnorm{\boldsymbol{p}_h^{n+1}}_p^2
    \le  3\big(c_p^2c_{bb}^{-2} \del[1]{\norm[0]{f^s}_{L^{\infty}(J;L^2(\Omega^s))}^2 + (M^n)^2}
    + c_{ac}^2c_{bb}^{-2} \mu^2 \sbr[2]{\min\del[2]{\frac{\mu c_{ae}^s}{2c_{si,2}(c_{pq}^2 + c_{si,4}^2)}, \frac{c_{ae}\mu}{2c_w}}}^2\big).
  \end{equation}
  Note that the data assumption \cref{eq:dataassumption} implies that
  \begin{equation}
    \label{eq:datapassump}
    c_p^2c_{bb}^{-2}\del[2]{
      \norm[0]{f^{s}}_{L^{\infty}(J;L^2(\Omega^s))}^2 + (M^n)^2
      }      
    \le \tfrac{1}{2}c_{ae}^2c_{bb}^{-2}\mu^2
    \sbr[2]{\min\del[2]{\frac{\mu c_{ae}^s}{2c_{si,2}(c_{pq}^2 + c_{si,4}^2)}, \frac{c_{ae}\mu}{2c_w}}}^2.    
  \end{equation}
  The result follows from \cref{eq:pnormboundhnp1a,eq:datapassump}.
\end{proof}

\section{A priori error estimates}
\label{sec:apriori}
Let $\Pi_Q$ be the $L^2$-projection into $Q_h$ and let $\bar{\Pi}_V$
and $\bar{\Pi}_Q^j$, $j=s,d$, be the $L^2$-projections into the facet
spaces $\bar{V}_h$ and $\bar{Q}_h^j$, $j=s,d$, respectively. Let
$\Pi_V: H({\rm div};\Omega) \cap \sbr[0]{L^r(\Omega)}^{\dim}
\rightarrow X_h\cap H({\rm div};\Omega)$, where $r>2$, be an interpolant such that
\begin{align}
  \label{eq:interpolant_a}
  (q_h, \nabla\cdot \Pi_V u)_K 
  &= (q_h, \nabla\cdot u)_K
  && \forall q_h\in P_{k-1}(K),
  \\
  \label{eq:interpolant_b}
  \langle \bar{q}_h, n\cdot \Pi_Vu\rangle_{F}
  &=\langle \bar{q}_h, n\cdot u \rangle_{F}
  && \forall \bar{q}_h\in P_{k}(F),\quad \forall \text{ faces } F \text{ of } K,
\end{align}
and with the properties that for any
$u\in \sbr[0]{H^{k+1}(K)}^{\dim}$,
\begin{equation}
  \label{eq:PiVuumKestimate}
  \norm[0]{u-\Pi_V u}_{m,K} \leq C h_K^{\ell-m}\norm[0]{u}_{\ell,K} 
  \quad m=0,1,2, \quad \max(1,m)\leq \ell\leq k+1,   
\end{equation}
and for any $u \in \sbr[0]{W^1_{\infty}(K)}^{\dim}$,
\begin{equation}
\label{eq:linfnormPiV}
  \norm[0]{u-\Pi_V u}_{L^{\infty}(K)} \leq Ch_K |u|_{W^1_{\infty}(K)}.
\end{equation}
Examples of such operators are the Brezzi--Douglas--Marini (BDM) and Raviart--Thomas (RT)
interpolation operators \cite{Boffi:book}.

We partition the errors into their interpolation and approximation
parts as $\zeta - \zeta_h = e_{\zeta}^{I} - e_{\zeta}^h$ for
$\zeta = u, \bar{u}, p, \bar{p}^j$ and for $j=s,d$, where
\begin{align*}
  e_u^I &= u- \Pi_V u,  &e_u^h &= u_h- \Pi_V u,  &e_p^I &= p- \Pi_Q p,  &e_p^h &= p_h- \Pi_Q p,\\
  \bar{e}_{u}^I &= \gamma(u)-\bar{\Pi}_V u,  &\bar{e}_{u}^h &= \bar{u}_h-\bar{\Pi}_V u,
  &\bar{e}_{p^j}^I &= \gamma(p^j)-\bar{\Pi}_Q^j p, &\bar{e}_{p^j}^h &= \bar{p}^j_h-\bar{\Pi}_Q^j p.
\end{align*}
To be consistent with the notation for elements in
$\boldsymbol{X}_h, \boldsymbol{Q}_h, \boldsymbol{Q}_h^j$, $j=s,d$, we
also define
\begin{equation*}
  \boldsymbol{e}_u^{\zeta} = (e_u^{\zeta}, \bar{e}_u^{\zeta}), \quad
  \boldsymbol{e}_p^{\zeta} = (e_p^{\zeta}, \bar{e}_{p^s}^{\zeta}, \bar{e}_{p^d}^{\zeta}),
  \quad \boldsymbol{e}_{p^j}^{\zeta} = (e_{p^j}^{\zeta}, \bar{e}_{p^j}^{\zeta}),
  \quad \zeta = I,h.
\end{equation*}
In the following we will use that the initial condition is given by
$u_h^{s,0} = \Pi_V u_0$ and so $e_u^{h,0} = 0$.

To determine the error equation we first note that by Lemma
\ref{lem:consistency}, the exact solution
$(\boldsymbol{u},\boldsymbol{p})$ satisfies
\cref{eq:hdg-semi}. Therefore, subtracting \cref{eq:hdg-semi} at time
level $t=t^{n+1}$, with $(\boldsymbol{u}_h,\boldsymbol{p}_h)$ replaced
by $(\boldsymbol{u},\boldsymbol{p})$, from \cref{eq:hdg-fully},
splitting the errors into their interpolation and approximation parts,
using that $b_h(\boldsymbol{v}_h, \boldsymbol{e}_p^{I,n+1})=0$ for all
$\boldsymbol{v}_h \in \boldsymbol{X}_h$ (since $\Pi_Q$, $\bar{\Pi}_Q$
are $L^2$-projections onto $Q_h$ and $\bar{Q}_h$, respectively, and
$\nabla \cdot V_h=Q_h$) and that
$b_h(\boldsymbol{e}_u^{I,n+1}, \boldsymbol{q}_h)=0$ for all
$\boldsymbol{q}_h \in \boldsymbol{Q}_h$ (by
\cref{eq:interpolant_a,eq:interpolant_b} and properties of the
$L^2$-projection $\bar{\Pi}_V$) we obtain:
\begin{multline}
  \label{eq:erreqn-2}
  (d_t e_u^{h,n+1}, v_h)_{\Omega^s}
  + t_h(u_h^n; \boldsymbol{u}_h^{n+1}, \boldsymbol{v}_h) - t_h(u^{n+1}; \boldsymbol{u}^{n+1}, \boldsymbol{v}_h)
  + a_h^L(\boldsymbol{e}_u^{h,n+1}, \boldsymbol{v}_h)
  + b_h(\boldsymbol{v}_h, \boldsymbol{e}_p^{h,n+1})
  + b_h(\boldsymbol{e}_u^{h,n+1}, \boldsymbol{q}_h)
  \\
  = \del[0]{ d_t e_u^{I,n+1}, v_h}_{\Omega^s} + \del[0]{ \partial_t u^{n+1}-d_t  u^{n+1}, v_h}_{\Omega^s}
  + a_h^L(\boldsymbol{e}_u^{I,n+1}, \boldsymbol{v}_h).
\end{multline}
The following theorem now determines an upper bound for the
approximation error $\boldsymbol{e}_u^{h,n}$.

\begin{theorem}
  \label{thm:velocityestimate}
  Suppose that $u\in L^{\infty}(0,T;\sbr[0]{H^{k+1}(\Omega)}^{\dim})$
  such that $u^s\in L^2(0,T;\sbr[0]{W^1_3(\Omega^s)}^{\dim})$, \\
  $\partial_t u\in L^2(0,T;\sbr[0]{H^{k}(\Omega^s)}^{\dim})$, and
  $\partial_{tt}u\in L^2(0,T;\sbr[0]{L^2(\Omega^s)}^{\dim})$. Suppose
  also that the data satisfies the assumptions of Lemma
  \ref{lem:boundednessvelocitypressure}. Then, for $1 \le m \le N$,
  \begin{equation}
    \label{eq:velocityerrorestimate}
    \begin{split}
      &\norm[0]{e_u^{h,m}}_{\Omega^s}^2
      + \Delta t^2 \sum_{n=0}^{m-1}\norm[0]{d_te_u^{h,n+1}}_{\Omega^s}^2
      + c_{ae}\mu\Delta t\sum_{n=0}^{m-1}\tnorm{\boldsymbol{e}_u^{h,n+1}}_v^2
      \\
      \le& CC_G\big[
      h^{2k} \cbr[1]{ \mu^{-1}\norm[0]{\partial_t u}^2_{L^2(J;H^k(\Omega^s))}
        +  T\del[1]{(\mu 
          + \mu^{-1}) \norm[0]{u}_{L^{\infty}(J;H^{k+1}(\Omega^s))}^2 } \norm[0]{u}_{L^{\infty}(J;H^{k+1}(\Omega))}^2 }
      \\
      & \hspace{2em} + (\Delta t)^2 \mu^{-1} \cbr[1]{ \norm[0]{\partial_{tt}u}_{L^2(J;L^2(\Omega^s))}^2
        + \norm[0]{\partial_tu}^2_{L^2(J;L^2(\Omega^s))} \norm[0]{u}_{L^{\infty}(J;H^1(\Omega^s))}^2 } \big],
    \end{split}
  \end{equation}    
  where $C_G = \exp(\Delta t \sum_{n=0}^{m-1} C\mu^{-1}
  \norm[0]{u^{n+1}}^2_{W^1_3(\Omega^s)})$. 
\end{theorem}
\begin{proof}
  Consider the convective terms in \cref{eq:erreqn-2}. We note that
  \begin{equation*}
    \begin{split}
      t_h(u_h^n; \boldsymbol{u}_h^{n+1}, \boldsymbol{v}_h) - t_h(u^{n+1}; \boldsymbol{u}^{n+1}, \boldsymbol{v}_h)
      =& t_h(u_h^n; \boldsymbol{e}_u^{h,n+1}, \boldsymbol{v}_h) - t_h(u^{n+1}; \boldsymbol{e}_u^{I,n+1}, \boldsymbol{v}_h)
      \\
      &+ [t_h(u_h^n; \boldsymbol{\Pi}_V u^{n+1}, \boldsymbol{v}_h) - t_h(u^n; \boldsymbol{\Pi}_V u^{n+1}, \boldsymbol{v}_h)]
      \\
      &+ [t_h(u^n; \boldsymbol{\Pi}_V u^{n+1}, \boldsymbol{v}_h)-t_h(u^{n+1}; \boldsymbol{\Pi}_V u^{n+1}, \boldsymbol{v}_h)].
    \end{split}
  \end{equation*}
  We furthermore note that
  \begin{equation*}
    \begin{split}
      a_h(u_h^n; \boldsymbol{e}_u^{h,n+1}, \boldsymbol{v}_h)
      &= t_h(u_h^n; \boldsymbol{e}_u^{h,n+1}, \boldsymbol{v}_h) + a_h^L(\boldsymbol{e}_u^{h,n+1}, \boldsymbol{v}_h),
      \\
      a_h(u^{n+1}; \boldsymbol{e}_u^{I,n+1}, \boldsymbol{v}_h)
      &= t_h(u^{n+1}; \boldsymbol{e}_u^{I,n+1}, \boldsymbol{v}_h) + a_h^L(\boldsymbol{e}_u^{I,n+1}, \boldsymbol{v}_h),
    \end{split}
  \end{equation*}
  so that we can write \cref{eq:erreqn-2} as
  \begin{equation}
    \label{eq:erreqn-2news}
    \begin{split}
      (d_t e_u^{h,n+1}&, v_h)_{\Omega^s}
      + a_h(u_h^n; \boldsymbol{e}_u^{h,n+1}, \boldsymbol{v}_h)
      + b_h(\boldsymbol{v}_h, \boldsymbol{e}_p^{h,n+1})
      + b_h(\boldsymbol{e}_u^{h,n+1}, \boldsymbol{q}_h)
      \\
      =& \del[0]{ d_t e_u^{I,n+1}, v_h}_{\Omega^s} + \del[0]{ \partial_t u^{n+1}-d_t  u^{n+1}, v_h}_{\Omega^s}
      + a_h(u^{n+1}; \boldsymbol{e}_u^{I,n+1}, \boldsymbol{v}_h)
      \\
      &+ [t_h(u^n; \boldsymbol{\Pi}_V u^{n+1}, \boldsymbol{v}_h) - t_h(u_h^n; \boldsymbol{\Pi}_V u^{n+1}, \boldsymbol{v}_h)]
      + [t_h(u^{n+1}; \boldsymbol{\Pi}_V u^{n+1}, \boldsymbol{v}_h) - t_h(u^n; \boldsymbol{\Pi}_V u^{n+1}, \boldsymbol{v}_h)].
    \end{split}
  \end{equation}
  Let us now choose
  $(\boldsymbol{v}_h, \boldsymbol{q}_h)=(\boldsymbol{e}_u^{h,n+1},
  -\boldsymbol{e}_p^{h,n+1})$ in \cref{eq:erreqn-2news}. By the
  assumption on the data we have coercivity of $a_h$
  \cref{eq:coercivity_awhvhvh} so that:
  \begin{equation}
    \label{eq:dterroreq}
    \begin{split}
      (d_t e_u^{h,n+1}, e_u^{h,n+1})_{\Omega^s}
      + c_{ae}\mu\tnorm{\boldsymbol{e}_u^{h,n+1}}_v^2      
      \le& \del[0]{ d_t e_u^{I,n+1}, e_u^{h,n+1}}_{\Omega^s} + \del[0]{ \partial_t u^{n+1}-d_t  u^{n+1}, e_u^{h,n+1}}_{\Omega^s}
      \\
      & + a_h(u^{n+1}; \boldsymbol{e}_u^{I,n+1}, \boldsymbol{e}_u^{h,n+1})
      \\
      &+ [t_h(u^n; \boldsymbol{\Pi}_V u^{n+1}, \boldsymbol{e}_u^{h,n+1}) - t_h(u_h^n; \boldsymbol{\Pi}_V u^{n+1}, \boldsymbol{e}_u^{h,n+1})]
      \\
      &+ [t_h(u^{n+1}; \boldsymbol{\Pi}_V u^{n+1}, \boldsymbol{e}_u^{h,n+1}) - t_h(u^n; \boldsymbol{\Pi}_V u^{n+1}, \boldsymbol{e}_u^{h,n+1})]
      \\      
      =&:\sum_{j=1}^{5} I_j.
    \end{split}
  \end{equation}
  Using \cref{eq:dtfnp1estimate}, \cref{eq:PiVuumKestimate},
  \cref{eq:dpoincareineq}, and Young's inequality we find:
  \begin{equation}
    \label{eq:I1}
    \begin{split}
      I_1 &\leq \norm[0]{d_t e_u^{I,n+1}}_{\Omega^s}\norm[0]{e_u^{h,n+1}}_{\Omega^s}
      \leq Ch^k (\Delta t)^{-1/2}\norm[0]{\partial_t u}_{L^2(t^n,t^{n+1};H^k(\Omega^s))}\tnorm{\boldsymbol{e}_u^{h,n+1}}_{v,s}
      \\
      & \leq \gamma \tnorm{\boldsymbol{e}_u^{h,n+1}}_{v}^2
      + \dfrac{C}{\gamma}h^{2k} (\Delta t)^{-1}\norm[0]{\partial_t u}^2_{L^2(t^n,t^{n+1};H^k(\Omega^s))},
    \end{split}
  \end{equation}
  where $\gamma > 0$ will be chosen later. By \cref{eq:pdtgddtg},
  \cref{eq:dpoincareineq}, and Young's inequality,
  \begin{equation}
    \label{eq:I2}
    \begin{split}
      I_2 &\leq \norm[0]{\partial_t u^{n+1}-d_t u^{n+1}}_{\Omega^s} \norm[0]{e_u^{h,n+1}}_{\Omega^s} 
      \\
      & \leq C (\Delta t)^{1/2}\norm[0]{\partial_{tt}u}_{L^2(t^n,t^{n+1};L^2(\Omega^s))}\tnorm{\boldsymbol{e}_u^{h,n+1}}_{v,s}
      \\
      & \leq \gamma \tnorm{\boldsymbol{e}_u^{h,n+1}}_{v}^2
      + \frac{C}{\gamma}\Delta t\norm[0]{\partial_{tt}u}_{L^2(t^n,t^{n+1};L^2(\Omega^s))}^2.
    \end{split}
  \end{equation}
  Observe that by \cref{eq:ahboundedX}, \cref{eq:equivalencetnorm},
  \cite[Lemma 7]{Cesmelioglu:2023}, and Young's inequality,
  \begin{equation*}
    \begin{split}
      I_3
      & \leq c_{ac}\mu \tnorm{\boldsymbol{e}_u^{I,n+1}}_{v'}\tnorm{\boldsymbol{e}_u^{h,n+1}}_{v'}
      \le C\mu h^k\norm[0]{u^{n+1}}_{k+1,\Omega}\tnorm{\boldsymbol{e}_u^{h,n+1}}_{v}
      \\
      &\le \gamma \tnorm{\boldsymbol{e}_u^{h,n+1}}_{v}^2 + \frac{C}{\gamma}\mu^2 h^{2k}\norm[0]{u^{n+1}}_{k+1,\Omega}^2. 
    \end{split}
  \end{equation*}
  For $I_4$ we have
  \begin{equation}
  \label{eq:I4}
    I_4
    \le
    2\gamma \tnorm{\boldsymbol{e}_u^{h,n+1}}_{v}^2
    + \frac{C}{\gamma}h^{2k}\norm[0]{u^{n+1}}_{k+1,\Omega^s}^2\norm[0]{u^{n}}_{k+1,\Omega^s}^2
    + \frac{C}{\gamma}\norm[0]{e_u^{h,n}}_{\Omega^s}^2\norm[0]{u^{n+1}}_{W_3^1(\Omega^s)}^2,
  \end{equation}
  the proof of which, due to its length, is given in \cref{ap:I4}.
  
  By \cref{eq:boundedness_th}, \cref{eq:dtfnp1estimate}, properties of
  $\Pi_V$ and $\bar{\Pi}_V$ so that
  $\tnorm{\boldsymbol{\Pi}_V u^{n+1}}_{v,s} \le c
  \norm[0]{u^{n+1}}_{1,\Omega^s}$ (see
  \cite[Eq. (28)]{Rhebergen:2020}) and Young's inequality,
  \begin{equation*}
    \begin{split}
      I_{5} 
      &\le c_w \norm[0]{\nabla(u^{n+1}-u^n)}_{\Omega^s} \tnorm{\boldsymbol{\Pi}_V u^{n+1}}_{v,s} \tnorm{\boldsymbol{e}_u^{h,n+1}}_{v,s}
      \\
      &\le C (\Delta t)^{1/2} \norm[0]{\partial_t u}_{L^2(t^n,t^{n+1};H^1(\Omega^s))}
      \norm[0]{u^{n+1}}_{1,\Omega^s}\tnorm{\boldsymbol{e}_u^{h,n+1}}_{v,s}
      \\
      & \le \gamma \tnorm{\boldsymbol{e}_u^{h,n+1}}_{v}^2
      + \frac{C}{\gamma}\Delta t \norm[0]{\partial_t u}^2_{L^2(t^n,t^{n+1};L^2(\Omega^s))} \norm[0]{u^{n+1}}_{1,\Omega^s}^2.
    \end{split}
  \end{equation*}
  Collecting the above estimates for $I_1,\hdots,I_5$, combining with \cref{eq:dterroreq},
  using that $a(a-b) = \tfrac{1}{2}(a^2-b^2+(a-b)^2)$, choosing
  $\gamma = \tfrac{1}{12}c_{ae}\mu$, and multiplying by $2\Delta t$: 
  \begin{equation*}
    \begin{split}
      &\big(\norm[0]{e_u^{h,n+1}}_{\Omega^s}^2 - \norm[0]{e_u^{h,n}}_{\Omega^s}^2\big) + \norm[0]{e_u^{h,n+1}-e_u^{h,n}}_{\Omega^s}^2
      + c_{ae}\mu\Delta t\tnorm{\boldsymbol{e}_u^{h,n+1}}_v^2
      \\
      \le
      &
      C\big[h^{2k}\big\{\mu^{-1}\norm[0]{\partial_t u}^2_{L^2(t^n,t^{n+1};H^k(\Omega^s))}
      + \Delta t\mu\norm[0]{u^{n+1}}_{k+1,\Omega}^2
      + \Delta t\mu^{-1}\norm[0]{u^{n}}_{k+1,\Omega^s}^2\norm[0]{u^{n+1}}_{k+1,\Omega^s}^2 \big\}
      \\
      &\quad + (\Delta t)^2\mu^{-1}\big\{
      \norm[0]{\partial_{tt}u}_{L^2(t^n,t^{n+1};L^2(\Omega^s))}^2
      + \norm[0]{\partial_t u}^2_{L^2(t^n,t^{n+1};L^2(\Omega^s))} \norm[0]{u^{n+1}}_{1,\Omega^s}^2
      \big\}
      \\
      &\quad + \Delta t\mu^{-1}\norm[0]{u^{n+1}}_{W_3^1(\Omega^s)}^2\norm[0]{e_u^{h,n}}_{\Omega^s}^2
      \big].
    \end{split}
  \end{equation*}
  Summing from $n=0$ to $n=m-1$ and noting that $e_u^{h,0}=0$ gives
  \begin{equation*}
    \begin{split}
      &\norm[0]{e_u^{h,m}}_{\Omega^s}^2
      + \Delta t^2 \sum_{n=0}^{m-1}\norm[0]{d_te_u^{h,n+1}}_{\Omega^s}^2
      + c_{ae}\mu\Delta t\sum_{n=0}^{m-1}\tnorm{\boldsymbol{e}_u^{h,n+1}}_v^2
      \\
      \le &C
      h^{2k} \cbr[1]{ \mu^{-1} \norm[0]{\partial_t u}^2_{L^2(J;H^k(\Omega^s))}
        +  T\del[1]{(\mu 
          + \mu^{-1}) \norm[0]{u}_{L^{\infty}(J;H^{k+1}(\Omega^s))}^2 } \norm[0]{u}_{L^{\infty}(J;H^{k+1}(\Omega))}^2 }
      \\
      & + C\Delta t^2\mu^{-1} \cbr[1]{\norm[0]{\partial_{tt}u}_{L^2(J;L^2(\Omega^s))}^2
        + \norm[0]{\partial_tu}^2_{L^2(J;L^2(\Omega^s))} \norm[0]{u}_{L^{\infty}(J,H^1(\Omega^s))}^2 }
      \\
      & + C \Delta t \sum_{n=0}^{m-1} \mu^{-1} \norm[0]{u^{n+1}}^2_{W^1_3(\Omega^s)} \norm[0]{e_u^{h,n}}_{\Omega^s}^2.
    \end{split}
  \end{equation*}
  The result now follows by Gr\"{o}nwall's inequality \cite[Lemma
  28]{Layton:book} for all $\Delta t > 0$. 
\end{proof}

By a triangle inequality and properties of the interpolant $\Pi_V$ and
projection $\bar{\Pi}_V$, we obtain the following velocity error
estimate that is independent of the pressure.

\begin{corollary}
  \label{cor:velocityestimate}
  Suppose that $u$, $\boldsymbol{u}_h$, and the data satisfy the
  assumptions of Theorem \ref{thm:velocityestimate}. Then, for
  $1 \le m \le N$,
  \begin{equation*}
    \begin{split}
      \| u^m &- u_h^m \|_{\Omega^s}^2
      + c_{ae}\mu\Delta t \sum_{n=0}^{m-1}\tnorm{\boldsymbol{u}^{n+1} - \boldsymbol{u}_h^{n+1}}_v^2
      \\
      \le& C\big[
      h^{2k} \big\{ \mu^{-1}\norm[0]{\partial_t u}^2_{L^2(J;H^k(\Omega^s))}
      +  \del[1]{1+T(\mu + \mu^{-1}) \norm[0]{u}_{L^{\infty}(J;H^{k+1}(\Omega^s))}^2 } \norm[0]{u}_{L^{\infty}(J;H^{k+1}(\Omega))}^2 \big\}
      \\
      & \hspace{2em} + (\Delta t)^2 \mu^{-1} \cbr[1]{\norm[0]{\partial_{tt}u}_{L^2(J;L^2(\Omega^s))}^2
        + \norm[0]{\partial_tu}^2_{L^2(J;L^2(\Omega^s))} \norm[0]{u}_{L^{\infty}(J,H^1(\Omega^s))}^2 } \big].
    \end{split}
  \end{equation*}
\end{corollary}

\section{Numerical examples}
\label{s:numexamples}
  
We implement the fully discrete HDG method \cref{eq:hdg-fully} in
Netgen/NGSolve \cite{Schoberl:1997,Schoberl:2014}. For all examples we
choose the penalty parameter as $\beta = 8k^2$ (see
\cite{Ainsworth:2012,Riviere:book}), where $k$ is the polynomial
degree in the approximation spaces.

\subsection{Rates of convergence}
\label{ss:rofc}
  
In this section we verify the rates of convergence by the method of
manufactured solutions. For this we consider the domains
$\Omega^s = (0,1) \times (0,0.5)$ and
$\Omega^d = (0,1) \times (-0.5, 0)$. The interface is given by
$\Gamma^I = \overline{\Omega}^s \cap \overline{\Omega}^d$ while
$\Gamma_D^s = \cbr{x \in \Gamma^s:\ x_1=0 \text{ or } x_2=0.5}$,
$\Gamma_N^s = \Gamma^s \backslash \Gamma_D^s$,
$\Gamma_D^d = \cbr{x \in \Gamma^d:\ x_2=-0.5}$, and
$\Gamma_N^d = \Gamma^d \backslash \Gamma_D^d$. To construct a
manufactured solution, we consider the following inhomogeneous
boundary conditions and modified interface conditions:
\begin{align*}
  u^s &= U^s & & \text{on } \Gamma^{s}_D \times J,
  \\
  \sigma_d^s n &= S^{s} & & \text{on } \Gamma^s_{N} \times J,
  \\
  u^d \cdot n &= U^d & & \text{on } \Gamma^d_N \times J,
  \\
  p^d &= P^d & & \text{on } \Gamma^d_D \times J,
  \\
  u^s \cdot n &= u^d \cdot n + M^u & & \text{on } \Gamma^I \times J,
  \\
  -2\mu\del[0]{\varepsilon(u^s)n}^t &= \alpha\mu\kappa^{-1/2} (u^s)^t + (M^e)^t
             & & \text{on } \Gamma^I \times J,
  \\
  (\sigma_d^s n)\cdot n &= p^d + M^p & & \text{on } \Gamma^I \times J,  
\end{align*}
where $U^s$, $S^s$, $U^d$, $P^d$, $M^u$, $M^e$, and $M^p$, and the
functions $f^s$ and $f^d$ in \cref{eq:ns-a,eq:darcy-b} are chosen such
that the exact solution is given by:
\begin{align*}
  p^s
  &= \sin(3x_1-t)\cos(4x_2) + \sin(2\pi x_1x_2),
  &
  u^s
  &=
  \begin{bmatrix}
    \pi x_1 \cos(\pi x_1 x_2 - t) + 1\\
    -\pi x_2 \cos(\pi x_1 x_2 - t) + 2x_1
  \end{bmatrix},
  \\
  p^d &= \cos(3x_1x_2-t/10),
  &
  u^d
  &= -(\kappa/\mu)\nabla p^d.
\end{align*}
The initial condition for the velocity is set by first solving the
stationary Stokes--Darcy problem with the above boundary/interface
conditions and functions $f^s$ and $f^d$. In our simulations we choose
$\kappa=10^{-4}$ and $\alpha=1$. We consider polynomial degrees $k=1$
(corresponding to approximating the cell pressure by piecewise
constants and the other unknowns by piecewise linear polynomials) and
$k=2$ (in which the cell pressure is approximated by piecewise linears
and the other unknowns by piecewise quadratic polynomials). We compare
results obtained by choosing $\mu=10^{-1}$, $\mu=10^{-3}$, and
$\mu=10^{-5}$.

Let us define $e_u := u - u_h$ and, similar to \cite{Girault:2009},
$\norm[0]{e_u}_E^2 := \del[1]{\sum_{K \in \mathcal{T}^s} |e_u|_{1,K}^2
  + \norm[0]{e_u}_{\Omega^d}^2}$. From Corollary
\ref{cor:velocityestimate} we expect that, for smooth enough
solutions, $\norm[0]{e_u}_E = \mathcal{O}(h^k + \Delta t)$. The
spatial rate of convergence is indeed observed in
\cref{tab:rates_tc1_h-rates} (to obtain these results we chose our
time step as $\Delta t = 0.8 h^{k+1}$ and set
$J=(0,0.1)$). \Cref{tab:rates_tc1_h-rates} also lists the $L^2$-norm
of $e_u$ and $e_p := p - p_h$. For the velocity we observe that
$\norm[0]{e_u}_{\Omega} = \mathcal{O}(h^{k+1})$ for $\mu=10^{-1}$ and
$\norm[0]{e_u}_{\Omega} \approx \mathcal{O}(h^{k+1/2})$ for
$\mu=10^{-5}$. For $\mu = 10^{-3}$ we have that
$\norm[0]{e_u}_{\Omega}$ lies between $\mathcal{O}(h^{k+1/2})$ and
$\mathcal{O}(h^{k+1})$, depending on whether $k=1$ or $k=2$. The
slower convergence in the $L^2$-norm for $\mu=10^{-5}$ is not
surprising; the flow problem is advection dominated and analysis of
HDG methods for the scalar advection equation reveals a priori error
estimates for the solution to be $\mathcal{O}(h^{k+1/2})$, see
\cite[Lemma 4.8]{Wells:2011}. We furthermore observe optimal rates of
convergence for the pressure:
$\norm[0]{e_p}_{\Omega} = \mathcal{O}(h^k)$.

We next consider the temporal rates of convergence. For this we
consider a fine mesh with 9508 cells and set $k=2$ and $J=(0,1)$. In
\cref{tab:rates_tc1_dt-rates} we vary the time step and present the
errors and rates of convergence. All errors are
$\mathcal{O}(\Delta t)$.

Finally, let us remark that despite our analysis holding only under
the small data assumption (see \cref{eq:dataassumption}), we are
nevertheless able to compute the solution for very small values of
viscosity. From \cref{tab:rates_tc1_h-rates,tab:rates_tc1_dt-rates} we
even observe that the variation in $\norm[0]{e_u}_E$ for the different
values of $\mu$ is small, despite the upper bound in Corollary
\ref{cor:velocityestimate} depending on $\mu$ and $\mu^{-1}$.

\begin{table}[tbp]
  \caption{Errors and spatial rates of convergence for a manufactured
    solution (see~\cref{ss:rofc}). Results are for $k=1$ and $k=2$
    with parameters $\kappa=10^{-4}$, $\alpha=1$, and
    $\mu\in\{10^{-1}, 10^{-3}, 10^{-5}\}$. Here $e_u = u-u_h$ and
    $e_p = p-p_h$. The rate of convergence is denoted by $r$.  }
  \begin{center}
    \begin{tabular}{ccccccc}
      \hline
      Cells & $\norm[0]{e_u}_E$ & $r$ & $\norm[0]{e_u}_{\Omega}$ & $r$ & $\norm[0]{e_p}_{\Omega}$ & $r$ \\
      \hline
      \multicolumn{7}{l}{$k=1$, $\mu=10^{-1}$} \\
      152 & 4.8e-01 & 0.9 & 9.4e-03 & 1.9 & 1.1e-01 & 0.7 \\
      580 & 2.1e-01 & 1.2 & 1.8e-03 & 2.4 & 4.5e-02 & 1.3 \\
      2362 & 1.0e-01 & 1.0 & 4.2e-04 & 2.1 & 2.2e-02 & 1.0 \\
      9508 & 5.1e-02 & 1.0 & 9.8e-05 & 2.1 & 1.0e-02 & 1.1 \\
      \multicolumn{7}{l}{$k=1$, $\mu=10^{-3}$} \\
      152 & 5.5e-01 & 0.9 & 1.4e-02 & 2.0 & 6.7e-02 & 1.1 \\
      580 & 2.5e-01 & 1.1 & 3.9e-03 & 1.9 & 3.3e-02 & 1.0 \\
      2362 & 1.2e-01 & 1.1 & 1.2e-03 & 1.7 & 1.6e-02 & 1.0 \\
      9508 & 5.6e-02 & 1.1 & 3.6e-04 & 1.7 & 7.9e-03 & 1.0 \\
      \multicolumn{7}{l}{$k=1$, $\mu=10^{-5}$} \\
      152 & 2.4e+00 & 3.8 & 1.5e-01 & 4.3 & 7.8e-02 & 6.5 \\
      580 & 3.2e-01 & 2.9 & 2.5e-02 & 2.6 & 3.3e-02 & 1.3 \\
      2362 & 1.4e-01 & 1.2 & 5.6e-03 & 2.2 & 1.6e-02 & 1.0 \\
      9508 & 7.2e-02 & 0.9 & 1.6e-03 & 1.8 & 7.9e-03 & 1.0 \\
      \hline
      \multicolumn{7}{l}{$k=2$, $\mu=10^{-1}$} \\
      152 & 3.7e-02 & 2.1 & 5.9e-04 & 3.1 & 8.7e-03 & 2.5 \\
      580 & 7.6e-03 & 2.3 & 5.6e-05 & 3.4 & 1.9e-03 & 2.2 \\
      2362 & 1.7e-03 & 2.2 & 5.6e-06 & 3.3 & 4.8e-04 & 2.0 \\
      9508 & 4.0e-04 & 2.1 & 6.4e-07 & 3.1 & 1.2e-04 & 2.1 \\      
      \multicolumn{7}{l}{$k=2$, $\mu=10^{-3}$} \\
      152 & 4.7e-02 & 2.0 & 1.0e-03 & 2.7 & 5.5e-03 & 2.2 \\
      580 & 9.2e-03 & 2.3 & 1.3e-04 & 3.0 & 1.3e-03 & 2.1 \\
      2362 & 2.0e-03 & 2.2 & 1.6e-05 & 2.9 & 3.0e-04 & 2.1 \\
      9508 & 4.9e-04 & 2.0 & 2.2e-06 & 2.9 & 7.6e-05 & 2.0 \\
      \multicolumn{7}{l}{$k=2$, $\mu=10^{-5}$} \\
      152 & 5.9e-02 & 4.0 & 3.6e-03 & 4.2 & 5.5e-03 & 3.1 \\
      580 & 1.0e-02 & 2.5 & 4.3e-04 & 3.1 & 1.3e-03 & 2.1 \\
      2362 & 2.4e-03 & 2.1 & 6.1e-05 & 2.8 & 3.0e-04 & 2.1 \\
      9508 & 5.3e-04 & 2.1 & 1.1e-05 & 2.5 & 7.6e-05 & 2.0 \\
      \hline
    \end{tabular}
    \label{tab:rates_tc1_h-rates}
  \end{center}
\end{table}

\begin{table}[tbp]
  \caption{Errors and temporal rates of convergence for a manufactured
    solution (see~\cref{ss:rofc}). Parameters: $k=2$,
    $\kappa=10^{-4}$, $\alpha=1$, and
    $\mu\in\{10^{-1}, 10^{-3}, 10^{-5}\}$. Here $e_u = u-u_h$ and
    $e_p = p-p_h$. The rate of convergence is denoted by $r$.  }
  \begin{center}
    \begin{tabular}{ccccccc}
      \hline
      $\Delta t$ & $\norm[0]{e_u}_E$ & $r$ & $\norm[0]{e_u}_{\Omega}$ & $r$ & $\norm[0]{e_p}_{\Omega}$ & $r$ \\
      \hline
      \multicolumn{7}{l}{$\mu=10^{-1}$} \\
      1/8  & 3.5e-02 & 1.1 & 2.4e-03 & 1.1 & 8.0e-02 & 0.9 \\
      1/16 & 1.7e-02 & 1.1 & 1.2e-03 & 1.0 & 4.2e-02 & 0.9 \\
      1/32 & 8.2e-03 & 1.0 & 5.8e-04 & 1.0 & 2.1e-02 & 1.0 \\
      1/64 & 4.1e-03 & 1.0 & 2.9e-04 & 1.0 & 1.1e-02 & 1.0 \\
      \multicolumn{7}{l}{$\mu=10^{-3}$} \\
      1/8  & 1.5e-01 & 1.0 & 2.1e-02 & 0.9 & 3.0e-02 & 0.8 \\
      1/16 & 7.7e-02 & 0.9 & 1.1e-02 & 0.9 & 1.6e-02 & 0.9 \\
      1/32 & 4.0e-02 & 1.0 & 5.5e-03 & 1.0 & 8.3e-03 & 1.0 \\
      1/64 & 2.0e-02 & 1.0 & 2.8e-03 & 1.0 & 4.2e-03 & 1.0 \\
      \multicolumn{7}{l}{$\mu=10^{-5}$} \\
      1/8  & 1.5e-01 & 0.9 & 2.8e-02 & 0.9 & 2.2e-02 & 0.7 \\
      1/16 & 7.8e-02 & 0.9 & 1.5e-02 & 0.9 & 1.2e-02 & 0.9 \\
      1/32 & 4.0e-02 & 1.0 & 7.4e-03 & 1.0 & 6.2e-03 & 0.9 \\
      1/64 & 2.0e-02 & 1.0 & 3.8e-03 & 1.0 & 3.2e-03 & 1.0 \\
      \hline
    \end{tabular}
    \label{tab:rates_tc1_dt-rates}
  \end{center}
\end{table}

\subsection{Surface/subsurface flow with nonuniform permeability field}
\label{ss:applicationtest}

In this example we consider surface/subsurface flow. For this example
we divide the domain $\Omega = (0,1) \times (-0.5, 0.5)$ into two
subdomains $\Omega^s$ and $\Omega^d$. We consider a case where the
interface $\Gamma^I = \overline{\Omega^s} \cap \overline{\Omega^d}$ is
not horizontal (see \cref{fig:ssdomain-dom}). Furthermore, let
$\Gamma_D^d = \cbr{x \in \Gamma^d:\ x_2=-0.5}$, and
$\Gamma_N^d = \Gamma^d \backslash \Gamma_D^d$. We then impose the
following boundary conditions:
\begin{align*}
  u^s &= (\tfrac{5}{42}(10 x_2 + 1)(1-x_1/5)(\cos(\pi t/5)+\tfrac{11}{10}), 0) && \text{on } \Gamma^s \times J,
  \\
  u^d \cdot n &= 0 && \text{on } \Gamma_N^d \times J,
  \\
  p^d &= 0 && \text{on } \Gamma_D^d \times J,
\end{align*}
and set $f^s=0$ and $f^d=0$. We consider both $\mu=10^{-1}$ and
$\mu = 10^{-3}$ together with $\alpha=0.5$, and choose the permeability to
be piecewise constant such that $\mu^{-1}\kappa = 10^{-r}$ with
$r \in [2,6]$ a random number that is chosen differently in each
element of the mesh in $\Omega^d$. (The analysis presented in this
paper assumes a constant permeability, but noting that
$0 < \kappa_{\min} \le \kappa(x) \le \kappa_{\max}$ the analysis is
easily extended to this situation.) A plot of the permeability is
given in \cref{fig:ssdomain-perm}. To set the initial condition for
the velocity in $\Omega^s$ we solve the stationary Stokes--Darcy
problem.

We compute the solution on a mesh consisting of 91720 elements, using
$k=2$, a time step of $\Delta t = 0.01$, and on the time interval
$J=(0,10)$. Plots of the velocity and pressure fields at different
time levels are shown in \cref{fig:ssvelocity,fig:sspressure}, both
for $\mu=10^{-1}$ and $\mu=10^{-3}$. The velocity fields at $t=0$ and
$t=10$ for both values of viscosity are similar: flow in $\Omega^s$
away from the interface is more or less horizontal while in $\Omega^d$
flow finds its way through the permeability maze in the direction of
negative pressure gradient. At $t=5.2$ (when the inflow magnitude of
the velocity is close to its minimum), the behavior of the velocity
fields when $\mu=10^{-1}$ and $\mu=10^{-3}$ are significantly
different: when $\mu=10^{-1}$ the velocity field is similar to that at
$t=0$ and $t=10$, but when $\mu=10^{-3}$ we obtain a large area of
circulation. The pressure fields are similar for the two values of
viscosity and follow a more or less linear profile in
$\Omega^d$. Pressure variations in $\Omega^s$ are small.

\begin{figure}[tbp]
  \centering
  \subfloat[Domain. \label{fig:ssdomain-dom}]{\includegraphics[width=0.51\textwidth]{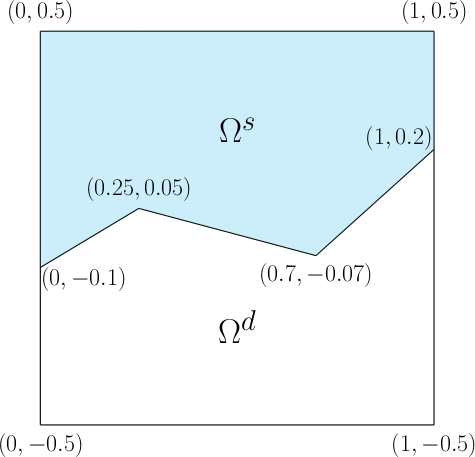}}
  \quad
  \subfloat[Permeability. \label{fig:ssdomain-perm}]{\includegraphics[width=0.45\textwidth]{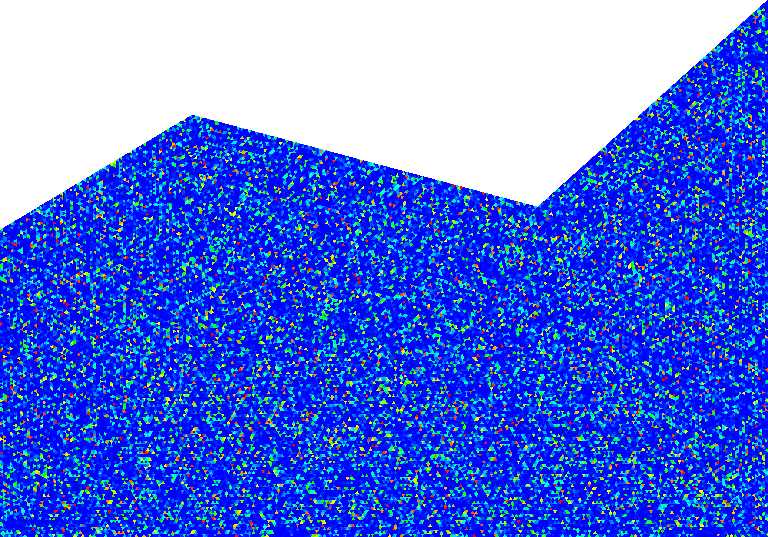}}
  \caption{The surface/subsurface domain $\Omega$ used in
    \cref{ss:applicationtest}.}
  \label{fig:ssdomain}
\end{figure}

\begin{figure}[tbp]
  \centering
  \subfloat[$\mu=10^{-1}$, $t=0$. \label{fig:ssvelocity-1em1-t1}]{\includegraphics[width=0.4\textwidth]{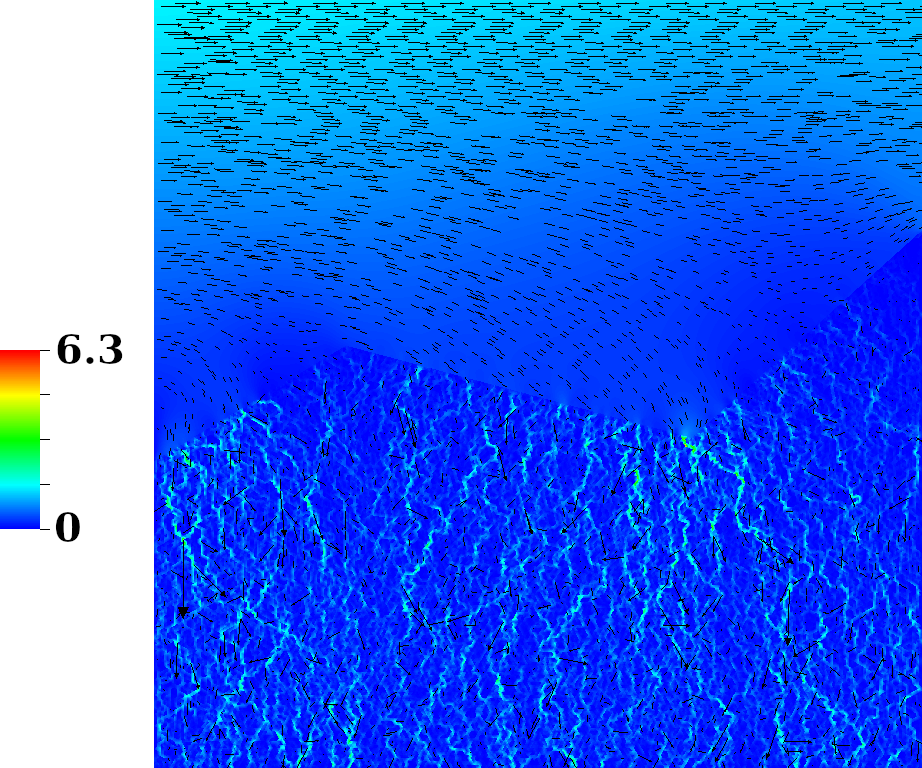}}
  \quad
  \subfloat[$\mu=10^{-3}$, $t=0$. \label{fig:ssvelocity-1em3-t1}]{\includegraphics[width=0.4\textwidth]{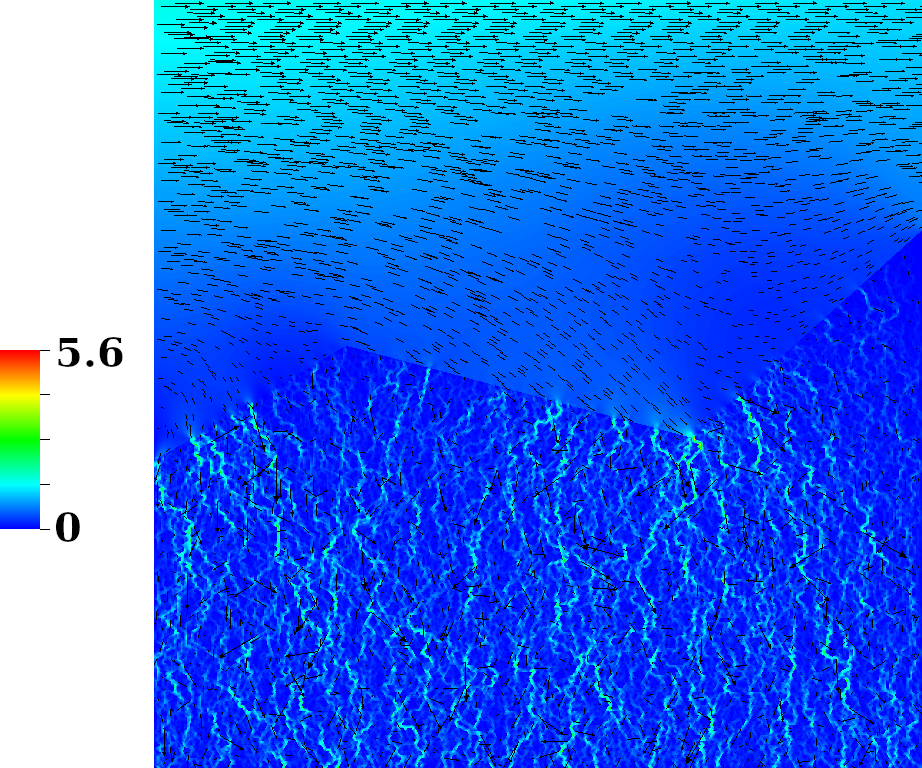}}
  \\
  \subfloat[$\mu=10^{-1}$, $t=5.2$. \label{fig:ssvelocity-1em1-t26}]{\includegraphics[width=0.4\textwidth]{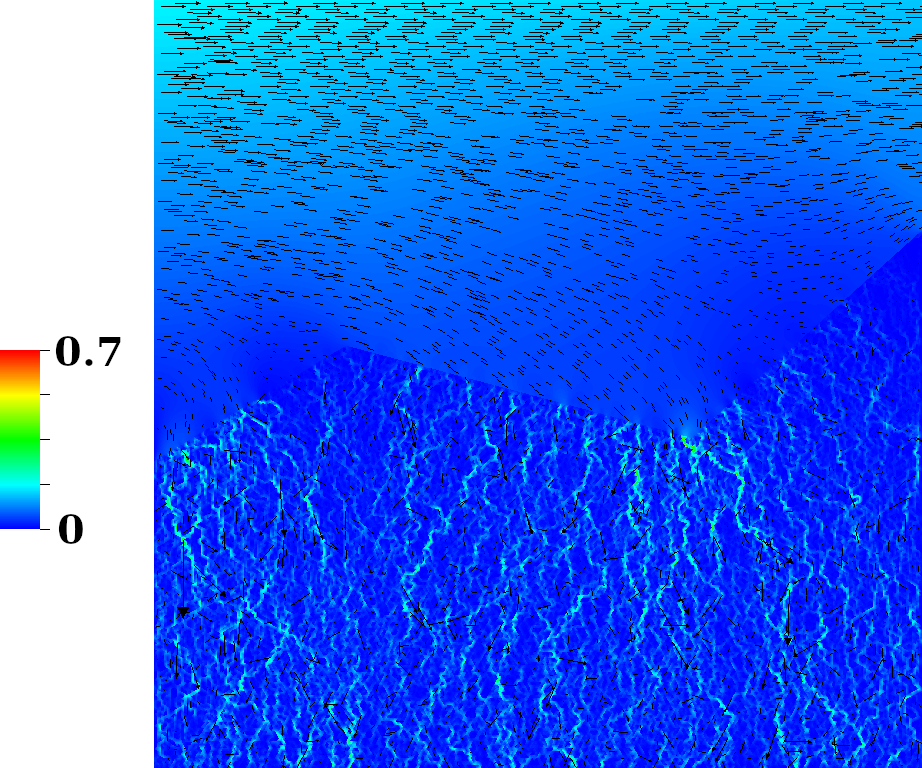}}
  \quad
  \subfloat[$\mu=10^{-3}$, $t=5.2$. \label{fig:ssvelocity-1em3-t26}]{\includegraphics[width=0.4\textwidth]{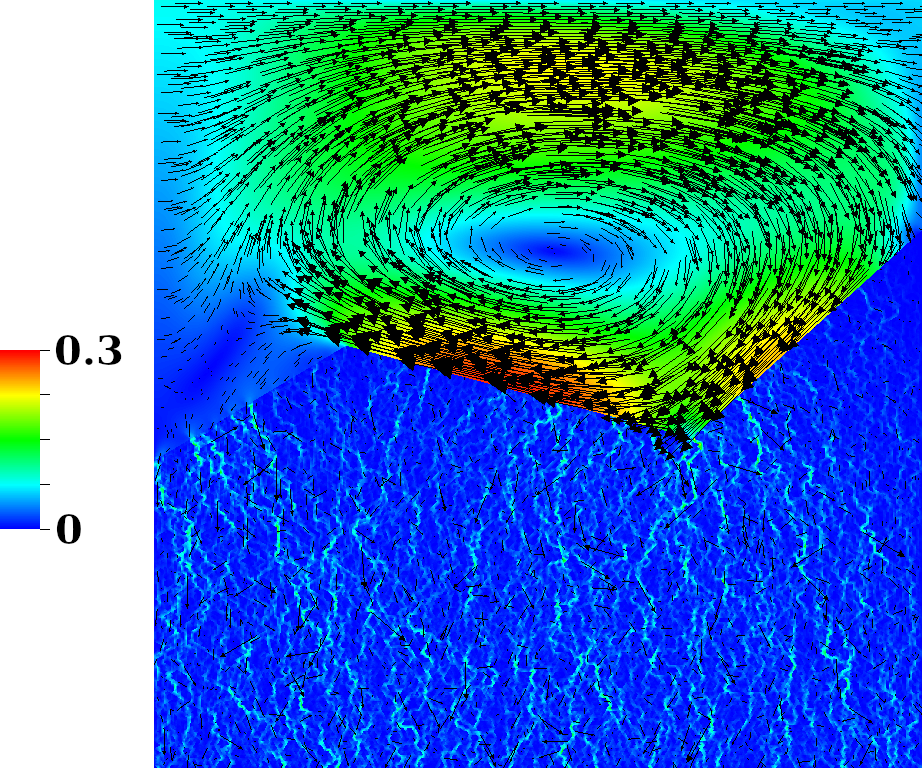}}
  \\
  \subfloat[$\mu=10^{-1}$, $t=10$. \label{fig:ssvelocity-1em1-t50}]{\includegraphics[width=0.4\textwidth]{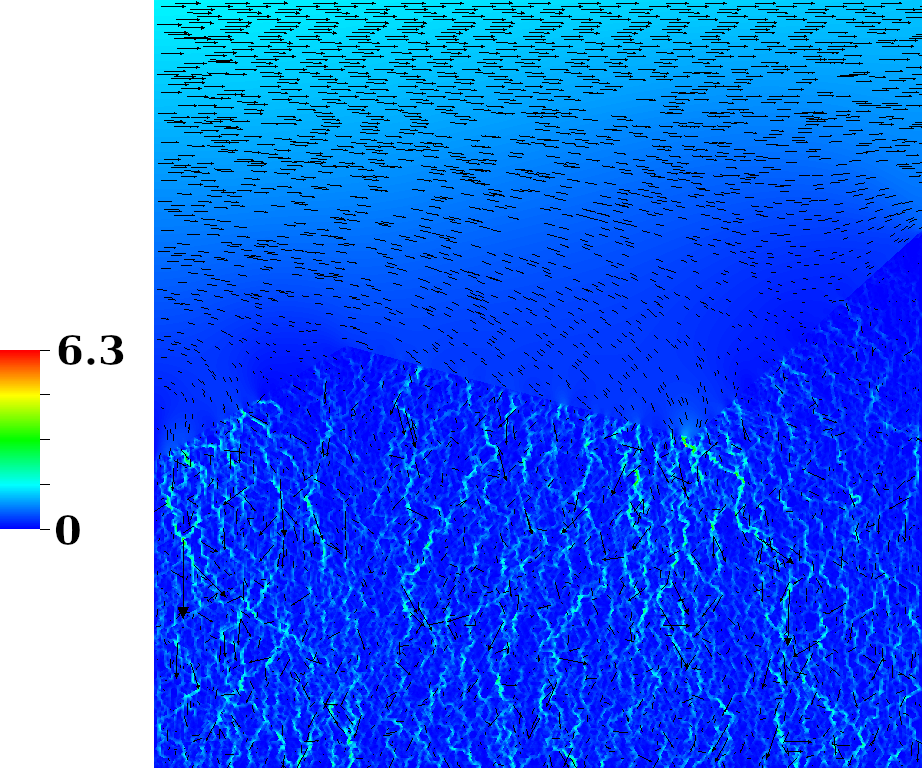}}
  \quad
  \subfloat[$\mu=10^{-3}$, $t=10$. \label{fig:ssvelocity-1em3-t50}]{\includegraphics[width=0.4\textwidth]{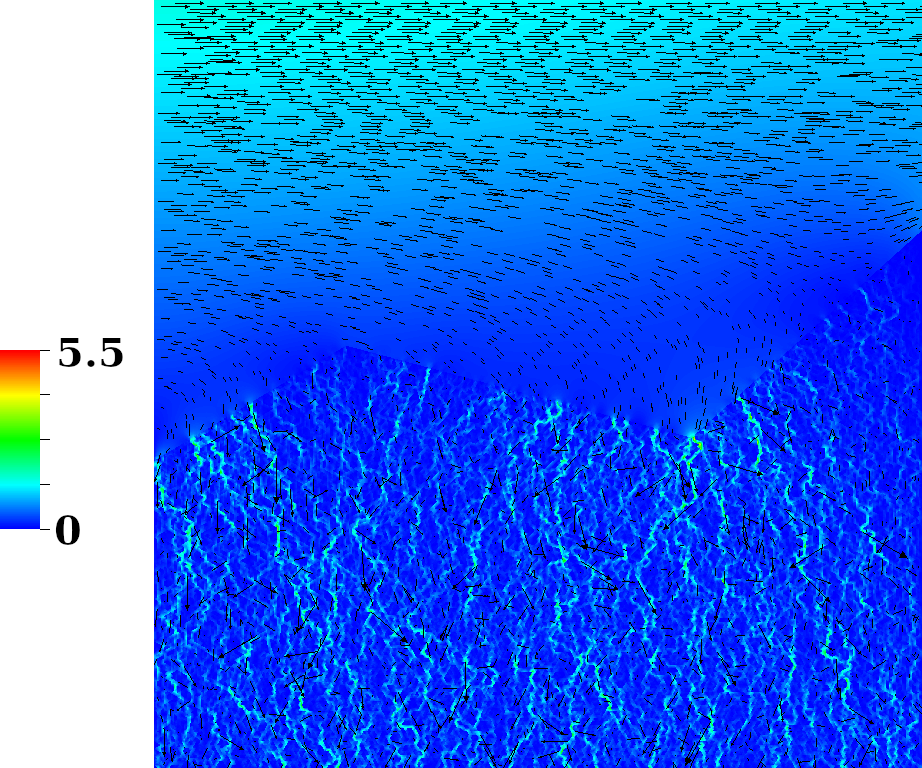}}  
  \caption{Velocity magnitude and velocity vector field at time levels
    $t=0$, $t=5.2$, and $t=10$. Left column: $\mu=10^{-1}$. Right
    column: $\mu=10^{-3}$. See also \cref{ss:applicationtest}.}
  \label{fig:ssvelocity}
\end{figure}

\begin{figure}[tbp]
  \centering
  \subfloat[$\mu=10^{-1}$, $t=0$. \label{fig:sspressure-1em1-t1}]{\includegraphics[width=0.4\textwidth]{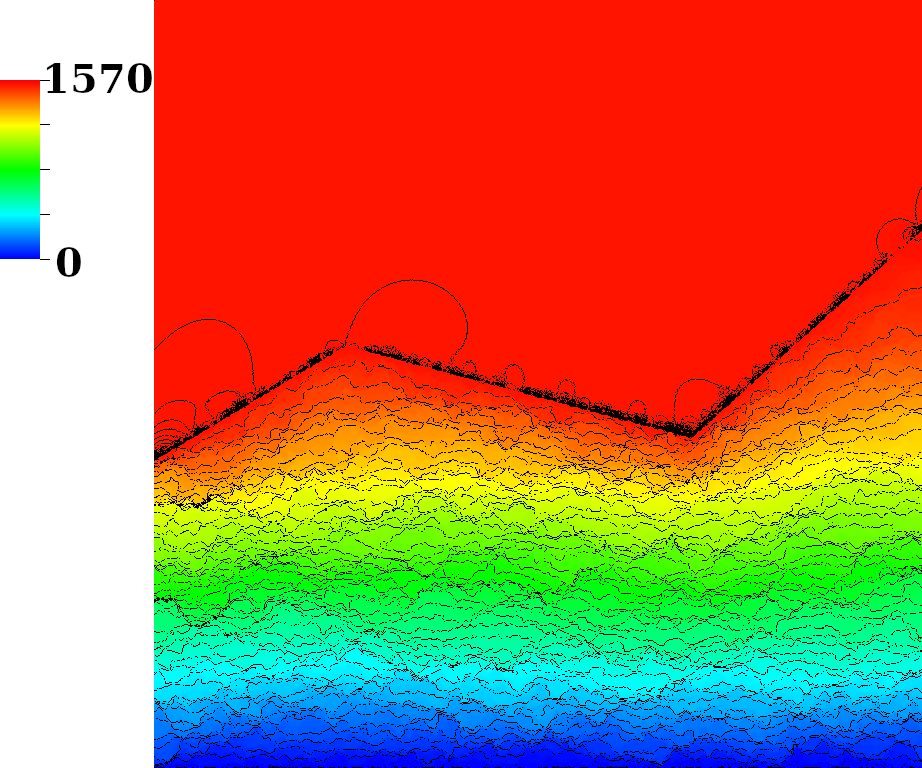}}
  \quad
  \subfloat[$\mu=10^{-3}$, $t=0$. \label{fig:sspressure-1em3-t1}]{\includegraphics[width=0.4\textwidth]{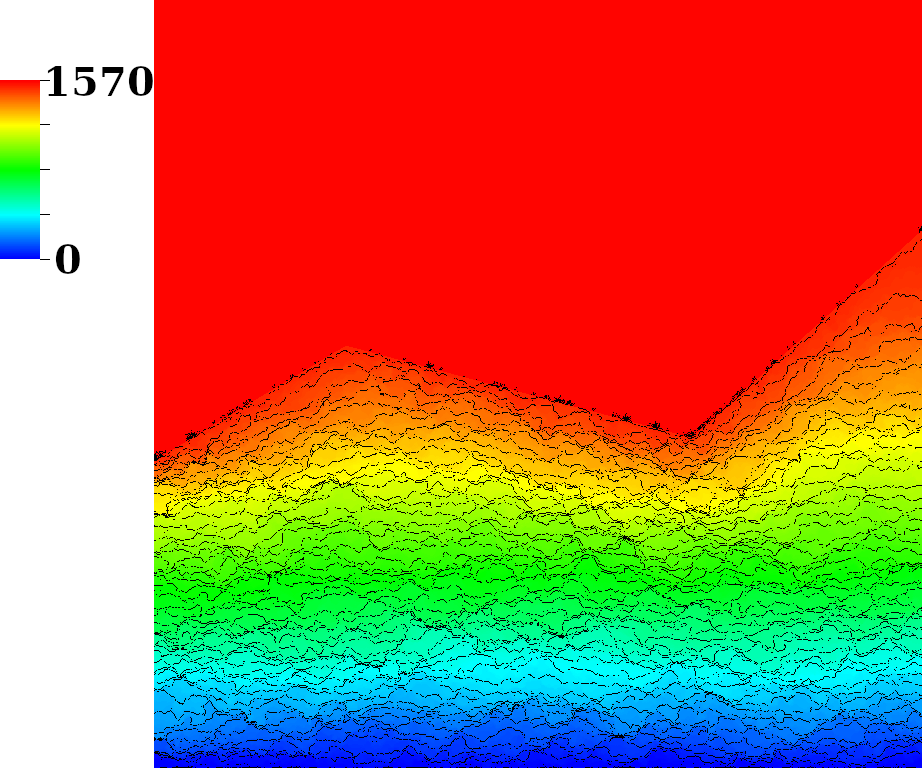}}
  \\
  \subfloat[$\mu=10^{-1}$, $t=5.2$. \label{fig:sspressure-1em1-t26}]{\includegraphics[width=0.4\textwidth]{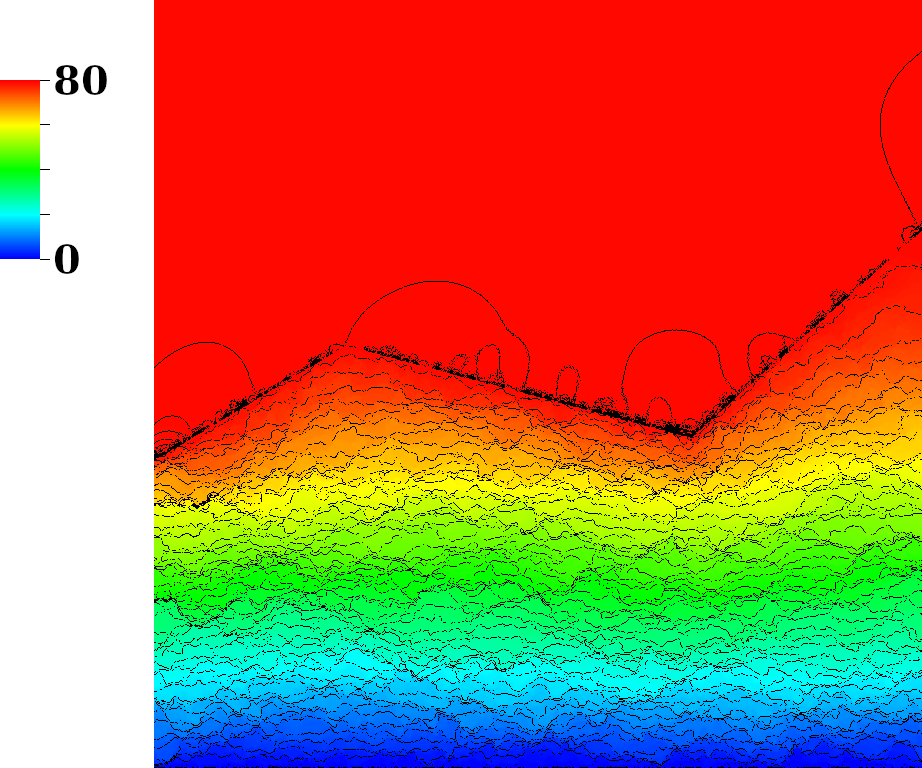}}
  \quad
  \subfloat[$\mu=10^{-3}$, $t=5.2$. \label{fig:sspressure-1em3-t26}]{\includegraphics[width=0.4\textwidth]{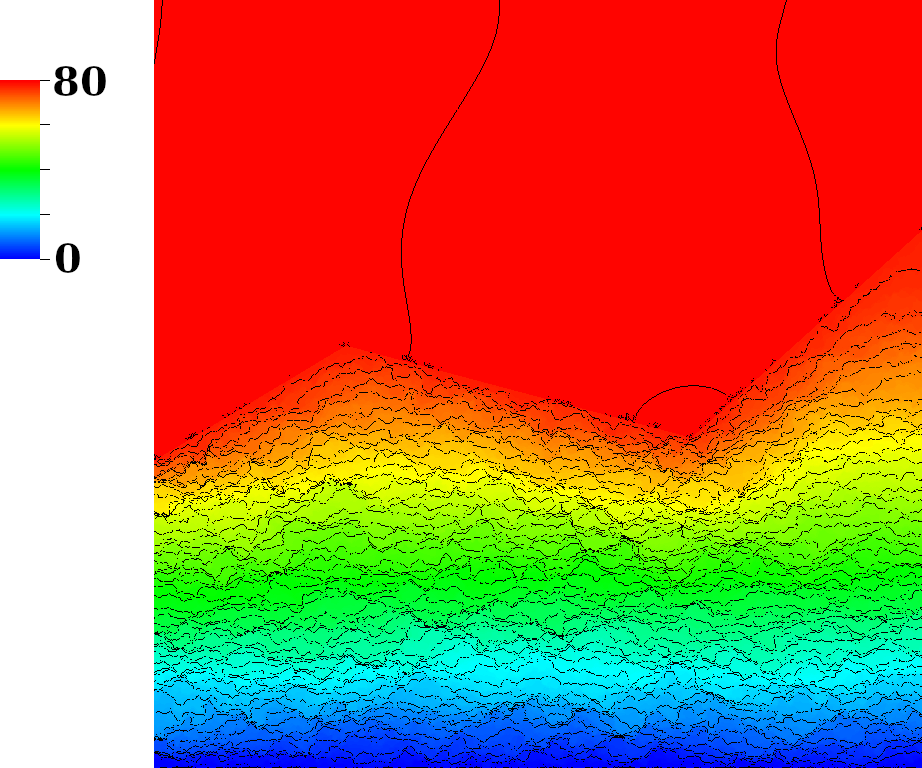}}
  \\
  \subfloat[$\mu=10^{-1}$, $t=10$. \label{fig:sspressure-1em1-t50}]{\includegraphics[width=0.4\textwidth]{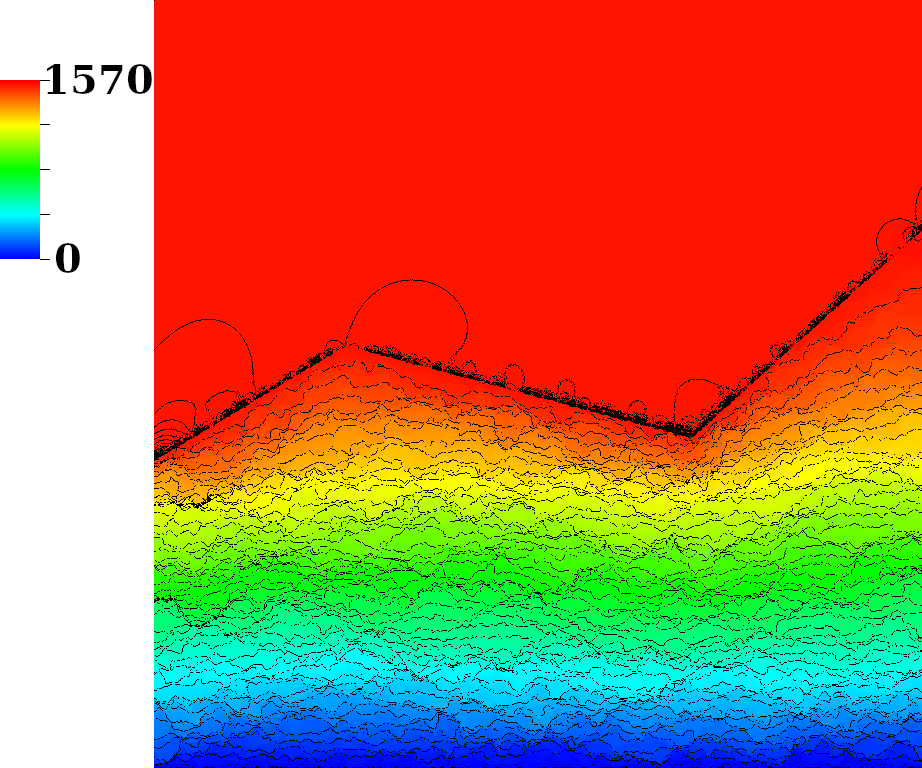}}
  \quad
  \subfloat[$\mu=10^{-3}$, $t=10$. \label{fig:sspressure-1em3-t50}]{\includegraphics[width=0.4\textwidth]{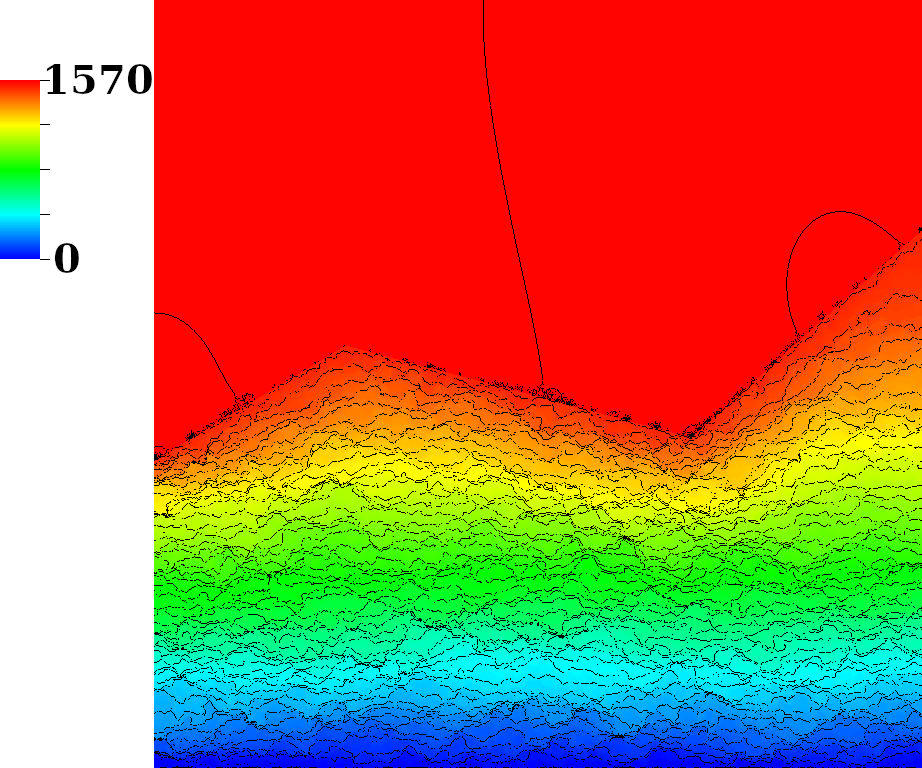}}  
  \caption{Pressure magnitude and contour plots at time levels $t=0$,
    $t=5.2$, and $t=10$. Left column: $\mu=10^{-1}$. Right column:
    $\mu=10^{-3}$. For visualization purposes at $t=0$ and $t=10$ we
    consider in $\Omega^d$ the pressure interval $[0, 1570]$ divided
    into 50 contour lines using a linear scale while in $\Omega^s$ we
    consider the pressure interval $[1500, 1570]$ divided into 100
    contour lines using a log scale. At $t=5.2$ we consider in
    $\Omega^d$ and $\Omega^s$ the pressure intervals $[0, 80]$ and
    $[75,80]$, respectively. See also \cref{ss:applicationtest}.}
  \label{fig:sspressure}
\end{figure}

\section{Conclusions}
\label{s:conclusions}

We presented a strongly conservative HDG method for the coupled
time-dependent Navier--Stokes and Darcy problem. Existence and
uniqueness of a solution to the fully discrete problem were proven
assuming a small data assumption. We furthermore determined a
pressure-independent a priori error estimate for the discrete
velocity. This estimate is optimal in space in the combined discrete
$H^1$-norm on $\Omega^s$ and $H(\text{div})$-norm on $\Omega^d$, and
optimal in time. Our analysis is supported by numerical examples.

\subsubsection*{Acknowledgements}

AC and JJL are funded by the National Science Foundation under grant
numbers DMS-2110782 and DMS-2110781. SR is funded by the Natural
Sciences and Engineering Research Council of Canada through the
Discovery Grant program (RGPIN-05606-2015).

\bibliographystyle{abbrvnat}
\bibliography{references}
\appendix
\section{Proof of the inf-sup condition \cref{eq:infsupbh}}
\label{ap:infsup}

An inf-sup condition of the form \cref{eq:infsupbh} was proven in
\cite[Lemma 2]{Cesmelioglu:2023} assuming that $u=0$ on $\Gamma^s$ and
$u \cdot n = 0$ on $\Gamma^d$. We modify this proof to take into
account the boundary conditions
\cref{eq:ic-bc-b,eq:ic-bc-c,eq:ic-bc-d}. The proof requires the
BDM interpolation operator $\Pi_V: H({\rm div};\Omega) \cap \sbr[0]{L^r(\Omega)}^{\dim}
\rightarrow X_h\cap H({\rm div};\Omega)$, $r>2$, which satisfies
\cref{eq:interpolant_a,eq:interpolant_b,eq:PiVuumKestimate} for all
$u \in \sbr[0]{H^{k+1}(K)}^{\dim}$. We will also require the following
function space:
\begin{equation*}
  \sbr[0]{H^1_{0,\Gamma^s\cup\Gamma_N^d}(\Omega)}^{\dim}
  := \cbr[0]{w \in \sbr[0]{H^1(\Omega)}^{\dim} \, :\,
    w|_{\Gamma^s\cup\Gamma_N^d}=0}.
\end{equation*}
Defining
\begin{equation*}
  \begin{split}
    \tilde{b}_h(\boldsymbol{v}_h, (\bar{q}^s_h,\bar{q}^d_h))
    &:= \sum_{j=s,d}\del[2]{ \langle \bar{q}_h^j, v_h \cdot n^j \rangle_{\partial\mathcal{T}_h^j}
      - \langle \bar{q}_h^j, \bar{v}_h \cdot n^j \rangle_{\Gamma^I} },
    \\
    \text{Ker}(\tilde{b}_h)
    &:= \cbr[0]{\boldsymbol{v}_h \in
      \boldsymbol{X}_h\,:\, \tilde{b}_h(\boldsymbol{v}_h, (\bar{q}^s_h,
      \bar{q}^d_h)) = 0 \ \forall (\bar{q}_h^s,\bar{q}_h^d) \in
      \bar{Q}_h^s \times \bar{Q}_h^d},
  \end{split}
\end{equation*}
and noting that
$b_h(\boldsymbol{v}_h, \boldsymbol{q}_h) = -(q_h, \nabla \cdot
v_h)_{\Omega} + \tilde{b}_h(\boldsymbol{v}_h, (\bar{q}^s_h,
\bar{q}^d_h))$, by \cite[Theorem 3.1]{Howell:2011} the inf-sup
condition \cref{eq:infsupbh} holds for all
$\boldsymbol{q}_h \in \boldsymbol{Q}_h$ if there exist constants
$c_{b1}>0$ and $c_{b2}>0$, independent of $h$ and $\Delta t$, such
that
\begin{subequations}
  \begin{align}
    \label{eq:infsup-b1}
    c_{b1} \norm[0]{q_h}_{\Omega}
    &\le \sup_{\substack{\boldsymbol{v}_h \in \text{Ker}(\tilde{b}_h) \\ \boldsymbol{v}_h \ne 0}}
    \frac{ -(q_h, \nabla \cdot v_h)_{\Omega}}{\tnorm{\boldsymbol{v}_h}_v} && \forall q_h \in Q_h,
    \\
    \label{eq:infsup-b2}
    \del[2]{c_{b2}\sum_{j=s,d}\sum_{K\in\mathcal{T}_h^j}h_K\norm[0]{\bar{q}_h^j}^2_{\partial K}}^{1/2}      
    &\le \sup_{\substack{\boldsymbol{v}_h \in \boldsymbol{X}_h \\ \boldsymbol{v}_h \ne 0}}
    \frac{ \tilde{b}_h(\boldsymbol{v}_h, (\bar{q}^s_h,\bar{q}^d_h))}{\tnorm{\boldsymbol{v}_h}_v}
    && \forall (\bar{q}_h^s,\bar{q}_h^d) \in \bar{Q}_h^s \times \bar{Q}_h^d.
  \end{align} 
\end{subequations}
Compared to \cite[Lemma 2]{Cesmelioglu:2023}, only the proof for
\cref{eq:infsup-b1} needs to be modified.

We first seek a suitable
$\boldsymbol{v}_h \in \text{Ker}(\tilde{b}_h)$. Let $q_h \in Q_h$. By
\cite[Remark 3.3]{JJLee:2017} there exists
$v \in \sbr[0]{H^1_{0,\Gamma^s\cup\Gamma_N^d}(\Omega)}^{\dim}$ such
that
\begin{equation}
\label{eq:inf-sup-appendix}
  -(\nabla \cdot v, q_h)_{\Omega} =  \norm[0]{q_h}_{\Omega}^2, \quad c_{vq}\norm[0]{v}_{1,\Omega} \le \norm[0]{q_h}_{\Omega},
\end{equation}
where $c_{vq}>0$ is a constant independent of $h$ and $\Delta t$. Let
$\bar{\Pi}_V:\sbr[0]{H^1(\Omega^s)}^{\dim} \to \bar{X}_h$ be the
$L^2$-projection into the facet velocity space and note that the pair
$\boldsymbol{v}_h = (\Pi_V v, \bar{\Pi}_V v)$ lies in
$\text{Ker}(\tilde{b}_h)$:
  \begin{equation*}
    \begin{split}
      \tilde{b}_h((\Pi_V v, \bar{\Pi}_V v), (\bar{q}^s_h,\bar{q}^d_h))
      =& \langle \bar{q}_h^s, (\Pi_Vv - \bar{\Pi}_Vv) \cdot n^s \rangle_{\Gamma^I} 
      + \langle \bar{q}_h^s, \Pi_Vv \cdot n^s \rangle_{\Gamma^s}
      + \langle \bar{q}_h^d, (\Pi_Vv - \bar{\Pi}_Vv) \cdot n^d \rangle_{\Gamma^I}
      + \langle \bar{q}_h^d, \Pi_Vv \cdot n^d \rangle_{\Gamma^d}
      \\
      =& \langle \bar{q}_h^s, (v - v) \cdot n^s \rangle_{\Gamma^I} + \langle \bar{q}_h^d, (v - v) \cdot n^d \rangle_{\Gamma^I} = 0,
    \end{split}
  \end{equation*}
where the first equality is because $\Pi_Vv \cdot n^j$ is continuous
on element boundaries and $\bar{q}_h^j$ is single-valued. The second
equality is by properties of $\Pi_V$ and $\bar{\Pi}_V$,
$v \cdot n^j = 0$ on $\Gamma^s \cup \Gamma_N^d$, and $\bar{q}_h^d = 0$
on $\Gamma_D^d$. Therefore,
$(\Pi_Vv, \bar{\Pi}_Vv)\in \text{Ker}(\tilde{b}_h)$.

We now proceed to find a bound for
$\tnorm{(\Pi_V v, \bar{\Pi}_V v)}_{v}$ in terms of
$\norm{v}_{1,\Omega}$. First, note that by definition,
\begin{equation*}
  \begin{split}
    \tnorm{(\Pi_V v, \bar{\Pi}_V v)}_{v,d}^2
    =&
    \norm[0]{\Pi_Vv}_{\text{div};\Omega^d}^2 
    + \sum_{F \in \mathcal{F}^d_h\backslash(\mathcal{F}_h^I\cup\mathcal{F}_h^{D,d})}h_F^{-1}\norm[0]{\jump{\Pi_Vv\cdot n}}_{F}^2
    \\
    &+ \sum_{K \in \mathcal{T}_h^d} h_K^{-1}\norm[0]{(\Pi_Vv - \bar{\Pi}_Vv)\cdot n}_{\partial K \cap \Gamma^I}^2    
    =: I_1 + I_2 + I_3.
  \end{split}
\end{equation*}
In \cite[Lemma 2]{Cesmelioglu:2023} it was shown that
$I_1 + I_3 \le C \norm[0]{v}_{1,\Omega^d}^2$. Furthermore, $I_2=0$
because $\Pi_Vv \in H(\text{div};\Omega^d)$ and $v=0$ on
$\Gamma_N^d$. Therefore,
$\tnorm{(\Pi_V v, \bar{\Pi}_V v)}_{v,d} \le C
\norm{v}_{1,\Omega^d}$. In the proof of \cite[Lemma
2]{Cesmelioglu:2023} it was also shown that
\begin{equation}
  \label{eq:tnormpivpibv}
  \tnorm{(\Pi_Vv, \bar{\Pi}_Vv)}_{v,s}
  \le C\norm[0]{v}_{1,\Omega^s},
  \quad
  \norm[0]{(\bar{\Pi}_Vv)^t}_{\Gamma^I}
  \le C\norm[0]{v}_{1,\Omega^s}.    
\end{equation}
By definition of $\tnorm{\cdot}_v$ and using the preceding bounds on
$\tnorm{(\Pi_Vv, \bar{\Pi}_Vv)}_{v,s}$,
$\norm[0]{(\bar{\Pi}_Vv)^t}_{\Gamma^I}$, and
$\tnorm{(\Pi_V v, \bar{\Pi}_V v)}_{v,d}$, we find
\begin{equation*}
  \tnorm{(\Pi_V v, \bar{\Pi}_V v)}_{v}
  \le C \norm{v}_{1,\Omega}.
\end{equation*}
\Cref{eq:infsup-b1} now follows from this and
\cref{eq:inf-sup-appendix}:
\begin{equation*}
  \sup_{\substack{\boldsymbol{v}_h \in \text{Ker}(\tilde{b}_h) \\ \boldsymbol{v}_h \ne 0}}
  \frac{-(q_h, \nabla \cdot v_h)_{\Omega}}{\tnorm{\boldsymbol{v}_h}_{v}}
  \ge
  \frac{-(q_h, \nabla \cdot \Pi_Vv)_{\Omega}}{\tnorm{(\Pi_Vv,\bar{\Pi}_V v)}_{v}}
  \ge
  \frac{\norm[0]{q_h}^2_{\Omega}}{C\norm[0]{v}_{1,\Omega}}
  \ge
  \frac{c_{vq}}{C}\norm[0]{q_h}_{\Omega}.
\end{equation*}

\section{Useful inequalities}
Let $g$ be a sufficiently smooth function. Using Taylor's theorem in
integral form, it is shown in \cite[Lemma 7.67]{John:book}) that
\begin{equation}
  \label{eq:pdtgddtg}
  \norm[0]{\partial_tg^{n+1} - d_tg^{n+1}}_{\Omega^s} \le C\sqrt{\Delta t}\norm[0]{\partial_{tt}g}_{L^2(t^n,t^{n+1};L^2(\Omega^s))}.
\end{equation}
A minor modification of the proof of \cref{eq:pdtgddtg} leads to:
\begin{equation}
  \label{ineq:taylor}
  \norm[0]{g^{n+1}-g^n}_{\Omega^s}
  \le C \Delta t (\norm[0]{\partial_tg^n}_{\Omega^s} + \norm[0]{\partial_{tt}g}_{L^2(t^{n},t^{n+1};L^2(\Omega^s))}).    
\end{equation}
We also have, by the fundamental theorem of Calculus and the
Cauchy--Schwarz inequality, that
\begin{equation}
  \label{eq:dtfnp1estimate}
  \norm[0]{g^{n+1}-g^n}_{\Omega^s} 
  = \norm[0]{\int_{t^n}^{t^{n+1}} \partial_tg \dif t}
  \le \del[2]{\int_{t^n}^{t^{n+1}}\dif t}^{1/2} \del[2]{\int_{t^n}^{t^{n+1}}\norm[0]{\partial_tg}_{\Omega^s}^2\dif t}^{1/2}
  \le \sqrt{\Delta t}\norm[0]{\partial_tg}_{L^2(t^n,t^{n+1};L^2(\Omega^s))}.
\end{equation}

\section{Proof of \cref{eq:I4}}
\label{ap:I4}
To prove \cref{eq:I4} we will use the following result, which is due
to a discrete Sobolev embedding \cite[Theorem 5.3]{Pietro:book} and
\cref{eq:dpoincareineq}:
\begin{equation}
  \label{eq:L6Sobolev}
  \del[2]{\sum_{K\in\mathcal{T}_h^s}\norm[0]{v_h}_{L^6(K)}^6}^{1/6} \le C \tnorm{\boldsymbol{v}_h}_{v,s} \quad \forall \boldsymbol{v}_h \in \boldsymbol{X}_h.
\end{equation}
Let us first write $I_4$ as:
\begin{equation*}
  \begin{split}
    I_4 
    =&[t_h(u^n; \boldsymbol{\Pi}_V u^{n+1}, \boldsymbol{e}_u^{h,n+1}) - t_h(u_h^n; \boldsymbol{\Pi}_V u^{n+1}, \boldsymbol{e}_u^{h,n+1})]
    \\
    =&[t_h(u^n; \boldsymbol{u}^{n+1}, \boldsymbol{e}_u^{h,n+1}) - t_h(u^n; \boldsymbol{u}^{n+1} - \boldsymbol{\Pi}_V u^{n+1}, \boldsymbol{e}_u^{h,n+1})]
    \\
    &-
    [t_h(u_h^n; \boldsymbol{u}^{n+1}, \boldsymbol{e}_u^{h,n+1}) - t_h(u_h^n; \boldsymbol{u}^{n+1} - \boldsymbol{\Pi}_V u^{n+1}, \boldsymbol{e}_u^{h,n+1})]
    \\
    =&[t_h(u^n; \boldsymbol{u}^{n+1}, \boldsymbol{e}_u^{h,n+1}) - t_h(u_h^n; \boldsymbol{u}^{n+1}, \boldsymbol{e}_u^{h,n+1})]
    \\
    &
    +[t_h(u_h^n; \boldsymbol{u}^{n+1} - \boldsymbol{\Pi}_V u^{n+1}, \boldsymbol{e}_u^{h,n+1})
    - t_h(u^n; \boldsymbol{u}^{n+1} - \boldsymbol{\Pi}_V u^{n+1}, \boldsymbol{e}_u^{h,n+1})]
    \\
    =&: I_{41} + I_{42}.
  \end{split}
\end{equation*}
For $I_{41}$ we note that since the second argument of $t_h$ is
continuous almost everywhere:
\begin{equation*}
  I_{41} 
  = t_h(u^n - u_h^n; \boldsymbol{u}^{n+1}, \boldsymbol{e}_u^{h,n+1})
  = t_h(e_u^{I,n}; \boldsymbol{u}^{n+1}, \boldsymbol{e}_u^{h,n+1})
  - t_h(e_u^{h,n}; \boldsymbol{u}^{n+1}, \boldsymbol{e}_u^{h,n+1})
  =: I_{411} + I_{412}.
\end{equation*}
We have by \cref{eq:boundedness_th} and Young's inequality,
\begin{equation}
  \label{eq:boundI411}
  \begin{split}
    I_{411}
    &\leq c_w \norm[0]{e_u^{I,n}}_{1,h,\Omega^s}\tnorm{\boldsymbol{u}^{n+1}}_{v,s}\tnorm{\boldsymbol{e}_u^{h,n+1}}_{v,s}
    \\
    &\leq Ch^{k}\norm[0]{u^n}_{k+1,\Omega^s}\norm[0]{\nabla u^{n+1}}_{\Omega^s}\tnorm{\boldsymbol{e}_u^{h,n+1}}_{v,s}
    \\
    &\leq \tfrac{1}{2}\gamma \tnorm{\boldsymbol{e}_u^{h,n+1}}_{v,s}^2
    + \frac{C}{\gamma}h^{2k}\norm[0]{u^n}_{k+1,\Omega^s}^2\norm[0]{\nabla u^{n+1}}_{\Omega^s}^2.      
  \end{split}
\end{equation}
Next, using that $u^{n+1} = \bar{u}^{n+1}$ on facets,
\begin{equation*}
  I_{412}
  =
  - (u^{n+1} \otimes e_u^{h,n}, \nabla e_u^{h,n+1})_{\Omega^s}
  + \langle e_u^{h,n} \cdot n, (e_u^{h,n+1}-\bar{e}_u^{h,n+1})\cdot u^{n+1} \rangle_{\partial\mathcal{T}_h^s}
  + \langle e_u^{h,n}\cdot n, \bar{e}_u^{h,n+1} \cdot u^{n+1} \rangle_{\Gamma^I}.
\end{equation*}
At this point we note that since $e_u^{h,n} \cdot n$,
$\bar{e}_u^{h,n+1}$, and $u^{n+1}$ are single-valued on facets, and
because $u = 0$ on $\Gamma^s$, we have that
$ \langle e_u^{h,n} \cdot n, \bar{e}_u^{h,n+1}\cdot u^{n+1}
\rangle_{\partial\mathcal{T}_h^s} = \langle e_u^{h,n}\cdot n,
\bar{e}_u^{h,n+1} \cdot u^{n+1} \rangle_{\Gamma^I}$. Therefore,
\begin{equation*}
  I_{412}
  =
  - (u^{n+1} \otimes e_u^{h,n}, \nabla e_u^{h,n+1})_{\Omega^s}
  + \langle e_u^{h,n} \cdot n, e_u^{h,n+1}\cdot u^{n+1} \rangle_{\partial\mathcal{T}_h^s}.
\end{equation*}
Integrating by parts, using that $\nabla \cdot e_u^{h,n} = 0$ on each
$K \in \mathcal{T}^s_h$, the generalized H\"older's inequality,
\cref{eq:L6Sobolev}, and Young's inequality:
\begin{equation}
  \label{eq:boundI412}
  \begin{split}
    I_{412}
    =&
    (\nabla \cdot (u^{n+1} \otimes e_u^{h,n}), e_u^{h,n+1})_{\Omega^s}    
    = (e_u^{h,n} \cdot \nabla u^{n+1}, e_u^{h,n+1})_{\Omega^s}
    \\
    \le& \norm[0]{e_u^{h,n}}_{\Omega^s}|u^{n+1}|_{W^1_3(\Omega^s)} \norm[0]{e_u^{h,n+1}}_{L^6(\Omega^s)}
    \\
    \le& C\norm[0]{e_u^{h,n}}_{\Omega^s}|u^{n+1}|_{W^1_3(\Omega^s)} \tnorm{\boldsymbol{e}_u^{h,n+1}}_{v,s}
    \\
    \le& \tfrac{1}{2}\gamma \tnorm{\boldsymbol{e}_u^{h,n+1}}_{v}^2 + \frac{C}{\gamma}\norm[0]{e_u^{h,n}}_{\Omega^s}^2|u^{n+1}|_{W^1_3(\Omega^s)}^2.
  \end{split}
\end{equation}
Combining \cref{eq:boundI411,eq:boundI412} we find
\begin{equation}
  \label{eq:boundI41}
  I_{41} \le 
  \gamma \tnorm{\boldsymbol{e}_u^{h,n+1}}_{v}^2
  + \frac{C}{\gamma}h^{2k}\norm[0]{u^n}_{k+1,\Omega^s}^2\norm[0]{\nabla u^{n+1}}_{\Omega^s}^2
  + \frac{C}{\gamma}\norm[0]{e_u^{h,n}}_{\Omega^s}^2|u^{n+1}|_{W^1_3(\Omega^s)}^2.
\end{equation}
We next consider $I_{42}$ which we first write as:
\begin{equation*}
  \begin{split}
    I_{42}
    =&
    [t_h(u_h^n; \boldsymbol{u}^{n+1} - \boldsymbol{\Pi}_V u^{n+1}, \boldsymbol{e}_u^{h,n+1})
    -
    t_h(\Pi_Vu^n; \boldsymbol{u}^{n+1} - \boldsymbol{\Pi}_V u^{n+1}, \boldsymbol{e}_u^{h,n+1})]
    \\
    &+
    [t_h(\Pi_Vu^n; \boldsymbol{u}^{n+1} - \boldsymbol{\Pi}_V u^{n+1}, \boldsymbol{e}_u^{h,n+1})
    - t_h(u^n; \boldsymbol{u}^{n+1} - \boldsymbol{\Pi}_V u^{n+1}, \boldsymbol{e}_u^{h,n+1})]
    \\
    =&
    [t_h(\Pi_Vu^n; \boldsymbol{e}_u^{I,n+1}, \boldsymbol{e}_u^{h,n+1})
    - t_h(u^n; \boldsymbol{e}_u^{I,n+1}, \boldsymbol{e}_u^{h,n+1})]        
    \\
    &+
    [t_h(u_h^n; \boldsymbol{e}_u^{I,n+1}, \boldsymbol{e}_u^{h,n+1})
    -
    t_h(\Pi_Vu^n; \boldsymbol{e}_u^{I,n+1}, \boldsymbol{e}_u^{h,n+1})]
    \\
    =&: I_{421} + I_{422}.
  \end{split}
\end{equation*}
For $I_{421}$ we have by \cref{eq:boundedness_th}, \cite[Lemma
7]{Cesmelioglu:2023}, properties of $\Pi_V$ and $\bar{\Pi}_V$, and
Young's inequality,
\begin{equation}
  \label{eq:boundI421}
  \begin{split}
    I_{421} 
    &\le
    c_w \norm{\Pi_Vu^n - u^n}_{1,h,\Omega^s} \tnorm{\boldsymbol{e}_u^{I,n+1}}_{v,s} \tnorm{\boldsymbol{e}_u^{h,n+1}}_{v,s}
    \\
    &=
    c_w \norm[0]{e_u^{I,n}}_{1,h,\Omega^s} \tnorm{\boldsymbol{e}_u^{I,n+1}}_{v,s} \tnorm{\boldsymbol{e}_u^{h,n+1}}_{v,s}
    \\
    &\le 
    C h^k\norm[0]{u^{n+1}}_{k+1,\Omega^s}\norm[0]{e_u^{I,n}}_{1,h,\Omega^s}\tnorm{\boldsymbol{e}_u^{h,n+1}}_{v,s}
    \\
    &\le 
    C h^{2k}\norm[0]{u^{n+1}}_{k+1,\Omega^s}\norm[0]{u^{n}}_{k+1,\Omega^s}\tnorm{\boldsymbol{e}_u^{h,n+1}}_{v,s}
    \\
    &\le \tfrac{1}{2}\gamma \tnorm{\boldsymbol{e}_u^{h,n+1}}_{v}^2 + \frac{C}{\gamma}h^{4k}\norm[0]{u^{n+1}}_{k+1,\Omega^s}^2\norm[0]{u^{n}}_{k+1,\Omega^s}^2.
  \end{split}
\end{equation}
For $I_{422}$ we find, after integrating by parts,
\begin{equation*}
  \begin{split}
    I_{422} 
    =&t_h(u_h^n; \boldsymbol{e}_u^{I,n+1}, \boldsymbol{e}_u^{h,n+1})
    -
    t_h(\Pi_Vu^n; \boldsymbol{e}_u^{I,n+1}, \boldsymbol{e}_u^{h,n+1})
    \\
    =&
    (\nabla e_u^{I,n+1}, e_u^{h,n+1}\otimes (u_h^n-\Pi_Vu^n))_{\Omega^s} - \langle ((e_u^{I,n+1}-\bar{e}_u^{I,n+1})\otimes (u_h^n-\Pi_Vu^n))n, e_u^{h,n+1} \rangle_{\partial\mathcal{T}^s}
    \\
    &  + \langle (\max(u_h^n\cdot n, 0) - \max(\Pi_Vu^n\cdot n, 0))(e_u^{I,n+1}-\bar{e}_u^{I,n+1}), e_u^{h,n+1}-\bar{e}_u^{h,n+1} \rangle_{\partial\mathcal{T}^s}       
    \\
    =&
    (\nabla e_u^{I,n+1}, e_u^{h,n+1}\otimes e_u^{h,n})_{\Omega^s} 
    - \langle ((e_u^{I,n+1}-\bar{e}_u^{I,n+1})\otimes e_u^{h,n})n, e_u^{h,n+1} \rangle_{\partial\mathcal{T}^s}
    \\
    &  + \langle (\max(u_h^n\cdot n, 0) - \max(\Pi_Vu^n\cdot n, 0))(e_u^{I,n+1}-\bar{e}_u^{I,n+1}), e_u^{h,n+1}-\bar{e}_u^{h,n+1} \rangle_{\partial\mathcal{T}^s}      
    \\
    =&
    (e_u^{h,n} \cdot \nabla e_u^{I,n+1}, e_u^{h,n+1})_{\Omega^s} 
    - \langle e_u^{h,n} \cdot n, e_u^{h,n+1} \cdot (e_u^{I,n+1}-\bar{e}_u^{I,n+1}) \rangle_{\partial\mathcal{T}^s}
    \\
    &  + \langle (\max(u_h^n\cdot n, 0) - \max(\Pi_Vu^n\cdot n, 0))(e_u^{I,n+1}-\bar{e}_u^{I,n+1}), e_u^{h,n+1}-\bar{e}_u^{h,n+1} \rangle_{\partial\mathcal{T}^s}   
    \\
    =& I_{422a} + I_{422b} + I_{422c}.
  \end{split}
\end{equation*}
For $I_{422a}$, using generalized H\"older's inequality,
\cref{eq:L6Sobolev}, that
$|u^{n+1} - \Pi_Vu^{n+1}|_{W_3^1(\Omega^s)} \le c
|u^{n+1}|_{W_3^1(\Omega^s)}$ (see \cite[Theorem 16.4]{Ern:book-I})
we have:
\begin{equation}
  \label{eq:T422a}
  \begin{split}
    I_{422a} 
    &= (e_u^{h,n} \cdot \nabla e_u^{I,n+1}, e_u^{h,n+1})_{\Omega^s}
    \\  
    &\le \norm[0]{e_u^{h,n}}_{\Omega^s} \norm[0]{\nabla e_u^{I,n+1}}_{L^3(\Omega^s)}\norm[0]{e_u^{h,n+1}}_{L^6(\Omega^s)}
    \\
    &\le C \norm[0]{e_u^{h,n}}_{\Omega^s}|e_u^{I,n+1}|_{W_3^1(\Omega^s)}\tnorm{\boldsymbol{e}_u^{h,n+1}}_{v,s}
    \\
    &= C \norm[0]{e_u^{h,n}}_{\Omega^s}|u^{n+1} - \Pi_Vu^{n+1}|_{W_3^1(\Omega^s)}\tnorm{\boldsymbol{e}_u^{h,n+1}}_{v,s}
    \\
    &\le C\norm[0]{e_u^{h,n}}_{\Omega^s}|u^{n+1}|_{W_3^1(\Omega^s)}\tnorm{\boldsymbol{e}_u^{h,n+1}}_{v,s}.
  \end{split}    
\end{equation}
To bound $I_{422b}$ let us first consider a single facet
$F \subset \partial K$. By H\"older's inequality,
\begin{equation}
  \label{eq:singleFtermI422b}
  |\langle e_u^{h,n} \cdot n, e_u^{h,n+1} \cdot (e_u^{I,n+1}-\bar{e}_u^{I,n+1}) \rangle_F|
  \le \norm[0]{e_u^{h,n}}_{L^{3/2}(F)}\norm[0]{e_u^{I,n+1}-\bar{e}_u^{I,n+1}}_{L^3(F)}\norm[0]{e_u^{h,n+1}}_{L^{\infty}(F)}.
\end{equation}
Noting that $\bar{\Pi}_V\Pi_Vu = \Pi_Vu$ on $F$, we have:
\begin{equation}
  \label{eq:boundingeuinp1minebaru}
  \begin{split}
    \norm[0]{e_u^{I,n+1} - \bar{e}_u^{I,n+1}}_{L^3(F)}
    &= \norm[0]{u^{n+1} - \Pi_Vu^{n+1} - \gamma(u^{n+1}) + \bar{\Pi}_Vu^{n+1}}_{L^3(F)}
    \\
    &= \norm[0]{\bar{\Pi}_Vu^{n+1} - \Pi_Vu^{n+1}}_{L^3(F)}
    \\
    &= \norm[0]{\bar{\Pi}_V(u^{n+1} - \Pi_Vu^{n+1})}_{L^3(F)}
    \\
    &\le C \norm[0]{u^{n+1} - \Pi_Vu^{n+1}}_{L^3(F)},
  \end{split}
\end{equation}
where the inequality is by \cite[Lemma 11.18]{Ern:book-I}. By a multiplicative trace inequality
\cite[Lemma 12.15]{Ern:book-I}, we have that
\begin{multline}
  \label{eq:boundingunp1minpivunp1L3}
  \norm[0]{u^{n+1} - \Pi_Vu^{n+1}}_{L^3(F)}
  \\
  \le c \norm[0]{u^{n+1} - \Pi_Vu^{n+1}}_{L^3(K)}^{2/3}
  \del[1]{h_K^{-1/3}\norm[0]{u^{n+1} - \Pi_Vu^{n+1}}_{L^3(K)}^{1/3}
    + \norm[0]{\nabla(u^{n+1} - \Pi_Vu^{n+1})}_{L^3(K)}^{1/3}},
\end{multline}
and by \cite[Theorem 16.4]{Ern:book-I} we have
\begin{equation}
  \label{eq:L2boundunp1L2W13}
  \begin{split}
    \norm[0]{u^{n+1} - \Pi_Vu^{n+1}}_{L^3(K)}
    \le& c h_K|\nabla u^{n+1}|_{W^1_3(K)},
    \\
    \norm[0]{\nabla(u^{n+1} - \Pi_Vu^{n+1})}_{L^3(K)}
    \le& c \norm[0]{u^{n+1}}_{W^1_3(K)}.
  \end{split}
\end{equation}
Combining
\cref{eq:boundingeuinp1minebaru,eq:boundingunp1minpivunp1L3,eq:L2boundunp1L2W13},
\begin{equation}
  \label{eq:euIminebuI}
  \begin{split}
    \norm[0]{e_u^{I,n+1} - \bar{e}_u^{I,n+1}}_{L^3(F)}
    &\le 
    c h_K^{2/3}|\nabla u^{n+1}|_{W^1_3(K)}^{2/3}
    \del[1]{h_K^{-1/3}  h_K^{1/3}|\nabla u^{n+1}|_{W^1_3(K)}^{1/3}
      + \norm[0]{u^{n+1}}_{W^1_3(K)}^{1/3}}
    \\
    &\le ch_K^{2/3}\norm[0]{u^{n+1}}_{W_3^1(K)}.        
  \end{split}
\end{equation}
We also have, by a discrete trace inequality \cite[Lemma
1.52]{Pietro:book}, that
\begin{equation}
  \label{eq:discreteTraceineq}
  \norm[0]{e_u^{h,n}}_{L^{3/2}(F)} \le C h_K^{-2/3}\norm[0]{e_u^{h,n}}_{L^{3/2}(K)}, 
  \qquad
  \norm[0]{e_u^{h,n+1}}_{L^{\infty}(F)} \le C \norm[0]{e_u^{h,n+1}}_{L^{\infty}(K)}.
\end{equation}
Combining \cref{eq:singleFtermI422b} with
\cref{eq:euIminebuI,eq:discreteTraceineq}
\begin{equation*}
  \begin{split}
    |\langle e_u^{h,n} \cdot n, e_u^{h,n+1} \cdot (e_u^{I,n+1}-\bar{e}_u^{I,n+1}) \rangle_F|
    \le&   
    C h_K^{-2/3}\norm[0]{e_u^{h,n}}_{L^{3/2}(K)}h_K^{2/3}\norm[0]{u^{n+1}}_{W_3^1(K)}\norm[0]{e_u^{h,n+1}}_{L^{\infty}(K)}
    \\
    =&
    C \norm[0]{e_u^{h,n}}_{L^{3/2}(K)}\norm[0]{u^{n+1}}_{W_3^1(K)}\norm[0]{e_u^{h,n+1}}_{L^{\infty}(K)}.        
  \end{split}    
\end{equation*}
By \cite[Lemma~1.50]{Pietro:book}, for $\dim=2,3$,
\begin{subequations}
  \begin{align}
    \label{eq:euhnL32bound}
    \norm[0]{ e_u^{h,n} }_{L^{3/2}(K)} 
    &\le C h_K^{\dim/6} \norm[0]{ e_u^{h,n} }_{L^2(K)}, 
    \\
    \label{eq:euhninftybound}
    \norm[0]{ e_u^{h,n+1} }_{L^{\infty}(K)} 
    &\le C h_K^{-\dim/6} \norm[0]{ e_u^{h,n+1} }_{L^6(K)},
  \end{align}
\end{subequations}
so that
\begin{equation*}
  |\langle e_u^{h,n} \cdot n, e_u^{h,n+1} \cdot (e_u^{I,n+1}-\bar{e}_u^{I,n+1}) \rangle_F|
  \le
  C \norm[0]{e_u^{h,n}}_{L^2(K)}\norm[0]{u^{n+1}}_{W_3^1(K)}\norm[0]{e_u^{h,n+1}}_{L^6(K)}.        
\end{equation*}
Since we assumed $F \subset \partial K$ it follows that
\begin{equation*}
  |\langle e_u^{h,n} \cdot n, e_u^{h,n+1} \cdot (e_u^{I,n+1}-\bar{e}_u^{I,n+1}) \rangle_{\partial K}|
  \le
  C \norm[0]{e_u^{h,n}}_{L^2(K)}\norm[0]{u^{n+1}}_{W_3^1(K)}\norm[0]{e_u^{h,n+1}}_{L^6(K)}.        
\end{equation*}
Summing over all elements in $\mathcal{T}_h^s$, using a generalized
H\"older's inequality for the summation over the elements, and
\cref{eq:L6Sobolev},
\begin{equation}
  \label{eq:T422b}
  \begin{split}
    I_{422b}
    &\le 
    C \sum_{K\in\mathcal{T}_h^s}\norm[0]{e_u^{h,n}}_{L^2(K)}\norm[0]{u^{n+1}}_{W_3^1(K)}\norm[0]{e_u^{h,n+1}}_{L^6(K)}
    \\
    &\le 
    C \del[2]{\sum_{K\in\mathcal{T}_h^s}\norm[0]{e_u^{h,n}}_{L^2(K)}^2}^{1/2}
    \del[2]{\sum_{K\in\mathcal{T}_h^s}\norm[0]{u^{n+1}}_{W_3^1(K)}^3}^{1/3}
    \del[2]{\sum_{K\in\mathcal{T}_h^s}\norm[0]{e_u^{h,n+1}}_{L^6(K)}^6}^{1/6}
    \\
    &\le
    C \norm[0]{e_u^{h,n}}_{L^2(\Omega^s)}\norm[0]{u^{n+1}}_{W_3^1(\Omega^s)}\tnorm{\boldsymbol{e}_u^{h,n+1}}_{v,s}.
  \end{split}
\end{equation}
Let us now consider $I_{422c}$. Starting again with a single facet
$F \subset \partial K$, we find using H\"older's inequality,
\begin{equation}
  \label{eq:holdermaxboundF}   
  \begin{split}
    &|\langle (\max(u_h^n\cdot n, 0) - \max(\Pi_Vu^n\cdot n, 0))(e_u^{I,n+1}-\bar{e}_u^{I,n+1}), e_u^{h,n+1}-\bar{e}_u^{h,n+1} \rangle_F|
    \\
    &\le \norm[0]{(\max(u_h^n\cdot n, 0) - \max(\Pi_Vu^n\cdot n, 0))}_{L^{3/2}(F)} \norm[0]{e_u^{I,n+1}-\bar{e}_u^{I,n+1}}_{L^3(F)} \times \\
    & \hspace{20em}\norm[0]{e_u^{h,n+1}-\bar{e}_u^{h,n+1}}_{L^{\infty}(F)}.        
  \end{split}
\end{equation}
Since $a \mapsto \max(a,0)$ is Lipschitz (\cite[Appendix
A.3.1]{Cesmelioglu:2017}), and using \cref{eq:discreteTraceineq}:
\begin{multline}
  \label{eq:lipschitzmaxboundF}
  \norm[0]{\max(u_h^n\cdot n, 0) - \max(\Pi_Vu^n\cdot n, 0)}_{L^{3/2}(F)}
  \\
  \le C \norm[0]{u_h^n - \Pi_Vu^n}_{L^{3/2}(F)}    
  = C \norm[0]{e_u^{h,n}}_{L^{3/2}(F)}
  \le C h_K^{-2/3}\norm[0]{e_u^{h,n}}_{L^{3/2}(K)}.    
\end{multline}
Furthermore, by \cite[Lemma 1.50]{Pietro:book},
\begin{equation}
  \label{eq:eminebarinf2F}
  \norm[0]{e_u^{h,n+1}-\bar{e}_u^{h,n+1}}_{L^{\infty}(F)}
  \le
  ch_K^{(1-\dim)/2}\norm[0]{e_u^{h,n+1}-\bar{e}_u^{h,n+1}}_{L^2(F)},
\end{equation}
From \cref{eq:holdermaxboundF}, \cref{eq:lipschitzmaxboundF},
\cref{eq:euIminebuI}, \cref{eq:euhnL32bound}, and
\cref{eq:eminebarinf2F} we therefore find that
\begin{align*}
  &|\langle (\max(u_h^n\cdot n, 0) - \max(\Pi_Vu^n\cdot n, 0))(e_u^{I,n+1}-\bar{e}_u^{I,n+1}), e_u^{h,n+1}-\bar{e}_u^{h,n+1} \rangle_F|
  &
  \\
  &
    \le C h_K^{-2/3}\norm[0]{e_u^{h,n}}_{L^{3/2}(K)}
    \norm[0]{e_u^{I,n+1}-\bar{e}_u^{I,n+1}}_{L^3(F)}\norm[0]{e_u^{h,n+1}-\bar{e}_u^{h,n+1}}_{L^{\infty}(F)}    
  & \text{(by \cref{eq:lipschitzmaxboundF})}
  \\
  &
    \le C h_K^{-2/3}\norm[0]{e_u^{h,n}}_{L^{3/2}(K)}
    h_K^{2/3}\norm[0]{u^{n+1}}_{W_3^1(K)}\norm[0]{e_u^{h,n+1}-\bar{e}_u^{h,n+1}}_{L^{\infty}(F)}
  & \text{(by \cref{eq:euIminebuI})}
  \\
  &
    \le C h_K^{\dim/6} \norm[0]{ e_u^{h,n} }_{L^2(K)}
    \norm[0]{u^{n+1}}_{W_3^1(K)}\norm[0]{e_u^{h,n+1}-\bar{e}_u^{h,n+1}}_{L^{\infty}(F)}
  & \text{(by \cref{eq:euhnL32bound})}
  \\
  &
    \le C h_K^{\dim/6} \norm[0]{ e_u^{h,n} }_{L^2(K)}
    \norm[0]{u^{n+1}}_{W_3^1(K)}
    h_K^{(1-\dim)/2}\norm[0]{e_u^{h,n+1}-\bar{e}_u^{h,n+1}}_{L^2(F)}
  & \text{(by \cref{eq:eminebarinf2F})}
  \\
  &
    \le C \norm[0]{e_u^{h,n}}_{L^2(K)}
    \norm[0]{u^{n+1}}_{W_3^1(K)}
    (h_K^{-1/2}\norm[0]{e_u^{h,n+1}-\bar{e}_u^{h,n+1}}_{L^2(F)}),
  &
\end{align*}
where the last inequality is because
$h_K^{\dim/6}h_K^{(1-\dim)/2} \le h_K^{-1/2}$ for $\dim=2,3$. Since
$F \subset \partial K$ it follows that
\begin{multline*}
  |\langle (\max(u_h^n\cdot n, 0) - \max(\Pi_Vu^n\cdot n, 0))(e_u^{I,n+1}-\bar{e}_u^{I,n+1}), e_u^{h,n+1}-\bar{e}_u^{h,n+1} \rangle_{\partial K}|
  \\
  \le C \norm[0]{ e_u^{h,n} }_{L^2(K)}
  \norm[0]{u^{n+1}}_{W_3^1(K)}
  (h_K^{-1/2}\norm[0]{e_u^{h,n+1}-\bar{e}_u^{h,n+1}}_{L^2(\partial K)}).
\end{multline*}
Summing over all elements in $\mathcal{T}_h^s$ and by the
Cauchy--Schwarz inequality,
\begin{equation}
  \label{eq:T422c}
  \begin{split}
    I_{422c}
    &\le C\sum_{K\in\mathcal{T}_h^s} \norm[0]{ e_u^{h,n} }_{L^2(K)}
    \norm[0]{u^{n+1}}_{W_3^1(K)}
    (h_K^{-1/2}\norm[0]{e_u^{h,n+1}-\bar{e}_u^{h,n+1}}_{L^2(\partial K)})
    \\
    &\le C\max_{K \in \mathcal{T}_h^s} \norm[0]{u^{n+1}}_{W_3^1(K)} \sum_{K\in\mathcal{T}_h^s} \norm[0]{ e_u^{h,n} }_{L^2(K)}
    (h_K^{-1/2}\norm[0]{e_u^{h,n+1}-\bar{e}_u^{h,n+1}}_{L^2(\partial K)}) 
    \\
    &\le C\max_{K \in \mathcal{T}_h^s} \norm[0]{u^{n+1}}_{W_3^1(K)} 
    \del[2]{\sum_{K\in\mathcal{T}_h^s} \norm[0]{ e_u^{h,n} }_{L^2(K)}^2}^{1/2}
    \del[2]{\sum_{K\in\mathcal{T}_h^s}h_K^{-1}\norm[0]{e_u^{h,n+1}-\bar{e}_u^{h,n+1}}_{L^2(\partial K)}^2}^{1/2}
    \\
    &\le C\norm[0]{u^{n+1}}_{W_3^1(\Omega^s)} \norm[0]{ e_u^{h,n} }_{\Omega^s} \tnorm{\boldsymbol{e}_u^{h,n+1}}_{v,s}.   
  \end{split}
\end{equation}
Combining \cref{eq:T422a,eq:T422b,eq:T422c}, and applying Young's
inequality, we find the following bound for $I_{422}$:
\begin{equation}
  \label{eq:T422}
  I_{422}
  \le
  \tfrac{1}{2}\gamma \tnorm{\boldsymbol{e}_u^{h,n+1}}_{v}^2 + \frac{C}{\gamma}\norm[0]{e_u^{h,n}}_{\Omega^s}^2\norm[0]{u^{n+1}}_{W_3^1(\Omega^s)}^2.
\end{equation}
Combining now \cref{eq:boundI421,eq:T422} we find that
\begin{equation*}
  \label{eq:boundT42}
  I_{42}
  \le
  \gamma \tnorm{\boldsymbol{e}_u^{h,n+1}}_{v}^2 + \frac{C}{\gamma}h^{4k}\norm[0]{u^{n+1}}_{k+1,\Omega^s}^2
  \norm[0]{u^{n}}_{k+1,\Omega^s}^2
  + 
  \frac{C}{\gamma}\norm[0]{e_u^{h,n}}_{\Omega^s}^2\norm[0]{u^{n+1}}_{W_3^1(\Omega^s)}^2,
\end{equation*}
which, when combined with \cref{eq:boundI41}, gives us:
\begin{equation*}
  I_4
  \le
  2\gamma \tnorm{\boldsymbol{e}_u^{h,n+1}}_{v}^2
  + \frac{C}{\gamma}h^{2k}\norm[0]{u^{n+1}}_{k+1,\Omega^s}^2\norm[0]{u^{n}}_{k+1,\Omega^s}^2
  + \frac{C}{\gamma}\norm[0]{e_u^{h,n}}_{\Omega^s}^2\norm[0]{u^{n+1}}_{W_3^1(\Omega^s)}^2,
\end{equation*}
which is the desired result.

\end{document}